\documentclass[11pt]{amsart}
\usepackage{hyperref}
\usepackage{amssymb,amsfonts,amsmath,amsopn,amstext,amscd,latexsym}
\usepackage{bbm}
\usepackage[all,cmtip]{xy}

\usepackage{enumerate}
\usepackage[shortlabels]{enumitem}

\theoremstyle{plain}

\newtheorem{theorem}{Theorem}[section]
\newtheorem{thm}[theorem]{Theorem}
\newtheorem{lemma}[theorem]{Lemma}
\newtheorem{lem}[theorem]{Lemma}

\newtheorem{prop}[theorem]{Proposition}

\newtheorem{cor}[theorem]{Corollary}

\theoremstyle{remark}

\newtheorem{remark}[theorem]{Remark}

\DeclareMathOperator{\ch}{ch}

\DeclareMathOperator{\Hom}{Hom}

\DeclareMathOperator{\Aut}{Aut}

\DeclareMathOperator{\Stab}{Stab}
\DeclareMathOperator{\ad}{ad}
\DeclareMathOperator{\diag}{diag}
\DeclareMathOperator{\Irr}{Irr}
\DeclareMathOperator{\IBr}{IBr}

\DeclareMathOperator{\Ext}{Ext}
\DeclareMathOperator{\ext}{Ext}

\DeclareMathOperator{\End}{End}
\DeclareMathOperator{\tr}{tr}

\DeclareMathOperator{\Lie}{Lie}

\DeclareMathOperator{\Ind}{Ind}

\DeclareMathOperator{\GL}{GL}
\DeclareMathOperator{\SL}{SL}
\DeclareMathOperator{\SU}{SU}

\newcommand{\cA}{{\mathcal A}}
\newcommand{\cB}{{\mathcal B}}
\newcommand{\cC}{{\mathcal C}}
\newcommand{\cD}{{\mathcal D}}
\newcommand{\cE}{{\mathcal E}}

\newcommand{\cG}{{\mathcal G}}
\newcommand{\cH}{{\mathcal H}}
\newcommand{\cI}{{\mathcal I}}

\newcommand{\cM}{{\mathcal M}}
\newcommand{\cN}{{\mathcal N}}
\newcommand{\cO}{{\mathcal O}}
\newcommand{\cP}{{\mathcal P}}

\newcommand{\cT}{{\mathcal T}}
\newcommand{\cU}{{\mathcal U}}

\newcommand{\cX}{{\mathcal X}}
\newcommand{\cY}{{\mathcal Y}}

\newcommand{\gtwo}{{\mathrm{G}_2}}

\newcommand{\bbC}{{\mathbb C}}
\newcommand{\bfC}{{\mathbf C}}

\newcommand{\bbE}{{\mathbb E}}
\newcommand{\bbF}{{\mathbb F}}

\newcommand{\bfN}{{\mathbf N}}

\newcommand{\bfO}{{\mathbf O}}

\newcommand{\bbQ}{{\mathbb Q}}

\newcommand{\bfZ}{{\mathbf Z}}
\newcommand{\AAA}{{\sf A}}
\newcommand{\SSS}{{\sf S}}
\newcommand{\bbone}{{\mathbbm{1}}}

\newcommand{\barFp}{{\overline{\bbF}_p}}

\newcommand{\al}{\alpha}
\newcommand{\ve}{\varepsilon}
\newcommand{\la}{\lambda}
\newcommand{\rA}{\mathrm A}
\newcommand{\rB}{\mathrm B}

\newcommand{\bP}{{\mathbb P}}
\newcommand{\p}{\partial}
\newcommand{\sh}{\sharp}

\newcommand{\Ker}{{\mathrm {Ker}}}
\newcommand{\Syl}{{\mathrm {Syl}}}
\newcommand{\Char}{{\mathrm {char}}}
\newcommand{\GP}{{G^+}}
\DeclareMathOperator{\Sp}{Sp}

\DeclareMathOperator{\soc}{soc}
\DeclareMathOperator{\PSL}{PSL}
\DeclareMathOperator{\PSU}{PSU}
\DeclareMathOperator{\PSp}{PSp}
\DeclareMathOperator{\PGL}{PGL}
\DeclareMathOperator{\GU}{GU}

\DeclareMathOperator{\PIM}{\cP}

\DeclareMathOperator{\EE}{End}

\theoremstyle{plain}

\renewcommand{\mod}{\bmod \,}

\newcommand{\fp}{\bbF_p}
\newcommand{\fq}{\bbF_q}

\newcommand{\fpb}{\overline{\bbF}_p}

\newcommand{\ang}[1]{\langle #1 \rangle}

\newcommand{\tw}[1]{{}^#1}
\numberwithin{equation}{section}

\newcommand{\dl}{\mathfrak{d}_p}

\newcommand{\hd}{{\mathrm {head}}}
\newcommand{\Out}{{\mathrm {Out}}}

\begin{document}

\author[R. Guralnick]{Robert Guralnick}
\address{Department of Mathematics, University of Southern California,
Los Angeles, CA 90089-2532, USA}
\email{guralnic@usc.edu}
\author[F. Herzig]{Florian Herzig}
\address{Department of Mathematics, University of Toronto,
40 St. George Street, Room 6290, Toronto, ON M5S 2E4, Canada}
\email{herzig@math.toronto.edu}
\author[P. Tiep]{Pham Huu Tiep}
\address{Department of Mathematics, University of Arizona,
Tucson, AZ 85721-0089, USA}
\email{tiep@math.arizona.edu}

\title{Adequate Groups of Low Degree}
\date{\today}

\thanks{The first author was partially supported by  NSF
  grants DMS-1001962, DMS-1302886 and the Simons Foundation Fellowship 224965.  He
  also thanks the Institute for Advanced Study for its support.
  The second author was partially supported by a Sloan Fellowship and an NSERC grant.
  The third author was partially supported by the NSF grant DMS-1201374 and the Simons 
  Foundation Fellowship 305247.}
\thanks{We thank the referee for careful reading of the paper, and 
Frank L\"ubeck and Klaus Lux for help with several computations.}



\begin{abstract}
The notion of adequate subgroups was introduced by Jack Thorne \cite{T}.
It is a weakening of the notion of big subgroups used in generalizations of
the Taylor-Wiles method for proving the automorphy of certain Galois representations.
Using this idea, Thorne was able to strengthen
many automorphy lifting theorems.
 It was shown in \cite{GHTT} that if the dimension is small compared to
 the characteristic then all absolutely irreducible representations are
 adequate.   Here we extend the result by showing that, in almost all cases, absolutely
 irreducible $kG$-modules in characteristic $p$, whose irreducible 
 $\GP$-summands have dimension less than $p$ (where $\GP$ denotes the 
 subgroup of $G$ generated by all $p$-elements of $G$),
 are adequate.
 \end{abstract}

\maketitle

\tableofcontents

\section{Introduction}

Throughout the paper,
let $k$ be a field of characteristic $p$ and let $V$ be a finite dimensional vector space over $k$.
Let $\rho:G \rightarrow  \GL(V)$ be an  absolutely irreducible representation.
Thorne  \cite{T} called $(G,V)$ {\it adequate}  if the following conditions hold
(we rephrase the conditions slightly by combining two of the properties into one):
\begin{enumerate}
\item  $p$ does not divide $\dim V$;
\item  $\ext^1_G(V,V)=0$; and
\item  $\EE(V)$ is spanned by the elements $\rho(g)$ with $\rho(g)$ semisimple.
\end{enumerate}
We remark that recently Thorne has shown that one can relax condition (i) above
(see \cite[Corollary 7.3]{T2} and \cite[\S1]{GHT}).

If $G$ is a finite group of order prime to $p$, then it is well known that $(G,V)$ is adequate.
In  this case, condition (iii) is often referred to as Burnside's Lemma.  It is a trivial
consequence of the  Artin-Wedderburn Theorem. Also, $(G,V)$ is adequate if $G$ is a connected
algebraic group over $k = \bar{k}$ and $V$ is a rational irreducible $kG$-module,
see \cite[Theorem 1.2]{Gapp}.

The adequacy conditions are a generalization to higher dimension of the conditions
used by Wiles and Taylor in studying the automorphic lifts of certain
$2$-dimensional Galois representations, and they are a weakening of the
previously introduced {\it bigness} condition \cite{CHT}. 
Thorne \cite{T} strengthened the existing automorphy lifting
theorems for $n$-dimensional Galois representations assuming the weaker
adequacy hypotheses. We refer the reader to \cite{T} for more references and
details.

It was shown  in \cite[Theorem 9]{GHTT} that:

\begin{theorem} \label{adequate}  Let $k$ be a field of characteristic $p$ and
$G$ a finite group.  Let $V$ be an absolutely irreducible faithful $kG$-module.
Let $G^+$ denote the subgroup generated by the $p$-elements of $G$.  If
$\dim W \le (p-3)/2$ for an absolutely irreducible $kG^+$-submodule $W$ of $V$,
then   $(G,V)$ is adequate.
\end{theorem}

The example $G=\SL_2(p)$ with $V$ irreducible of dimension $(p-1)/2$ shows
that the previous theorem is best possible.  However, the counterexamples are
rare.
Our first goal is to prove a similar theorem under
the assumption that $\dim W < p$.  We show that almost always
$(G,V)$ is adequate, see Corollary \ref{cor: adequate}.
Indeed, we show that the spanning condition always holds under the weaker hypothesis.
We show that there are only a handful
of examples where $\ext_G^1(V,V) \ne 0$.  See Theorems
\ref{thm: weak adequacy} and \ref{thm: ext=0} for more precise statements.

Theorem~\ref{adequate} was crucial in several recent applications of
automorphy lifting theorems, such as \cite{BLGGT}, \cite{C}, \cite{D}.
In fact, the main two technical hypotheses in the most recent automorphy lifting theorems
are potential diagonalizability (a condition in $p$-adic Hodge theory) and adequacy of
the residual image \cite{DG}.
Since some important applications of automorphy lifting theorems \cite{BCDT}, \cite{KW}, \cite{D}
require working with primes $p$ that are small relative to the dimension of the
Galois representation, we expect that our results will be useful in obtaining
further arithmetic applications of automorphy lifting theorems. (Note that 
adequacy of $2$-dimensional linear groups has been analyzed in Appendix A of
\cite{gee}.)

\smallskip
An outgrowth of our results leads us to prove
an analogue of the first author's result \cite{Gcr} and answer a question of Serre on
complete reducibility of finite subgroups of orthogonal and symplectic groups of small degree.
This is done in the sequel \cite{GHT}, where we essentially classify indecomposable modules
in characteristic $p$ of dimension less than $2p-2$. We also extend adequacy results
to the case of linear groups of degree $p$ and generalize the asymptotic result
\cite[Theorem 1.2]{Gapp} to disconnected algebraic groups $\cG$
(with $p\nmid [\cG:\cG^0]$) allowing at the same time that $p$ divides the dimension of the
$\cG$-module.

\smallskip
Note that if the kernel of $\rho$ has order prime to $p$, then there is no harm
in passing to the quotient.  So we will generally assume that either
$\rho$ is faithful or more generally has kernel  of order prime to $p$.
Also, note that the dimensions of cohomology groups and the dimension
of the span of the semisimple elements in $G$ in $\EE(V)$ do not
change under extension of scalars.  Hence, most of the time we will work over
an algebraically closed field~$k$.

Following \cite{G2},  we say that the representation $\rho:G \to \GL(V)$,
respectively the pair $(G,V)$, is {\it weakly adequate} if $\EE(V)$ is
spanned by the elements $\rho(g)$ with $\rho(g)$ semisimple.

\smallskip
Our main results are the following:

\begin{thm} \label{thm: weak adequacy}
Let $k$ be a field of characteristic $p$ and
$G$ a finite group.  Let $V$ be an absolutely irreducible faithful $kG$-module.
Let $G^+$ denote the subgroup generated by the $p$-elements of $G$.  If
$p  > \dim W$ for an irreducible $kG^+$-submodule $W$ of $V$,
then  $(G,V)$ is weakly adequate.
\end{thm}

\begin{thm}  \label{thm: ext=0}  Let $k = \overline{k}$ be a field of characteristic $p$ and
$G$ a finite group.  Let $V$ be an irreducible faithful $kG$-module.
Let $G^+$ denote the subgroup generated by the $p$-elements of $G$.  Suppose that
$p  > d:= \dim W$ for an irreducible $kG^+$-submodule $W$ of $V$, and let
$H < \GL(W)$ be induced by the action of $\GP$ on $W$. Then one of the following
holds.

\smallskip
{\rm (a)} $p$ is a Fermat prime, $d = p-1$, $\GP$ is solvable
(and so $G$ is $p$-solvable), and $H/\bfO_{p'}(H) = C_p$.

\smallskip
{\rm (b)} $H^1(G,k) = 0$. Furthermore, either
$\ext^1_G(V,V) = 0$, or one of the following holds.
\begin{enumerate}[\rm(i)]
\item  $H = \PSL_2(p)$ or $\SL_2(p)$, and $d = (p \pm 1)/2$.
\item  $H = \SL_2(p) \times \SL_2(p^a)$ (modulo a central subgroup), $d=p-1$
and $W$ is a tensor product of a $(p-1)/2$-dimensional $\SL_2(p)$-module and a
$2$-dimensional $\SL_2(p^a)$-module.
\item  $p = (q+1)/2$, $d = p-1$, and $H \cong \SL_2(q)$.
\item $p = 2^f +1$ is a Fermat prime, $d = p-2$, and $H \cong \SL_2(2^f)$.
\item $(H,p,d) = (3\AAA_6,5,3)$ and $(2\AAA_7,7,4)$.
\item $(H,p,d) = (\SL_2(3^a),3,2)$ and $a \geq 2$.
\end{enumerate}
\end{thm}

Theorems \ref{thm: weak adequacy} and \ref{thm: ext=0} immediately imply:

\begin{cor}  \label{cor: adequate}  Let $k$ be a field of characteristic $p$ and
$G$ a finite group.  Let $V$ be an absolutely irreducible faithful $kG$-module, and
let $G^+$ denote the subgroup generated by the $p$-elements of $G$.  Suppose that
the dimension of any irreducible $kG^+$-submodule in $V$ is less than $p$.
Let $W$ be an irreducible $\overline{k}\GP$-submodule of $V \otimes_k \overline{k}$. Then
$(G,V)$ is adequate, unless the group
$H < \GL(W)$ induced by the action of $\GP$ on $W$ is as described in one of the
exceptional cases {\rm (a), (b)(i)--(vi)} listed in Theorem \ref{thm: ext=0}.
\end{cor}

\begin{cor}  \label{cor: adequate2}  Let $k$ be a field of characteristic $p$ and
$G$ a finite group.  Let $V$ be an absolutely irreducible faithful $kG$-module, and
let $G^+$ denote the subgroup generated by the $p$-elements of $G$.  Suppose that
the dimension $d$ of any irreducible $kG^+$-submodule in $V$ is less than $p-3$.
Let $W$ be an irreducible $\overline{k}\GP$-submodule of $V \otimes_k \overline{k}$. Then
$(G,V)$ is adequate, unless $d = (p \pm 1)/2$ and the group
$\bar{H} < \PGL(W)$ induced by the action of $\GP$ on $W$ is $\PSL_2(p)$.
\end{cor}

One should emphasize that, in all the aforementioned results, the dimension bound 
$\dim W < p$ is imposed only on an irreducible $\GP$-summand of $V$. In 
general, $G/\GP$ can be an arbitrary $p'$-group, and likewise, 
$(\dim V)/(\dim W)$ can be arbitrarily large. So a major bulk of the proofs, 
especially for Theorem \ref{thm: weak adequacy}, is spent to establish the 
results under these more general hypotheses.

\medskip
This paper is organized as follows. In \S2, based on results of \cite{BZ}, we
describe the structure of (non-$p$-solvable) finite linear groups $G < GL(V)$
such that the dimension of irreducible $\GP$-summands in $V$ is less than $p$,
cf.\ Theorem \ref{str}. \S\S3 and 4 are devoted to establish weak adequacy
for Chevalley groups in characteristic $p$. In the next two sections 5 and 6,
we prove adequacy for the remaining families of finite groups occurring in
Theorem \ref{str}, and complete the proof of Theorem \ref{thm: weak adequacy}.
In \S7 we collect various facts concerning extensions and self-extensions of simple modules. The main result of \S8, Proposition \ref{prop1},
classifies self-dual indecomposable $\SL_2(q)$-modules for $p|q$. In \S9, we
describe the structure of finite groups $G$ possessing a reducible
indecomposable module of dimension $\leq 2p-3$ (cf.\ Proposition \ref{indec-str}).
Theorem \ref{thm: ext=0} and Corollary \ref{cor: adequate} are proved in the final section 10.

\medskip
{\bf Notation.} If $V$ is a $kG$-module and $X \leq G$ is a subgroup, then
$V_X$ denotes the restriction of $V$ to $X$. The containments $X \subset Y$ (for
sets) and $X < Y$ (for groups) are strict. Fix a prime $p$ and an
algebraically closed field $k$ of characteristic $p$. Then for any finite group $G$,
$\IBr_p(G)$ is the set of isomorphism classes of irreducible
$kG$-representations (or their Brauer characters, depending on the context), $\dl(G)$ denotes the smallest degree of nontrivial $\varphi \in \IBr_p(G)$,
and $B_0(G)$ denotes the principal $p$-block of $G$. Sometimes we use
$\bbone$ to denote the principal representation.
$\bfO_p(G)$ is the largest normal $p$-subgroup of $G$, $\bfO^p(G)$ is the
smallest normal subgroup $N$ of $G$ subject to $G/N$ being a $p$-group, and similarly
for $\bfO_{p'}(G)$ and $\bfO^{p'}(G) = \GP$. Furthermore, the {\it Fitting subgroup} $F(G)$ is
the largest nilpotent normal subgroup of $G$, and $E(G)$ is the product of all
subnormal quasisimple subgroups of $G$, so that $F^*(G) = F(G)E(G)$ is the {\it
generalized Fitting subgroup} of $G$.
Given a finite-dimensional $kG$-representation
$\Phi:G \to \GL(V)$, we denote by $\cM$ the $k$-span
$$\langle \Phi(g) \mid \Phi(g) \mbox{ semisimple}\rangle_k.$$
If $M$ is a finite length module over a ring $R$,  then define $\soc_i(M)$
by $\soc_0(M) = 0$ and $\soc_j(M)/\soc_{j-1}(M) = \soc (M/\soc_{j-1}(M))$.
If $M = \soc_j(M)$
with $j$ minimal, we say that $j$ is the {\it socle length} of $M$.

\section{Linear groups of low degree}

First we describe the structure of absolutely irreducible non-$p$-solvable linear groups
of low degree, relying on the main result of Blau and Zhang \cite{BZ}:

\begin{thm}\label{bz}
Let $W$ be a faithful, absolutely irreducible $kH$-module for a finite group $H$ with
$\bfO ^{p'}(H) = H$. Suppose that $1 < \dim W < p$. Then one of the following cases holds,
where $P \in \Syl_p(H)$.

\smallskip
{\rm (a)} $p$ is a Fermat prime, $|P| = p$, $H = \bfO _{p'}(H)P$ is solvable, $\dim W = p-1$,
and $\bfO_{p'}(H)$ is absolutely irreducible on $W$.

\smallskip
{\rm (b)} $|P| = p$, $\dim W = p-1$, and one of the following conditions holds:

\hskip1pc
{\rm (b1)} $(H,p) = (\SU_n(q), (q^n+1)/(q+1))$, $(\Sp_{2n}(q),(q^n+1)/2)$,
$(2\AAA_7,5)$, $(3J_3,19)$, or $(2Ru,29)$.

\hskip1pc
{\rm (b2)} $p = 7$ and $H = 6_1\cdot \PSL_3(4)$, $6_1\cdot \PSU_4(3)$, $2J_2$,
$3\AAA_7$, or $6\AAA_7$.

\hskip1pc
{\rm (b3)} $p = 11$ and $H = M_{11}$, $2M_{12}$, or $2M_{22}$.

\hskip1pc
{\rm (b4)} $p = 13$ and $H = 6\cdot Suz$ or $2\gtwo(4)$.

\smallskip
{\rm (c)} $|P| = p$, $\dim W = p-2$, and
$(H,p) = (\PSL_n(q), (q^n-1)/(q-1))$, $(\AAA_p,p)$, $(3\AAA_6,5)$,
$(3\AAA_7,5)$, $(M_{11},11)$, or $(M_{23},23)$.

\smallskip
{\rm (d)} $(H,p,\dim W) = (2\AAA_7,7,4)$, $(J_1,11,7)$.

\smallskip
{\rm (e)} Extraspecial case: $|P| = p = 2^n+1 \geq 5$, $\dim W = 2^n$,
$\bfO _{p'}(H) = R\bfZ(H)$, $R = [P,R]\bfZ (R) \in \Syl_2(\bfO_{p'}(H))$, $[P,R]$ is an
extraspecial $2$-group of order $2^{1+2n}$, $V_{[P,R]}$ is absolutely irreducible.
Furthermore,  $S := H/\bfO _{p'}(H)$ is simple non-abelian, and either
$S= \Sp_{2a}(2^b)'$ or $\Omega^-_{2a}(2^b)'$ with
$ab = n$, or $S = \PSL_2(17)$ and $p = 17$.

\smallskip
{\rm (f)} $Lie(p)$-case: $H/\bfZ (H)$ is a direct product of simple groups of Lie type in characteristic $p$.\\
Furthermore, in the cases {\rm (b)--(d)}, $H$ is quasisimple with $\bfZ(H)$ a $p'$-group.
\end{thm}

\begin{proof}
We apply Theorem A of \cite{BZ} and arrive at one of the possibilities (a)--(j) listed therein.
Note that possibility (j) is restated as our case (f), and possibilities (f)--(i) do not occur since
$H$ is absolutely irreducible. Possibility (a) does not arise either since $\dim W > 1$, and
possibility (b) is restated as our case (a). Next, in the case of possibility (c),
either we are back to our case (a), or else we are in case (e), where the simplicity of
$S$ follows from the assumption that $H = \bfO ^{p'}(H)$. (Also,
$S \not\cong \Omega^+_{2a}(2^b)$ since $|S|_p = |P| = p$.)

In the remaining cases (d), (e), and (g) of \cite[Theorem A]{BZ}, we have that
$H/\bfZ (H) = S$ is a simple non-abelian group, and $\bfZ(H)$ is a cyclic $p'$-group
by Schur's Lemma. Hence, $H^{(\infty)}$ is a perfect normal subgroup of $p'$-index
in $H = \bfO ^{p'}(H)$. It follows that $H = H^{(\infty)}$ and so it is quasisimple. Also,
the possibilities for $(S,\dim W,p)$ are listed. Using \cite{GT1} (if $S = \PSL_n(q)$), \cite{GMST} (if $S = \PSU_n(q)$ or $\PSp_{2n}(q)$), \cite[Lemma 6.1]{GT2} (if $S = \AAA_p$
and $p \geq 17$), and \cite{JLPW} (for the other simple groups), we arrive at cases (b)--(d).
\end{proof}

Next we prove some technical lemmas in the spirit of \cite[Lemma 3.10]{BZ}.

\begin{lem}\label{dec1}
Let $G$ be a finite group with normal subgroups $K_1$ and $K_2$ such that
$K_1 \cap K_2 \leq \bfO_{p'}(G)$. For any finite group $X$, let $\overline{X}$ denote
$X/\bfO_{p'}(X)$.
Suppose that  $\overline{G/K_1} \cong \prod_{i \in I}M_i$ and
$\overline{G/K_2} \cong \prod_{j \in J}N_j$ are direct products of simple non-abelian groups.
Then there are some sets $I' \subseteq I$ and
$J' \subseteq J$ such that
$$\overline{G} \cong \prod_{i \in I'}M_i \times \prod_{j \in J'}N_j.$$
\end{lem}

\begin{proof}
For $i = 1,2$, let $K_i \leq H_i \lhd G$ be such that $H_i/K_i = \bfO_{p'}(G/K_i)$. Then
$$G/H_1 \cong  \prod_{i \in I}M_i,~~G/H_2 \cong  \prod_{j \in J}N_j.$$
By \cite[Lemma 3.9]{BZ}, there are some sets $I' \subseteq I$ and
$J' \subseteq J$ such that
$$G/(H_1 \cap H_2) \cong \prod_{i \in I'}M_i \times \prod_{j \in J'}N_j.$$
It remains to show that $H_1 \cap H_2 = \bfO_{p'}(G)$. Certainly,
$H_1 \cap H_2 \geq \bfO_{p'}(G)$. Conversely,
$$(H_1 \cap K_2)/(K_1 \cap K_2) \hookrightarrow  H_1/K_1,~~
   (H_1 \cap H_2)/(H_1 \cap K_2) \hookrightarrow  H_2/K_2,$$
and $K_1 \cap K_2 \leq \bfO_{p'}(G)$. It follows that $H_1 \cap H_2$ is a $p'$-group.
\end{proof}

\begin{lem}\label{dec2}
Let $G$ be a finite group with a faithful $kG$-module $V$. Suppose that
$V = W_1 \oplus \ldots \oplus W_t$ is a direct sum of $kG$-submodules, and let
$H_i \leq \GL(W_i)$ be the linear group induced
by the action of $G$ on $W_i$. Suppose that for each $i$,
$S_i := H_i/\bfO_{p'}(H_i)$ is a simple non-abelian group. Then there is a subset
$J \subseteq \{1,2, \ldots ,t\}$ such that
$$G/\bfO_{p'}(G) \cong \prod_{j \in J}S_j.$$
In particular, if $\bfO_{p'}(H_i) = 1$ for all $i$, then
$G \cong \prod_{j \in J}S_j$.
\end{lem}

\begin{proof}
We proceed by induction on $t$. The induction base $t=1$ is obvious. For the induction
step, let $K_i$ denote the kernel of the action of $G$ on $W_i$, so that
$H_i = G/K_i$. The faithfulness of $V$ implies that $\cap^t_{i=1}K_i = 1$.
Adopt the bar notation $\overline{X}$ of Lemma \ref{dec1}. By the assumption,
$\overline{G/K_1} \cong S_1$. Next, observe that $L := \cap^t_{i=2}K_i$ is the kernel
of the action of $G$ on $V' := W_2 \oplus \ldots \oplus W_t$, and
the action of $G/L$ on $W_i$ induces $H_i$ for all $i \geq 2$. Applying the induction hypothesis to $G/L$ acting on $V'$, we see that
$\overline{G/L} \cong \prod_{j \in J'}S_j$ for some $J' \subseteq \{2,3, \ldots ,t\}$.
Also, $K_1 \cap L = 1$. Hence we can apply Lemma \ref{dec1} to get
$\overline{G} \cong \prod_{j \in J}S_j$ for some $J \subseteq \{1,2,3, \ldots ,t\}$.

Finally, if $\bfO_{p'}(H_i) = 1$ for all $i$, then the action of $\bfO_{p'}(G)$ on $W_i$ induces a normal
$p'$-subgroup of $H_i$ for all $i$, whence $\bfO_{p'}(G) \leq \cap^t_{i=1}K_i = 1$,
and we are done.
\end{proof}

\begin{thm}\label{str}
Let $V$ be a finite dimensional vector space over an algebraically closed field $k$
of characteristic $p$ and $G < \GL(V)$ a finite irreducible subgroup. Suppose that an irreducible $\GP$-submodule $W$ of $V$ has dimension $<p$ and $\GP$ is not
solvable. Then $\GP$ is perfect and has no composition factor isomorphic to $C_p$;
in particular, $H^1(G,k) = 0$. Furthermore, if $H$ is the image of $\GP$ in $\GL(W)$, then one of the following statements holds.

\smallskip
{\rm (i)} One of the cases (b)--(d) of Theorem \ref{bz} holds for $H$, and
$\GP/\bfZ(\GP) = S_1 \times \ldots \times S_n \cong S^n$ is a direct product  of $n$ copies of the simple non-abelian group
$S  = H/\bfZ(H)$. Here,
$G$ permutes these $n$ direct factors $S_1, \ldots ,S_n$ transitively.
Furthermore, $\GP = L_1 * \ldots * L_n$ is a central product of quasisimple groups
$L_i$, each being a central cover of $S$, and the action of $\GP$ on each
irreducible $\GP$-submodule $W_i$ of $W$ induces a quasisimple subgroup of
$\GL(W_i)$. Finally, if $H$ is the full covering group of $S$ or if $H = S$, then
$$\GP = L_1 \times L_2 \times \ldots \times L_n \cong H^n.$$

\smallskip
{\rm (ii)} The case (e) of Theorem \ref{bz} holds for $H$. Furthermore, $\bfO _{p'}(\GP)$ is
irreducible on any irreducible $\GP$-submodule $W_i$ of $V$, and
$\GP/\bfO _{p'}(\GP) \cong S^m$ is a direct product of $m \geq 1$ copies of the simple
non-abelian group $S$ listed in Theorem \ref{bz}(e).

\smallskip
{\rm (iii)} The case (f) of Theorem \ref{bz} holds for $H$, and
$\GP = L_1 * \ldots * L_n$ is a central product of quasisimple groups $L_i$ of Lie type in characteristic $p$ with $\bfZ(L_i)$ a $p'$-group.
\end{thm}

\begin{proof}
(a) By Clifford's theorem, $V_{\GP} \cong e\sum^t_{i=1}W_i$ for some $e,t \geq 1$, and
$\{W_1, \ldots ,W_t\}$ is a full set of representatives of isomorphism classes of
$G$-conjugates of $W \cong W_1$. Let $\Phi_i~:~\GP \to \GL(W_i)$ denote the corresponding representation, and let $K_i := \Ker(\Phi_i)$, so that $\GP/K_i \cong H$ for all $i$, where
we denote by $H$ the subgroup of $\GL(W)$ induced by the action of $\GP$ on $W$.
The faithfulness of the action of $G$ on $V$ implies
that $\cap^t_{i=1}K_i = 1$. In particular, $\GP$ injects into
$\prod^t_{i=1}(\GP/K_i) \cong H^t$. Hence case (a) of Theorem \ref{bz} is
impossible since $\GP$ is not solvable. In the case (f) of Theorem \ref{bz},
an argument similar to the proof of Lemma \ref{dec2} shows that
$\GP/\bfZ(\GP) = S_1 \times \ldots \times S_n$ is a direct product of simple groups
$S_i$ of Lie type in characteristic $p$.
Since $\GP = \bfO^{p'}(\GP)$ and $\bfO_{p}(\GP) \leq \bfO_p(G) =1$, it then follows
that $\GP = L_1 * \ldots * L_n$, a central product of quasisimple groups $L_i$ of Lie type in characteristic $p$ with $\bfZ(L_i)$ a $p'$-group (just take $L_i$ to be a perfect
inverse image of $S_i$ in $\GP$), i.e.\ (iii) holds.
In the remaining cases (b)--(e) of Theorem \ref{bz},
$H/\bfO_{p'}(H) \cong S$, where
$S$ is a non-abelian simple group described in Theorem \ref{bz}(b)--(e). By Lemma \ref{dec2}, $\GP/\bfO_{p'}(\GP) \cong S^n$, a direct product of $n \geq 1$ copies of $S$.
Thus in all cases, $\GP$ has no composition factor $\cong C_p$ and
$\bfZ(\GP) \leq \bfO_{p'}(\GP)$. Furthermore, $\GP = (\GP)^{(\infty)}\bfO_{p'}(\GP)$ and so
$(\GP)^{(\infty)}$ is a normal subgroup of $p'$-index in $\GP = \bfO^{p'}(\GP)$, whence
$\GP$ is perfect.
Thus the first claim of Theorem \ref{str} holds in all cases.

\smallskip
(b) Suppose next that we are in the cases (b)--(d) of Theorem \ref{bz}. Then $H$ is quasisimple
and $\bfZ(H)$ is a $p'$-group; in particular, $\bfO_{p'}(H) = \bfZ(H)$ and
$H/\bfZ(H) = S$. Note that $\Phi_i(\bfO_{p'}(\GP))$ is a normal $p'$-subgroup of
$H_i = \Phi_i(\GP) \cong H$, whence $\Phi_i(\bfO_{p'}(\GP)) \leq \bfZ(H_i)$. Thus,
for any $z \in \bfO_{p'}(\GP)$ and any $g \in \GP$, $[z,g]$ acts trivially on each $W_i$
and so $[z,g] \in \cap^t_{i=1}K_i = 1$, i.e.\ $z \in \bfZ(\GP)$. We have shown that
$\bfO_{p'}(\GP) = \bfZ(\GP)= : Z$.

Now we can write $\GP/Z = S_1 \times \ldots \times S_n$ with
$S_i \cong S$. Let $M_i$ denote the full inverse image of $S_i$ in $\GP$ and let
$L_i := M_i^{(\infty)}$.  Then $M_i = L_iZ$, $L_i/(L_i \cap Z) \cong M_i/Z \cong S$,
and so $L_i$ is quasisimple and a central cover of $S$. Next, for $i \neq j$ we have
$[L_i,L_j] \leq Z$ and so, since $L_i$ is perfect,
$$[L_i,L_j] = [[L_i,L_i],L_j] = 1$$
by Three Subgroups Lemma. It follows that $M := L_1L_2 \ldots L_n$ is a central
product of $L_i$. But $\GP = MZ$ and $\GP$ is perfect, so $\GP = M$.

The remaining claims in (i) are obvious if $t = 1$, so we will now assume $t > 1$. First
we show that $G$ acts transitively on $\{S_1, \ldots ,S_n\}$. Relabeling the $W_i$
suitably we may assume that $K_1Z/Z \geq \prod_{i \neq 1}S_i$ and
$K_2Z/Z \geq \prod_{i \neq 2}S_i$. But $\GP/K_j = \Phi_j(\GP)$ is quasisimple, so
in fact $K_jZ/Z = \prod_{i \neq j}S_i$ for $j = 1,2$. By Clifford's theorem,
$W_2 = W_1^g$ for some $g \in G$. Now $g$ sends $K_1$ to $K_2$, and so
it sends $S_1$ to $S_2$, as desired. If furthermore $H = S$, then $\bfO_{p'}(H) = 1$,
whence $\GP = S_1 \times \ldots \times S_n \cong H^n$ by Lemma \ref{dec2}.
Consider the opposite situation: $H$ is the full covering group of $S$. Again relabeling
the $W_i$ suitably and arguing as above, we may assume that
$K_1Z/Z = \prod_{i \neq 1}S_i$. In this case, $K_1Z \geq L_i$ for $i \geq 2$,
whence $L_i = [L_i,L_i] \leq [K_1Z,K_1Z] \leq K_1$ and $K_1 \geq L_2L_3 \ldots L_n$.
It also follows that  $\GP = K_1L_1$ and so
$L_1/(K_1 \cap L_1) \cong \GP/K_1 \cong H$. Recall that $L_1$ is perfect and
$L_1/(L_1 \cap Z) \cong S$, i.e.\ $L_1$ is a central extension of the simple group
$S$. But $H$ is the full covering group of $S$, so $|L_1| \leq |H|$.
It follows that $L_1 \cap K_1 = 1$ and $L_1 \cong H$;
in particular, $L_1 \cap \prod_{j \neq 1}L_j = 1$. Similarly,
$L_i \cong H$ and $L_i \cap \prod_{j \neq i}L_j  = 1$ for all $i$. Thus
$\GP = L_1 \times \ldots \times L_n \cong H^n$.

\smallskip
(c) Assume now that we are in case (e) of Theorem \ref{bz}. Then
$P_i := \Phi_i(\bfO_{p'}(\GP))$ is again a normal $p'$-subgroup of $H_i$, and so
$P_i \leq \bfO_{p'}(H_i)$. On the other hand,
$H_i/P_i$ is a quotient of $\GP/\bfO_{p'}(\GP) \cong S^n$, whence all composition factors
of $H_i/P_i$ are isomorphic to $S$. Since $H_i/\bfO_{p'}(H_i) \cong S$, we conclude
that $P_i = \bfO_{p'}(H_i)$; in particular, $\bfO_{p'}(\GP)$ is irreducible on $W_i$.
\end{proof}

\section{Weak adequacy for $\SL_2 (\fp)$}   \label{sec:slp}

\begin{prop}
\label{pr1a}
Any non-trivial irreducible representation $V$ of $\SL_2 (\fp)$ over $\fpb$
is weakly adequate except when $\dim V = p$ and $p \le 3$.
\end{prop}

\begin{remark}
  When $p \le 3$ the $p'$-elements of $\SL_2 (\fp)$ generate a normal subgroup of index $p$. If moreover $\dim V = p$ then
  this subgroup does not act irreducibly, hence $V$ cannot be weakly adequate.
\end{remark}

The rest of the section is devoted to proving Proposition \ref{pr1a}. 
Note that $p > 2$. In the following we write $V=L(a)$ with $0 < a \le p-1$.
If $a \le \frac{p-3}{2}$ then the argument of \cite[Theorem 9]{GHTT}  applies.
(Let $\mathcal T \subset \SL_2$ denote the diagonal maximal torus. Then
distinct weights of $\mathcal T_{/\fpb}$ on $L(a)$ restrict distinctly to $\mathcal T(\fp)$, and
$\End V$ is semisimple by \cite{Serre1} with $p$-restricted highest weights.)
We will assume from now on that $a \ge \frac{p-1}{2}$.

\begin{lem}
\label{lm2a}
Suppose that $\frac{p-1}{2} \le a \le p-1$.
Then
\[
\hd_{\SL_2} \big( L(a) \otimes L(a)\big) \cong \bigoplus_{i=0}^{(p-1)/2}
L(2i) .
\]
Moreover, if $a\ne p-1$, $\hd_{\SL_2(\fp)} \big( L(a) \otimes L(a)\big) =
\hd_{\SL_2} \big( L(a)\otimes L(a)\big)$, whereas if $a=p-1$,
\[
\hd_{\SL_2(\fp)} \big( L(a) \otimes L(a)\big)  \cong
\bigoplus_{i=0}^{(p-1)/2}
L(2i)
\oplus L(p-1) .
\]
\end{lem}

\begin{proof}  By \cite[Lemmas 1.1, 1.3]{DH},
we see that  for $\SL_2$,
\begin{equation}
L(a)\otimes L(a) \cong \bigoplus_{i=0}^{p-2-a} L(2i) \oplus
\bigoplus_{i=p-1-a}^{(p-3)/2} T(2p - 2 - 2i) \oplus L(p-1) ,\label{eq:1}
\end{equation}
where the tilting module $T(2p-2-r)$ for $0\le r \le p-2$ is uniserial of the form $(L(r)|L(2p-2-r)|L(r))$.
This proves the first part of the lemma.
As is pointed out in Lemma 1.1 of \cite{DH}, $T(2p-2-r)\cong Q(r)$
for $0 \le r \le p-2$, which implies that $T(2p-2-r)|_{\SL_2(\fp)}$ is
projective.  See also \cite[\S2.7]{jantzen}.

Noting that $L(2p-2-r)|_{\SL_2(\fp)}\cong L(p-1-r) \oplus L(p-3-r)$
and using that $L(p-1)$ is the only irreducible projective $\SL_2 (\fp)$-module,
it follows that
\begin{equation}\label{eq:2}
T(2p-2-r)|_{\SL_2 (\fp)} \cong
\begin{cases}
U(r)&\text{if $0 < r \le p-2$}\\
U(0)\oplus L(p-1)&\text{if $r=0$,}
\end{cases}
\end{equation}
where $U(i)$ denotes the projective cover of $L(i)$.
The claim follows.
\end{proof}

In the following, we will think of $V\cong L(a)$ as the space of
homogeneous polynomials in $X,Y$ of degree $a$.

\begin{lem}
\label{lm3a}
$(\End V)^{\mathcal U} \cong \bigoplus\limits_{k=0}^a \fpb \cdot \big( X \frac{\partial}{\partial Y}
\big)^k$, where ${\mathcal U}=\left(\begin{smallmatrix} 1&*\\ &1\end{smallmatrix}\right)\subset \SL_2$.
\end{lem}

\begin{proof}
The torus
$\mathcal T = \left( \begin{smallmatrix}*\\ &*\end{smallmatrix}\right) \subset\SL_2$ acts on $(\End V)^{\mathcal U}$ and
for $\la \in X(\mathcal T)$,
\begin{equation}
\Hom_\mathcal T \big( \la , (\End V)^{\mathcal U}\big) \cong \Hom_{\SL_2}\big( V(\la) , \End V\big) .\label{eq:3}
\end{equation}
So it follows from \eqref{eq:1} that $\dim (\End V)^{\mathcal U} = a+1$.
(Namely, $\la = 0,2,\ldots, 2a$ each work once.)
A simple calculation shows that $X \frac{\partial}{\partial Y}$ is ${\mathcal U}$-invariant,
hence so are $\big( X \frac{\partial}{\partial Y}\big)^k$, $0 \le k \le a$, which are
clearly non-zero.
Since $\big( X\frac{\partial}{\partial Y}\big)^k$ has weight $2k$, they are
linearly independent.
\end{proof}

By Lemma~\ref{lm3a} and (\ref{eq:1}), for $0 \le k \le a$, the $\SL_2$-representation generated by $\big( X \frac{\partial}{\partial Y}\big)^k$ 
is $V(2k) \subset \End (V)$.

\begin{lem}
\label{lm4a}
The weight $0$ subspace in $V(2k)\subset \End V$ is the line spanned by
\[
\Delta_k := \sum_{i=0}^k (-1)^{k-i} \binom{k}{i}^2 X^i Y^{k-i}
\Big( \frac{\p}{\p X}\Big)^i \Big(\frac{\p}{\p Y}\Big)^{k-i} \quad (0\le k \le a) .
\]
\end{lem}

\begin{proof}
We compute the weight $0$ part of $\big(\begin{smallmatrix}1 \\ -1&1\end{smallmatrix}\big) \cdot
\big( X \frac{\p}{\p Y}\big)^k$.
Take $f\in\fpb [X,Y]$ homogeneous of degree $a$.
Under $\big(\begin{smallmatrix}1 \\ -1&1\end{smallmatrix}\big) \cdot
\big( X \frac{\p}{\p Y}\big)^k$ the element $f$ is sent to
\begin{align*}
&\left( \begin{pmatrix}1 \\ -1 & 1\end{pmatrix} \cdot \Big( X \frac{\p}{\p Y}\Big)^k\right)
f(X+Y,Y) \\
&= \begin{pmatrix}1 \\ -1 & 1\end{pmatrix} \left[ X^k \sum_{i=0}^k \binom{k}{i}
\left( \Big( \frac{\p}{\p X}\Big)^i \Big( \frac{\p}{\p Y}\Big)^{k-i} f\right) (X+Y,Y)\right]\\
&= (X-Y)^k \sum_{i=0}^k \binom{k}{i} \Big( \frac{\p}{\p X}\Big)^i
\Big( \frac{\p}{\p Y}\Big)^{k-i} f .
\end{align*}
The weight $0$ part is the part that does not change the monomial degree, so
it is $\Delta_k$.
Finally, note that $\Delta_k\ne 0$ as $\Delta_k (X^a) \ne 0$.
\end{proof}

Now suppose that $0\le k \le \frac{p-1}{2}$.
By the $\SL_2$-invariant trace pairing on $\End V$ the element $\Delta_k\in
\soc_{\SL_2} (\End V)$ induces an element $\delta_k \in \big( \hd_{\SL_2}
(\End V)\big)^*$ that is zero on all irreducible constituents of $\hd_{\SL_2} (\End V)$
except for $L(2k)$.
Let $\pi_\ell \in \End V$ $(0 \le \ell \le a)$ denote the projection
$X^i Y^{a-i} \ \mapsto \delta_{i\ell} X^i Y^{a-i}$.

\begin{lem}
\label{5a}
If $0\le k \le \frac{p-1}{2}$, then
$\delta_k (\pi_\ell)$ is a polynomial in $\ell$ of degree exactly $k$.
\end{lem}

\begin{proof}
Note that $\delta_k (\pi_\ell) = \tr (\pi_\ell \circ \Delta_k)$ is the eigenvalue
of $\Delta_k $ on $X^\ell Y^{a-\ell}$, hence equals
\[
\sum_{i=0}^k (-1)^{k-i} \binom{k}{i}^2 \ell (\ell-1) \cdots (\ell-i+1) (a-\ell)
(a-\ell - 1) \cdots (a-\ell -k +i +1) .
\]
This is a polynomial in $\ell$ of degree at most $k$.
The coefficient of $\ell^k$ is $\sum_{i=0}^k \binom{k}{i}^2 = \binom{2k}{k} \not\equiv 0\ ({\rm mod}\ p)$,
as $k < \frac p2$.
\end{proof}

Let us denote this polynomial by $p_k (z) \in \fp [z]$.

\begin{proof}[Proof of Proposition \ref{pr1a}]
Recall that $\frac{p-1}{2} \le a \le p-1$.
Let $\cM$ denote the span of the image of the $p'$-elements in $\End V$ and let $M$ denote the image of $\cM$
in $\hd_{\SL_2 (\fp)} (\End V)$. Since $\cM$ is $\SL_2 (\fp)$-stable, it suffices to show that $M = \hd_{\SL_2 (\fp)} (\End V)$.


(a) Suppose that $a < p-1$. By Lemma \ref{lm2a}, $\hd_{\SL_2 (\fp)} (\End V) \cong \bigoplus\limits_{i=0}^{(p-1)/2}
L(2i)$. Suppose that for some $0\le k \le \frac{p-1}{2}$, $M$ does not contain $L(2k)$. Then $\delta_k$ annihilates the image
of all $p'$-elements.
The images of the diagonal elements of $\SL_2(\fp)$ in $\End(V)$ span the subspace with basis
\[
\pi_i \left( a - \textstyle{ \frac{p-3}{2}} \le i \le \textstyle{\frac{p-3}{2}}\right) , \quad
\pi_i + \pi_{i + \frac{p-1}{2}}\ \left(0 \le i \le a - \textstyle{\frac{p-1}{2}}\right) .
\]
Hence
\begin{equation}
\label{eq1}
\begin{aligned}
p_k (i) &= 0 \quad \left( a - \textstyle{\frac{p-3}{2}} \le i \le \textstyle{\frac{p-3}{2}}\right), \\
p_k (i)+ p_k \left(i +\textstyle{\frac{p-1}{2}}\right) &= 0 \quad
\left( 0 \le i \le a - \textstyle{\frac{p-1}{2}}\right) .
\end{aligned}
\end{equation}
Now repeat the same argument with a non-split Cartan subgroup.
After a linear change of variables $(X,Y) \mapsto (X',Y')$ over $\bbF_{p^2}$,
this subgroup acts as
$\{ \left( \begin{smallmatrix}x \\ &x^p\end{smallmatrix}\right) :\ x \in \bbF_{p^2}^\times ,\ x^{p+1} = 1\}$.
In this new basis of $V$ we have corresponding elements $\Delta_k'$, $\delta_k'$, $\pi_\ell'$.
However, $p_k$ is unchanged as it is given by the explicit formula in the proof of Lemma~\ref{5a}. We thus get
\begin{equation}
\label{eq2}
\begin{aligned}
p_k (i) &= 0 \quad \left( a - \textstyle{\frac{p-1}{2}} \le i \le \textstyle{\frac{p-1}{2}}\right), \\
p_k (i )+ p_k \left( i +\textstyle{\frac{p+1}{2}}\right) &= 0 \quad
\left( 0 \le i \le a - \textstyle{\frac{p+1}{2}}\right) .
\end{aligned}
\end{equation} From
\eqref{eq1} and \eqref{eq2} we get that $p_k (\ell)=0$  for all $0\le \ell \le a$.
This contradicts the fact that $\deg {p_k} = k \le \frac{p-1}{2} \le a$.

(b) Suppose that $a=p-1$, so that $p\ge 5$ by our assumption.
By Lemma \ref{lm2a}, $\hd_{\SL_2 (\fp)} (\End V) \cong
\bigoplus\limits_{i=0}^{(p-1)/2} L(2i) \oplus L(p-1)$.

(b1) Suppose that $M$ does not contain $L(2k)$ for some $k \le \frac{p-3}{2}$.
Then $\delta_k$ and $\delta_k'$ annihilate the image of all $p'$-elements, so by an argument analogous to the one in (a) we get
\begin{equation}
\label{eq3}
\begin{aligned}
p_k (i) + p_k \left( i + \textstyle{\frac{p-1}{2}}\right) &= 0 \quad \left( 0 <
i < \textstyle{\frac{p-1}{2}}\right), \\
p_k (0 )+ p_k \left( \textstyle{\frac{p-1}{2}}\right) + p_k (p-1) &= 0;
\end{aligned}
\end{equation}
\begin{equation}
\label{eq4}
\begin{aligned}
p_k (i )+ p_k \left( \textstyle{i + \frac{p+1}{2}}\right) &= 0 \qquad
\left( 0 \le i \le \textstyle{\frac{p-3}{2}}\right), \\
p_k \left( \textstyle{\frac{p-1}{2}}\right) &= 0.
\end{aligned}
\end{equation}
Then $p_k (z+1) - p_k (z)$ is a polynomial of degree $k-1 < \textstyle{\frac{p-1}{2}}$
with zeroes at $z=0, 1, \ldots, \textstyle{\frac{p-5}{2}}$ and
$z = \textstyle{\frac{p+1}{2}}, \textstyle{\frac{p+3}{2}}, \ldots, p-2$.
As $p-3 \ge \textstyle{\frac{p-1}{2}}$, it follows that $p_k (z+1) \equiv p_k(z)$,
hence by \eqref{eq4} we get $p_k (\ell)=0$  for all $0\le \ell \le p-1$, contradicting the fact that
$p_k$ has degree $0 \le k < p$.

(b2) Suppose that $M$ does not contain $L(p-1)^{\oplus 2}$.
Note first that the second copy of $L(p-1) \subset \End(V)$ is contained in the Weyl module
$V(2p-2)\hookrightarrow T(2p-2)$.
Using \eqref{eq:2} we have $V(2p-2) |_{\SL_2 (\fp)} \cong L(p-1) \oplus M$, where
$0 \to L(0) \to M \to L(p-3) \to 0$ is non-split.
It follows using \eqref{eq:3} that $V(2p-2)^{{\mathcal U}(\fp)} = V(2p-2)^{\mathcal U}$ (both are two-dimensional).
Hence there is a ${\mathcal U}(\fp)$-fixed vector in the second copy of $L(p-1)$ of
the form $v := \big( X \frac{\p}{\p Y}\big)^{p-1} + c$ for some $c\in \fpb$. We first compute $c$.
Note that if $V$ is an $\SL_2 (\fp)$-representation over $\fpb$ and $v\ne 0$
is fixed by the Borel $B := (\begin{smallmatrix} *&* \\ &*\end{smallmatrix})\subset \SL_2 (\fp)$,
then $v$ generates the $p$-dimensional irreducible representation of
$\SL_2 (\fp)$ iff
\[
\sum_{\SL_2 (\fp) / (\begin{smallmatrix} *&* \\ &*\end{smallmatrix})} gv = 0\  \Leftrightarrow
\sum_{u\in  \fp} \begin{pmatrix} 1 \\ -u & 1\end{pmatrix} v +
\begin{pmatrix} &-1 \\ 1 \end{pmatrix} v = 0 .
\]
As in Lemma~\ref{lm4a},
\[
\begin{pmatrix} 1 \\ -u & 1\end{pmatrix} \cdot \Big( X \frac{\p}{\p Y}\Big)^{p-1} =
(X-uY)^{p-1} \sum_{i=0}^{p-1} (-u)^i \Big( \frac{\p}{\p X}\Big)^i
\Big( \frac{\p}{\p Y}\Big)^{p-1-i} ,
\]
hence
\[
\sum_{u\in\fp} \begin{pmatrix}1 \\ -u&1\end{pmatrix} \cdot \left[ \Big( X \frac{\p}{\p Y}
\Big)^{p-1} + c \right] = - \left[ \Delta_{p-1} + Y^{p-1} \cdot \Big( \frac{\p}{\p X}\Big)^
{p-1}\right] .
\]
Since
\[
\begin{pmatrix} &-1 \\ 1\end{pmatrix} \cdot \left[ \Big( X \frac{\p}{\p Y}\Big)^{p-1} + c
\right] = \Big( Y \frac{\p}{\p X}\Big)^{p-1} + c ,
\]
we deduce $c=-1$.

Consider the annihilator $M^\perp \subset \soc_{\SL_2 (\fp)} (\End V)$ of $M$ under the trace pairing. By assumption, $N :=
M^\perp \cap L(p-1)^{\oplus 2} \ne 0$. Let $\psi \in N^B - \{0\}$, so that by the previous paragraph we can write
$\psi = \lambda \big( X \frac{\p}{\p Y}\big)^{\frac{p-1}{2}} + \mu (\big( X \frac{\p }{\p Y}\big)^{p-1} - 1)$ for some
$(\lambda,\mu) \in \fpb^2-\{0\}$. As $\psi \in M^\perp$ we get by a simple calculation that $0 = \tr\big((\begin{smallmatrix}
  \alpha & \\ & \alpha^{-1} \end{smallmatrix}) \circ \psi\big) = -\mu$ for any $\alpha \in \fp^\times-\{\pm 1\} \ne
\varnothing$.  Thus we may assume that $\psi = \big( X \frac{\p}{\p Y}\big)^{\frac{p-1}{2}}$. As the $\SL_2
(\fp)$-subrepresentation of $\End(V)$ generated by $\psi$ is the unique $\SL_2$-subrepresentation $L(p-1) \subset \End(V)$, we see
that $N$ contains $\Delta_k$ and $\Delta_k'$ for $k = \frac{p-1}2$, so $\delta_k$ and $\delta_k'$ annihilate $M$.
Now the argument of (b1) gives a contradiction.
\end{proof}

\section{Weak adequacy for Chevalley groups}

\begin{lem}
\label{lm1}
Suppose $(X,\Phi, X^\vee, \Phi^\vee)$ is a reduced based root datum with $\Phi$
irreducible.

\begin{enumerate}[{\rm(a)}]
\item If $\Phi$ is not of type $\rA_1$, then
\[
2\al_0^\vee \leq \sum_{\al \in\Phi_+} \al^\vee ,
\]
where $\al_0^\vee$ is the highest coroot.

\item
If $\Phi$ is not of type $\rA_1$, $\rA_2$, $\rA_3$, $\rB_2$, then
\[
4\beta_0^\vee \le \sum_{\al\in\Phi_+} \al^\vee ,
\]
where $\beta_0^\vee$ is the highest short coroot.
\end{enumerate}
\end{lem}

\begin{proof}
(a) Let $\{\al_i:\ i=1,\ldots, r\}$ denote the simple roots.
Then $\ang{\al_0^\vee,\al_i}\ge 0$ for all $i$ and $\ang{\al_0^\vee , \al_j}>0$
for some $j$.  Since $\al_0^\vee \ne \al_j^\vee$ (as $\Phi$ is not
of type $\rA_1$),
$\beta^\vee := \al_0^\vee - \al_j^\vee \in \Phi^\vee$.
Since $\al_0^\vee = \al_j^\vee + \beta^\vee$ it follows that $\beta^\vee > 0$.
Also, $\al_j^\vee \ne \beta^\vee$, as $\Phi$ is reduced.
Hence
\[
2\al_0^\vee = \al_0^\vee + \al_j^\vee + \beta^\vee \le \sum_{\al\in\Phi_+}\al^\vee .
\]

(b) We pass to the dual root system to simplify notation.
We want to show that
\[
4\beta_0 \le \sum_{\al \in \Phi_+} \al ,
\]
where $\beta_0$ is the highest short root.
It suffices to express $\beta_0$ as sum of positive roots in three
non-trivial ways that do not overlap (similarly as in proof of (a)).
If $\Phi$ is not simply laced, we only need two non-trivial ways because we
can also use that
$\beta_0 < \al_0$, where $\al_0$ is the highest root.

In the following we use Bourbaki notation.

\begin{description}[itemsep=5pt,parsep=5pt,font=\normalfont]
\item[{\rm Type} $\mathrm{A}_{n-1}$ $(n\ge 5)$]
\begin{align*}
\beta_0 &= \ve_1 - \ve_n \\
&= (\ve_1 - \ve_i ) + (\ve_i - \ve_n) &(1< i < n).
\end{align*}

\item[{\rm Type} $\mathrm{B}_n$
$(n\ge 3)$]
\begin{align*}
\beta_0 &= \ve_1 \\
&= (\ve_1 - \ve_i) + \ve_i &(1 < i \le n) .
\end{align*}
If $n=3$, we also use $\beta_0 < \al_0 = \ve_1 + \ve_2$.

\item[{\rm Type} $\mathrm{C}_n$
$(n\ge 3)$]
\begin{align*}
\beta_0 &= \ve_1 +\ve_2 \\
&= (\ve_1 - \ve_i) + (\ve_2 + \ve_i) &(2 < i \le n) \\
&= (\ve_1 + \ve_i) + (\ve_2 -\ve_i) .
\end{align*}
If $n=3$, we also use $\beta_0 < \al_0 = 2\ve_1$.

\item[{\rm Type} $\mathrm{D}_n$
$(n\ge 4)$]
\begin{align*}
\beta_0 &= \ve_1 + \ve_2 \\
&= (\ve_1 - \ve_i) + (\ve_2 + \ve_i) &(2 < i \le n) \\
&= (\ve_1 + \ve_i) + (\ve_2 -\ve_i) .
\end{align*}

\item[{\rm Type} $\mathrm{E}_6$]
\[
\beta_0 =\tfrac12 (\ve_1 + \ve_2 + \ve_3 + \ve_4 + \ve_5 - \ve_6 -\ve_7 + \ve_8).
\]
Note that $\beta_0 - (\ve_i + \ve_j)$ and $\ve_i + \ve_j$ are positive
$(1 \le i  < j \le 5)$.

\item[{\rm Type} $\mathrm{E}_7$]
\begin{align*}
\beta_0 &= \ve_8 -\ve_7 \\
&=\frac12 \Big( \ve_8 - \ve_7 +\ve_6 + \sum_{i=1}^5 (-1)^{v (i)} \ve_i\Big)
+\frac12 \Big( \ve_8 - \ve_7 -\ve_6 - \sum_{i=1}^5 (-1)^{v (i)} \ve_i\Big),
\end{align*}
where $\sum_{i=1}^5 v (i)$ is odd.

\item[{\rm Type} $\mathrm{E}_8$]
\begin{align*}
\beta_0 &= \ve_7 + \ve_8 \\
&= (-\ve_i + \ve_7) + (\ve_i +\ve_8) &(1\le i < 7) .
\end{align*}

\item[{\rm Type} $\mathrm{F}_4$]
\begin{align*}
\beta_0 &= \ve_1 \\
&= (\ve_1 - \ve_i) + \ve_i &(1 < i \le 4).
\end{align*}

\item[{\rm Type} $\mathrm{G}_2$]
\begin{align*}
\beta_0 &= 2\al_1 + \al_2 \\
&= \al_1 + (\al_1 + \al_2) \\
\beta_0 &< 3\al_1 + \al_2 \\
\beta_0  &< \al_0 = 3\al_1 + 2\al_2 .\qedhere
\end{align*}
\end{description}
\end{proof}

We now prove variants of several results in \cite{GHTT}.

\begin{lem}
\label{lm2}
Suppose that $\cG$ is a connected simply connected semisimple
algebraic group over $\fpb$ and $\Theta : \cG \to \GL(V)$ a semisimple
 finite-dimensional representation. Let $\cG > \cB > \cT$ denote a Borel
 subgroup and a
 maximal torus, and suppose that
 \begin{equation}
   \label{eq:4}
   \parbox{0.9\linewidth}{for any irreducible component $V'$ of $V$ and for any two distinct weights
     $\mu_1, \mu_2$ of $\cT$ on $V'$ we have $\mu_1-\mu_2\not\in p X (\cT)$.}
 \end{equation}
Then there exist connected, simply connected, semisimple algebraic subgroups
$\mathcal I$ and $\mathcal J$ of $\mathcal G$ such that
${\mathcal G} = {\mathcal I} \times {\mathcal J}$,
$\Theta({\mathcal J}) =1$, and $\Theta$ induces a central isogeny of $\mathcal I$ onto its image, which is a semisimple algebraic group.
Moreover, assumption \eqref{eq:4} holds if for all irreducible constituents
 $V'$ of $V$  the highest weight
 of $V'$ is $p$-restricted and either
 \begin{enumerate}[{\rm (i)}]
 \item $\dim V' < p$ , or
 \item $\dim V' \le p$ and ($p\ne 2$ or $\mathcal G$ has no $\SL_2$-factor).
 \end{enumerate}
  \end{lem}

  \begin{proof}
 Write $V=\bigoplus V_i$ with $V_i$ irreducible and $\cG=\prod\limits_{s\in S} \cG_s$
 with each $\cG_s$ almost simple.
 The last sentence of the proof of Lemma 4 in \cite{GHTT} together with \eqref{eq:4} show that the conclusion
 of that lemma holds for $\Theta_i$: $\cG\to \GL(V_i)$ for all $ i$.
 Hence there exists $S_i \subset S$ such that $\ker \Theta_i = \prod_{s\in S_i}
 \cG_s \times Z_i$, where $Z_i$ is a central subgroup of $\prod_{s\not\in S_i} \cG_s$ (maybe non-reduced).
 Then $\ker \Theta = \bigcap \ker \Theta_i = \prod_{s\in \cap S_i} \cG_s \times Z$, where $Z$ is a central
 subgroup of $\prod_{s\not\in \bigcap S_i} \cG_s$. So we can take
 ${\mathcal I} = \prod_{s \notin \cap S_i}{\mathcal G}_s$ and
${\mathcal J} = \prod_{s \in \cap S_i}{\mathcal G}_s$.

 To prove the last part, we may suppose that $V$ is irreducible.  So
 $V \cong \bigotimes_{s\in S} V_s$, where $V_s$ is an irreducible
 $\cG_s$-representation.
 It is easy to see that if \eqref{eq:4} fails, then it fails for $\cG_s \to \GL(V_s)$ for some $s\in S$,
 so we may assume that $\cG$ is almost simple.

 \smallskip
 (a) First suppose that $\cG\cong \SL_2$.
The highest weight of $V$ is $\left( \begin{smallmatrix} x \\ &x^{-1}\end{smallmatrix}\right) \mapsto x^a$,
 some $0 \le a \le p-1$ and $a\ne p-1$ if $p=2$.
 Therefore the weights of $\ad V$ are $\left(\begin{smallmatrix} x\\ &x^{-1}\end{smallmatrix}\right) \mapsto x^b$ where
 $b\in \{ -2 a , -2 a + 2 , \ldots, 2a-2 , 2a \}$.
 It follows that \eqref{eq:4} holds because $b\equiv 0$ (${\rm mod}\  p$) implies that $b = 0$.

 \smallskip
 (b) Next suppose that $\cG\not\cong \SL_2$.
 Let $\la$ denote the highest weight of $V$; it is $p$-restricted by assumption.
 By Lemma \ref{lm1}(a) and Jantzen's inequality \cite[Lemma 1.2]{jantzenlow} we get
 \[
 |\ang{\mu, \beta^\vee}| \le \ang{\la , \al_0^\vee} \le \frac12 \big\langle \la,\sum_{\al >0}
 \al^\vee\big\rangle < \frac12 \dim V \leq \frac p2
 \]
for all weights $\mu$ of $V$
and all roots $\beta$.
Hence $| \ang{\mu_1 - \mu_2 , \beta^\vee} | < p$
for all root $\beta$
and all weights $\mu_i$ of $V$, so \eqref{eq:4} holds.
 \end{proof}

 \begin{lem}
 \label{lm3}
 Suppose that $\cG \leq \prod \GL(W_i)$ is a connected reductive group over~$\fpb$,
 where for all $i$ the representation $W_i$ is irreducible with $p$-restricted highest weight and
 $\dim W_i \le p$.
   Let $\cT$ be a maximal torus and $\cU$ be the unipotent radical of a Borel subgroup of $\cG$ that contains $\cT$.
   Let $V = \bigoplus W_i$.
  \begin{enumerate}[\rm(i)]
  \item The maps $\exp$ and $\log$ induce inverse isomorphisms of varieties between $\Lie \cU \leq \End(V)$ and $\cU \leq \GL(V)$.
  \item For any positive root $\alpha$ we have $\exp(\Lie \cU_\alpha) = \cU_\alpha$.
  \item The map $\exp : \Lie \cU \to \cU$ depends only on $\cG$ and $\cU$, but not on $V$, $W_i$, or the representation $\cG
    \hookrightarrow \GL(V)$.
  \item If $\theta$ is an automorphism of $\cG$ that preserves $\cT$ and $\cU$, then we have a commutative diagram:
    \begin{equation*}
      \xymatrix{\Lie \cU \ar[r]^{d\theta} \ar[d]_{\exp} & \Lie \cU \ar[d]^{\exp} \\ \cU \ar[r]^\theta & \cU }
    \end{equation*}
  \end{enumerate}
 \end{lem}

 \begin{proof}  The proof is the same as that of \cite[Lemma 5]{GHTT}
 where there was an extra assumption on the $\mu_i$.
 The assumption on the weights $\mu_i$ is {\it only} used to prove that
 $X_{\al , n}$ acts trivially on $V = \bigoplus W_i$ for all $n \ge p$.
Fix any $i$.
It is enough to show that
$X_{\al ,n}$ acts trivially on $W_i$ for  all $n \ge p$.
 So  it is enough to show that  $W_i$ cannot have two weights $\la$ and $\la + n\al$ ($\al \in \Phi$, $n \ge p$).
 As $\dim W_i \le p$ it follows from  \cite{jantzenlow}
 that the weights
 of $W_i$ are the same as those of the irreducible characteristic $0$ representation of the same highest weight.
 But in characteristic $0$ it is known that if $\la$ and $\la + n\al$ are weights
 of an irreducible representation, then so are $\la$, $\la + \al$, $\la + 2\al , \ldots, \la + n\al$, so $\dim W_i > n \ge p$, contradicting the assumption.
 \end{proof}

 \begin{prop}
 \label{prop4}  Let $p > 3$ be prime.
 Suppose that $V$ is a finite-dimensional vector space over $\overline{\bbF}_p$ and
 that $G \leq \GL(V)$ is a finite subgroup that
  acts semisimply on $V$. Let $\GP \leq G$ be the subgroup generated
  by $p$-elements of $G$. Then $V$
  is a semisimple $\GP$-module. Let $d \geq 1$ be the maximal dimension of an irreducible $\GP$-submodule of $V$.
  Suppose that $p > d$ and that $\GP$ is a central product of quasisimple
  Chevalley groups in characteristic $p$.
  Then there exists an algebraic group $\cG$ over $\bbF_p$ and a semisimple representation
  $\Theta : \cG_{/\fpb} \to \GL(V)$ with the following properties:
  \begin{enumerate}[\rm(i)]
  \item The connected component $\cG^0$ is semisimple, simply connected.
  \item $\cG \cong \cG^0 \rtimes H$, where $H$ is a finite group of order prime to $p$.
  \item $\Theta(\cG(\bbF_p)) = G$.
  \item $\ker(\Theta) \cap \cG^0(\fp)$ is central in $\cG^0(\fp)$.
  \end{enumerate}
  Moreover, any highest weight of $\cG^0_{/\fpb}$ on $V$ is $p$-restricted.
  Also, $G$ does not have any composition factor of order~$p$.
 \end{prop}

 \begin{proof}
 The proof is essentially identical to the proof of \cite[Proposition 7]{GHTT}.
 We do not get $\ang{\la , \al^\vee} < \frac{p-1}{2}$ in Step 2, but this
 was only used to apply
 Lemmas 4 and 5 in  \cite{GHTT}.
 By Lemmas  \ref{lm2} and \ref{lm3} above one can bypass this assumption,
 as we now explain. Both
times Lemma 4 in \cite{GHTT} is applied, condition (ii) in Lemma \ref{lm2} holds. In Step 4 we
can apply Lemma \ref{lm3} instead of Lemma 5 in \cite{GHTT} because $\bar{{\mathcal I}}$ acts irreducibly on $W_i$ and its highest weight is $p$-restricted (as
${\mathcal I} \to \bar{{\mathcal I}}$ is a central isogeny).
Similarly we can avoid Lemma 5 in \cite{GHTT} in Step 5, noting that the highest weights
of $V'$ are Galois conjugate to the highest weights of $V$ and recalling that
$\psi_{/\fpb}$ is a central isogeny onto its image. Finally, note that (iv) follows by construction.
 \end{proof}


 \begin{thm}
 \label{thm6}
 Suppose that $p >3$, $V$ is a finite-dimensional vector space over $\overline{\bbF}_p$,
 and that
 $G \leq \GL(V)$ is a finite subgroup that
  acts irreducibly on $V$. Let $\GP \leq G$ be the subgroup generated by
  $p$-elements of $G$. Let $d \geq 1$ be the maximal dimension of an
  irreducible $\GP$-submodule of $V$.
  Suppose that $p > d$ and that $\GP$ is a central product of quasisimple
  Chevalley groups in characteristic $p$.
  Then the set of $p'$-elements of $G$ spans $\ad V$ as an $\fpb$-vector space.
  \end{thm}

\begin{remark}
Theorem \ref{thm6} generalizes \cite[Theorem 9]{GHTT}. We take the
opportunity to point out a small gap in the last paragraph of the proof of
\cite[Theorem 9]{GHTT}. In its notation, it is implicitly assumed that (i) $r(T(\bbF_l)) \subset r(H)$, so
that the span of $r(H)$ equals the span of $r(T(\bar{\bbF}_l)H)$, and (ii) $H$ normalizes
the pair $(B,T)$. Both assumptions are satisfied provided that when we apply
\cite[Proposition 7]{GHTT} in the proof of \cite[Theorem 9]{GHTT} we take $r$,
$G = G^0 \rtimes H$, $B$, $T$, $\ldots$ as constructed in the proof of that proposition.
\end{remark}

 \begin{proof}
 Without loss of generality, $d>1$.
Let $\Theta$: $\cG_{/\fpb }\to \GL(V)$  be as in Proposition \ref{prop4}. Then $V = \bigoplus W_i$, where
$W_i$ is an irreducible $\cG^0_{/\fpb}$-subrepresentation with $p$-restricted highest weight.
Write $\cG^0_{/\fpb} \cong \cG_1 \times \ldots \times \cG_r$, where
$\cG_i$ is almost simple over $\fpb$.  Let $\cG^0 > \cB> \cT$
denote a Borel subgroup and a maximal torus, and let $\Phi$ denote the roots of
$\cG^0_{/\fpb}$ with respect to $\cT_{/\fpb}$.

\smallskip
(a) We consider the case where one of the $W_i$ (equivalently any) is tensor decomposable as a $\cG^0_{ / \fpb}$-representation.
Note that $W_i \cong X_{i1} \boxtimes \cdots \boxtimes X_{ir}$ where $X_{ij}$ is an irreducible
$\cG_j$-representation with $p$-restricted highest weight.
Since $\dim X_{ij} \le p-1$, its highest weight lies in the lowest alcove (by \cite{jantzenlow, Serre1}),
hence $X_{ij}$ is tensor indecomposable (as the highest
weight is in the lowest alcove, we are reduced to the characteristic $0$ case,
where this is well known). 
Hence our assumption implies that $X_{ij}\not\cong \bbone$ for at least
two values of $j$.
Hence $\dim X_{ij} \le \frac{p-1}{2}$ for all $i,j$.
Therefore $X_{ik}^* \otimes X_{jk}$ is a semisimple $\cG_k$-representation by \cite{Serre1}
so
$\End V$ is a semisimple $\cG^0_{ / \fpb}$-representation.
Moreover, all its highest weights are $p$-restricted: this follows exactly as
in Step 2 of the proof of \cite[Proposition 7]{GHTT} (use that $\dim X_{ij} \le \frac{p-1}{2}$).
Hence any $\cG^0 (\bbF_p)$-submodule of $\End V$ is a $\cG^0 (\fpb)$-submodule.

Furthermore, arguing as in Step 2 of the proof of \cite[Proposition 7]{GHTT} for
each $\cG_k$, we deduce that for all weights $\mu$ of the maximal torus
$\cT_{/\fpb}$ on $V$ we have
$|\ang{\mu,\al^\vee}| < \frac{p-1}{2}$ for all $\al\in\Phi$.
We conclude as in the last paragraph of the proof of  \cite[Theorem 9]{GHTT}.

\smallskip
(b) We consider the case where $\cG^0 _{/ \fpb}$ has no
factors of type $\rA_1$, $\rA_2$, $\rA_3$, $\rB_2$.
We claim that
$|\ang{\mu,\al^\vee}| < \frac{p-1}{4}$
for all weights $\mu$ of $\cT_{/\fpb}$ on $V$ and for all short coroots
$\al^\vee \in \Phi^\vee$.
It suffices to show that
$\ang{\la , \beta_0^\vee} < \frac{p-1}{4}$ for all highest weights
$\la$ of $\cT_{/\fpb}$ on $V$ and all
highest short coroots $\beta_0^\vee$
(one for each component of $\cG^0_{ / \fpb}$).
So it is enough to show that if $\cG'$ is an almost simple, simply connected
group over $\fpb$, not of type $\rA_1$, $\rA_2$, $\rA_3$, $\rB_2$, then
$\ang{\la, \beta_0^\vee} < \frac{p-1}{4}$ for all $p$-restricted weights $\la$ of $\cG'$ such that $\dim L(\la) < p$,
where $\beta_0^\vee$ is the highest short coroot of $\cG'$.
But this follows from Lemma \ref{lm1}(b) and Jantzen's inequality and this proves
the claim.

Since the short coroots span $X_* (\cT_{/\fpb}) \otimes \bbQ$ over $\bbQ$, \cite[Lemma 3]{GHTT}
plus the claim show that distinct weights of $\cT_{/\fpb}$ on $\End V$ (resp.\ $V$)
remain distinct on $\cT(\fp)$.
Then  \cite[Lemma 1.1]{Gapp}
 shows that any $\cG^0 (\fp)$-subrepresentation of $\End V$ is $\cG^0 (\fpb)$-stable,
 so we can conclude as in the last paragraph of the proof of Theorem 9 in \cite{GHTT}.

(c) If neither (a) nor (b) apply, then the $W_i$ are tensor indecomposable,
in particular the almost simple factors of $\cG^0 _{/ \fpb}$ are pairwise isomorphic.
(Write $\cG^0 \cong \prod \cH_i$,
where the subgroups $\cH_i$ are almost simple over $\fp$. Note that for each $i$,
$\cG^0(\fp)$ acts irreducibly on $W_i$ with all but one $\cH_j(\fp)$ acting trivially. As
$\cG(\fp)$ is irreducible on $V$ and by Proposition \ref{prop4}(iv), the subgroups $\cH_i(\fp)$ are pairwise isomorphic and, as $p > 3$,
so are the $\cH_i$.)
Hence $\cG^0_{ / \fpb} \cong \SL_2^n$ or $\SL_3^n$ or $\SL_4^n$ or $\Sp_4^n$, some $n\ge 1$.

\smallskip
(d) We consider the case where $\cG^0_{ / \fpb }\cong \SL_3^n$ or $\SL_4^n$ or $\Sp_4^n$.
We claim that for all weights $\mu$ of $\cT_{/ \fpb}$ on $V$ and for all $\alpha \in \Phi$,
\begin{equation}
|\ang{\mu,\al^\vee}|  < \frac12 (p-1).\label{eq:5}
\end{equation}
To see this, note that $|\ang{\mu, \alpha^\vee}| \le \ang{\lambda, \alpha_0^\vee}$ for some highest
weight $\lambda$ of $V$ and some highest coroot $\alpha_0^\vee$. Applying Lemma \ref{lm1}(a) to the
component $\Phi_j$ of $\Phi$ such that $\alpha_0^\vee \in \Phi_j^\vee$ and using Jantzen's
inequality we get
$$\ang{\lambda,\alpha_0^\vee} \le \frac12 \sum_{\Phi_{j,+}} \ang{\lambda,\alpha^\vee} < \frac12 (p-1).$$
By Lemma 3 in \cite{GHTT}, \eqref{eq:5} shows that
distinct weights of $\cT_{/\fpb}$ on $V$ remain distinct on
$\cT(\fp)$.
As usual, it thus suffices to show that $\End V$ is a semisimple $\cG^0_{ / \fpb}$-module
with $p$-restricted highest weights.
We can argue independently for each factor of $\cG^0_{ / \fpb}$ so it suffices to show that
if $X, Y$ are nontrivial irreducible $\cG'$-representations
which are conjugate by $\Aut(\cG')$
(with $\cG' = \SL_3$ or $\SL_4$
or $\Sp_4$) with $p$-restricted highest weights $\la$, $\la'$
of dimension less than $p$, then $X\otimes Y$
is semisimple with $p$-restricted highest weights.
By \cite{jantzenlow, Serre1}, $\la$ and $\la'$ lie in the lowest alcove, so $\ch L(\la)$ and
$\ch L(\la')$ are given by Weyl's character formula.

In the following, note that $\ang{\lambda,\beta_0^\vee} = \ang{\la',\beta_0^\vee}$.

If $\cG'\cong \SL_4$,
write $\la = r \varpi_1 + s \varpi_2 + t\varpi_3$, $r,s,t \ge 0$, where $\varpi_i$ is the $i$-th
fundamental weight.
Then
\begin{align*}
p-1 &\ge \dim L(\la) = \frac
{[(r+1)(s+1)(t+1)][(r+s+2)(s+t+2)](r+s+t+3)}{2\cdot 2\cdot 3} \\
&\ge \frac{(r+s+t+1)(r+s+t+2)(r+s+t+3)}{2\cdot 3}.
\end{align*}
If
$\ang{\la,\beta_0^\vee} = r+s+t \ge \frac{p-1}{4}$, then
\[
p-1 \ge \frac{\frac{p+3}{4}\cdot\frac{p+7}{4}\cdot\frac{p+11}{4}}{6} .
\]
Equivalently $(p-5) [(p +13)^2 - 292] \leq 0$, i.e.\ $p =5$.
In this case, equality holds throughout so $\la=\varpi_1$ or $\varpi_3$.
The maximal weight of $X\otimes Y$, namely $2\varpi_1$ or
$\varpi_1 + \varpi_3$ or $2\varpi_3$, lies in the closure of the lowest alcove.
Then $X\otimes Y$ is semisimple by the linkage principle (or just \cite{jantzen},
Proposition II.4.13) and it has $p$-restricted highest weights.
If $\ang{\la,\beta_0^\vee} < \frac{p-1}{4}$ the argument in (b) goes through
instead.

If $\cG'\cong \Sp_4$,
write $\la = r\varpi_1 + s\varpi_2$, $r, s \ge 0$ (type ${\rm B}_2$).
Then
\begin{align*}
p-1 &\ge \dim L(\la) = \frac{[(r+1)(s+1)](r+s+2)(2r+s+3)}{6} \\
&\ge \frac{(r+s+1)(r+s+2)(r+s+3)}{6} \ .
\end{align*}
If $\ang{\la , \beta_0^\vee} = r+s \ge \frac{p-1}{4}$, then $p=5$ as above and
$\la = \varpi_2$.
Again, $X\otimes Y$ has maximal weight $2\varpi_2$ lying in the closure
of the lowest alcove, hence $X\otimes Y$ is semisimple with $p$-restricted highest
weights.
If $\ang{\la,\beta_0^\vee} < \frac{p-1}{4}$ we are done as in (b).

If $\cG'\cong \SL_3$,
write $\la = r\varpi_1 + s\varpi_2$, $r,s \ge 0$.
If $r+s \ge \frac{p-1}{2}$, then
\begin{align*}
p-1 &\ge \dim L(\la) = \frac{[(r+1)(s+1)](r+s+2)}{2} \\
&\ge \frac{(r+s+1)(r+s+2)}{2} \ge \frac{\frac{p+1}{2}\cdot\frac{p+3}{2}}{2}\ .
\end{align*}
Equivalently $(p-2)^2 + 7 \le 0$, which is impossible.
Hence $r+s \le \frac{p-3}{2}$, which implies that the maximal weight of $X\otimes Y$
lies in the lowest alcove.
So $X\otimes Y$ is semisimple with $p$-restricted highest weights.

\smallskip
(e) We consider the case where $\cG^0_{/\fpb} \cong \SL_2^n$ and each $W_i$ is
tensor indecomposable as a $\cG^0_{/\fpb}$-representation.
Here, $\cG^0 (\fp) \cong \SL_2 (\fq )^m$, where $[\fq : \fp]\cdot m = n$.
Also, $V$ is irreducible, each $W_i$ is tensor indecomposable, and $\SL_2$ has no
outer automorphism. It follows that $V\cong \left[ \bigoplus\limits_{i=1}^\ell V_i\right]^{\oplus k}$
 as $\cG^0_{/\fpb}$-representations, where each $V_i$ is of the form
 $\bbone \boxtimes \cdots \boxtimes V_0
 \boxtimes \cdots \boxtimes \bbone$ (precisely one factor is $V_0$), the $V_i$ are
 pairwise non-isomorphic, and $V_0$ is an irreducible $\SL_2$-representation
 such that $1 < \dim V_0 < p$ with $p$-restricted highest weight.

 \smallskip
 (e1) We claim that the span of the $p'$-elements of $\cG^0 (\fp)$ in $\End V$ contains the span of $\mathcal T(\fpb)$
 in $\End V$.

 If $ q>p$, note from the description of $V_i$ above that distinct weights of $\cT_{/\fpb}$
 on $V$ remain distinct on $\cT(\fp)$.
 Hence the span of $\cT(\fp)$ in $\End V$ equals the span of $\cT(\fpb)$ in $\End V$.

 If $q=p$, we will show that the $p'$-elements of $\cG^0(\fp)$ span the same subspace
 of $\End V$ as all of $\cG^0(\fpb)$. First, from Proposition  \ref{prop4}(iv) we deduce that
$\ell = n$.
 As the $V_i$ are distinct and irreducible $\cG^0_{/\fpb}$-representations, by the
 Artin-Wedderburn theorem
 we need to show that the $p'$-elements in $\cG^0(\fp)$ span $\prod\limits_{i=1}^n
 \End (V_i)$, or equivalently its $\cG^0 (\fp)$-head.
 (Note that the span of the $p'$-elements is $\cG^0(\fp)$-stable.)
 By Lemma~\ref{lm2a} we see that
 the $n$ representations $\hd_{\cG^0(\fp)} \big( \End (V_i)\big)$
 have no $\cG^0(\fp)$-irreducible constituent in common except for the trivial
 direct summand of scalar matrices in $\End (V_i)$.
 By Proposition \ref{pr1a} the image of the $p'$-elements span $\End (V_i)$ for any $i$.
 Hence it suffices to show that the image of the $p'$-elements under the map
 \begin{align*}
 \cG^0 (\fp) \ &\to \ \fpb^n \\
 g\ &\mapsto \ \big ( \tr (g |_{V_i})\big)_{i=1}^n
 \end{align*}
 spans $\bbF_p^n$.
 Note that as $1 < \dim V_0 < p$ the split torus $\left(\begin{smallmatrix} * \\ & *\end{smallmatrix}\right)
 < \SL_2 (\fp)$ has a non-trivial eigenvalue $\chi$ on $V_0$
 with multiplicity $1$ or $2$.
 Given $1\le i \le n$ there exists an element in $\fp [\cT(\fp)]$ that projects
 onto the $1\otimes \cdots \otimes \chi \otimes\cdots\otimes 1$ eigenspace in any
 $\cT(\fp)$-representation, so as $p > 2$ it has non-zero trace on $V_i$ but is zero
 on $\bigoplus\limits_{j\ne i} V_j$.
 This proves the claim.

 \smallskip
 (e2) We claim that $\hd_{\cG^0_{/\fpb}}(\End V) = \hd_{\cG^0 (\fp)} (\End V)$
 and moreover any highest weight of this representation is $p$-restricted.

 If $d \le \frac{p+1}{2}$, then $\End V$ is a semisimple $\cG^0_{/\fpb}$-module
 by \cite{Serre1} and clearly any highest weight of $\End V$ is $p$-restricted.
The claim follows.

 If $d\ge \frac{p+3}{2}$, note that
 $\hd$ is compatible with direct sums, so we can consider
 each $V_{i}^* \otimes V_j$ separately.
  If $i \ne j$, then $V_i^* \otimes V_j$ is irreducible with $p$-restricted highest weight.
 If $i=j$, from Lemma~\ref{lm2a} we get
 \[
 \hd_{\SL_2} (V_0^* \otimes V_0) \cong L(0) \oplus L(2)\oplus \cdots \oplus
 L(p-1) .
 \]
 In particular, any highest weight of $\hd_{\cG^0_{/\fpb}} (V_i^* \otimes V_i)$ is
 $p$-restricted.
To see
$$\hd_{\cG^0_{/\fpb}} (V_i^*\otimes V_i) = \hd_{\cG^0(\fp)} (V_i^* \otimes V_i),$$
by Lemma \ref{lm2a} it is equivalent
(after a Frobenius twist) to show that
$$\hd_{\SL_2} (T(2p-2-2j)) =\hd_{\SL_2(\fq)} (T(2p-2-2j))$$
for $0 \le j \le (p-3)/2$. If
$q = p$ this follows from Lemma \ref{lm2a}, as $d < p$. This in turn implies the statement for $q > p$, as any irreducible $\SL_2$-constituent of $T(2p-2-2j)$ restricts irreducibly to $\SL_2(\fq)$ if
$q > p$ and semisimply to $\SL_2(\fp)$. This proves the claim.

\smallskip
(e3) Let $\cM$ denote the span of the image of
the $p'$-elements of $\cG(\fp)$ in $\hd_{\cG^0(\fp)} (\End(V))$. Note that $\cM$
is a $\cG^0(\fp)$-subrepresentation. To prove weak adequacy, it suffices to show
that $\cM = \hd_{\cG^0(\fp)} (\End(V))$. By (e2) we have that
$\hd_{\cG^0_{/\fpb}} (\End(V)) = \hd_{\cG^0(\fp)} (\End(V))$ and distinct irreducible
$\cG^0_{/\fpb}$-subrepresentations of
$\hd_{\cG^0_{/\fpb}} (\End(V))$ restrict to distinct
irreducible $\cG^0(\fp)$-representations. Hence, any
$\cG^0(\fp)$-subrep\-re\-sentation of
$\hd_{\cG^0_{/\fpb}} (\End(V))$ is $\cG^0(\fpb)$-stable. By (e1) we know that
$\cM$ contains the span of the image of $\cT(\fpb)\cdot H$. Therefore, by Lemma 8 in \cite{GHTT},
$\cM$ contains the span of the image of $\cG(\fpb)$. But the latter span equals
$\hd_{\cG^0_{/\fpb}} (\End(V))$ by the Artin-Wedderburn theorem.
  \end{proof}

\section{Weak adequacy in cross characteristic}\label{sec:cross}
Recall that, given a finite dimensional absolutely irreducible
representation $\Phi~:~G \to \GL(V)$, the pair $(G,V)$ is called {\it weakly
adequate} if $\End(V)$ equals
$$\cM := \langle \Phi(g) \in \Phi(G) \mid \Phi(g) \mbox{ semisimple}\rangle_k.$$
Assume $k = \bar{k}$ has characteristic $p$. First we recall:

\begin{lem}\label{p-sol}{\rm \cite[Lemma 2.3]{G2}}
If $G < \GL(V)$ is $p$-solvable and $p\nmid \dim(V)$, then $(G,V)$ is weakly adequate.
\end{lem}

In general, a key tool to prove weak adequacy is provided by the following criterion:

\begin{lem}\label{key}
Let $V$ be a finite dimensional vector space over $k$ and $G \leq \GL(V)$ a finite
irreducible subgroup. Write $V|_{G^+} = e\sum^t_{i=1}W_i$, where the
$G^+$-modules $W_i$ are irreducible and pairwise non-isomorphic. Suppose
there is a subgroup $Q \leq \GP$ with the following properties:


{\rm (i)} $\{Q^g \mid g \in G\} = \{Q^x \mid x \in \GP\}$; and

{\rm (ii)} The $Q$-modules $W_i$ are irreducible and pairwise non-isomorphic.\\
Then $\bfN_G(Q)$ is an irreducible subgroup of $\GL(V)$. If, furthermore,

{\rm (iii)} $\bfN_{\GP}(Q)$ is a $p'$-group,\\
then $(G,V)$ is weakly adequate.
\end{lem}

\begin{proof}
The condition (i) is equivalent to that $G = N\GP$, where $N := \bfN_G(Q)$. Since
$G/\GP$ is a $p'$-group, this implies that $N$ is a $p'$-group if $\bfN_\GP(Q)$ is
a $p'$-group. By the Artin-Wedderburn theorem, it therefore suffices to show that
$N$ is irreducible on $V$.

Set $V_i = eW_i$ so that $V = \oplus^m_{i=1}V_i$, $G_1 := I_G(W_1) = \Stab_G(V_1)$ the
inertia group of the $\GP$-module $W_1$ in $G$, and $N_1 :=  N \cap G_1$. Then we
have that $G_1 = N_1\GP$ and $[N:N_1] = [G:G_1] = t$. Trivially, the condition (ii)
implies that the $N^+$-modules $W_i$, $1 \leq i \leq t$,
are irreducible and pairwise non-isomorphic, where we set
$N^+ := \bfN_{\GP}(Q)$. It now follows that $N_1 = I_N(W_1)$, the inertia group of the
$N^+$-module $W_1$ in $N$; moreover, $N$ acts transitively on $\{V_1, \ldots ,V_t\}$, and
$V|_{N} = \Ind^N_{N_1}(V_1|_{N_1})$.
By the Clifford correspondence, it suffices to show that the $N_1$-module $V_1$ is
irreducible.

Let $\Phi$ denote the corresponding representation of $G_1$ on $V_1$
and let $\Psi$ denote the corresponding representation of $\GP$ on $W_1$.  By
\cite[Theorem 8.14]{N}, there is a projective representation $\Psi_1$ of $G_1$ such that
$$\Psi_1(n) = \Psi(n), ~~\Psi_1(xn) =  \Psi_1(x)\Psi_1(n), ~~
      \Psi_1(nx) = \Psi_1(n)\Psi_1(x)$$
for all $n \in \GP$ and $x \in G_1$. Let $\alpha$ denote the factor set on $G_1/\GP$ induced
by $\Psi_1$. By \cite[Theorem 8.16]{N}, there is an $e$-dimensional
projective representation $\Theta$ of $G_1/\GP$ with factor set $\alpha^{-1}$ such
that $\Phi(g) = \Theta(g) \otimes \Psi_1(g)$ for all $g \in G_1$. (Here and in what follows, we will write $\Theta(g)$ instead of $\Theta(g\GP)$). Since $\Phi$ is irreducible,
$\Theta$ is irreducible.

Observe that $N_1/N^+$ is canonically isomorphic to $G_1/\GP$. Restricting to
$N_1$, we then have that $\Phi(g) = \Theta(g) \otimes \Psi_1(g)$ for all
$g \in N_1$, $\Psi_1(n) = \Psi(n)$ for all $n \in N^+$, $(\Psi_1)_{N_1}$ is a projective representation of $N_1$ with factor set $\alpha$, and $\Theta_{N_1/N^+}$ is a
projective representation of $N_1/N^+$ with factor set $\alpha^{-1}$. Furthermore,
$\Theta_{N_1/N^+}$ is irreducible.  It follows by \cite[Theorem 8.18]{N} that
$\Phi_{N_1}$ is irreducible, as stated.
\end{proof}

In certain cases we will also need the following modification of Lemma \ref{key}:

\begin{lem}\label{key2}
Let $V$ be a finite dimensional vector space over $k$ and $G \leq \GL(V)$ a finite
irreducible subgroup. Write $V|_{G^+} = e\sum^t_{i=1}W_i$, where the
$G^+$-modules $W_i$ are irreducible and pairwise non-isomorphic. Suppose
there is a subgroup $Q \leq \GP$ with the following properties:

{\rm (i)} $\{Q^g \mid g \in G\} = \{Q^x \mid x \in \GP\}$; and

{\rm (ii)} $W_i \cong A_i \oplus B_i$ as $Q$-modules, where
all the $2t$ $Q$-modules $A_i$ and $B_j$ are irreducible and pairwise non-isomorphic.\\
If $\{A_1, \ldots ,A_t\}$ and $\{B_1, \ldots ,B_t\}$ are two disjoint $N$-orbits on
$\IBr(Q)$ for $N := \bfN_G(Q)$, then we have that $V_N \cong A \oplus B$ as
$N$-modules, where $A$ and $B$ are irreducible,
$A_Q \cong e(\oplus^t_{i=1}A_i)$ and $B_Q \cong e(\oplus^t_{i=1}B_i)$. On the other hand, if $\{A_1, B_1, \ldots ,A_t,B_t\}$ forms a single $N$-orbit, then $N$ is irreducible
on $V$.
\end{lem}

\begin{proof}
Again, the condition (i) implies that $G = N\GP$. Adopt the notations $G_1$,
$N_1$, $N^+$, $\Phi$, $\Psi$, $\Psi_1$, $\alpha$ of the proof of Lemma \ref{key}. As shown there, there is an irreducible $e$-dimensional projective representation
$\Theta$ of $G_1/\GP$ with factor set $\alpha^{-1}$ such that
$\Phi(g) = \Theta(g) \otimes \Psi_1(g)$ for all $g \in G_1$. Also, $N_1/N^+$ is
canonically isomorphic to $G_1/\GP$. According to (ii), $(W_i)_Q \cong A_i \oplus B_i$, with $A_i \not\cong B_i$. Hence we can decompose $(V_i)_Q = C_i \oplus D_i$, where $(C_i)_Q \cong eA_i$ and
$(D_i)_Q \cong eB_i$, and define $A := \oplus^t_{i=1}C_i$, $B := \oplus^t_{i=1}D_i$.

\smallskip
(a) First we consider the case where $\{A_1, \ldots ,A_t\}$ and $\{B_1, \ldots ,B_t\}$
are two disjoint $N$-orbits. Then, for any $x \in N$, every composition factor of the $Q$-module $xA$ is of the form $A_j$ for some $j$, and every composition factor
of $B$ is of the form $B_{j'}$ for some $j'$. Hence we conclude that $xA = A$, and similarly
$xB = B$. Thus $A$ and $B$ are $N$-modules. Certainly, $N$ permutes
$C_1, \ldots, C_t$
transitively and $N_1$ fixes $C_1$. But $t = [N:N_1]$, hence $N_1 = \Stab_N(C_1)$
and $A = \Ind^N_{N_1}(C_1)$. Since $(C_i)_Q = eA_i$ and the $Q$-modules $A_i$ are
pairwise non-isomorphic, we also see that $N_1 = I_N(A_1)$. Similarly,
$N_1 = I_N(B_1)$ and $B = \Ind^N_{N_1}(D_1)$. Therefore, by the Clifford
correspondence, it suffices to prove that the $N_1$-modules $C_1$ and $D_1$
are irreducible.

Recall the decompositions $(W_1)_Q = A_1 \oplus B_1$
and $\Phi(g) = \Theta(g) \otimes \Psi_1(g)$ for all $g \in G_1$.
Without loss, we may assume
that the representation $\Psi$ of $\GP$ on $W_1$ is written with respect to some basis
$(v_1, \ldots ,v_{a+b})$ which is the union of a basis
$(v_1, \ldots ,v_a)$ of $A_1$ and a basis $(v_{a+1}, \ldots ,v_{a+b})$
of $B_1$. Since $\Phi(g) = \Theta(g) \otimes \Psi_1(g)$ for all $g \in G_1$ acting
on $V_1$, we can also choose a basis
$$(u_i \otimes v_j \mid 1 \leq i \leq e, 1 \leq j \leq a+b)$$
of $V_1$ such that $\Theta(g)$ is written with respect to $(u_1, \ldots ,u_e)$ and
$\Psi_1(g)$ is written with respect to
$(v_1, \ldots, v_{a+b})$. For any $x \in N_1$, writing $\Theta(x) = (\theta_{i'i})$ and $\Psi_1(x) = (\psi_{j'j})$ we then have that
$$\Phi(x)(u_i \otimes v_j) = \sum_{i',j'}\theta_{i'i}\psi_{j'j}u_{i'} \otimes v_{j'}.$$
Recall we are also assuming that the $Q$-modules $A_1$ and $B_1$ are not $N$-conjugate. Therefore, $\Phi(x)$ fixes each of
$$C_1 = \langle u_i \otimes v_j \mid 1 \leq i \leq e,1 \leq j \leq a \rangle_k, ~
    D_1 = \langle u_i \otimes v_j \mid 1 \leq i \leq e,a+1 \leq j \leq a+b \rangle_k.$$
In particular, $\theta_{i'i}\psi_{j'j} = 0$ whenever $j' > a$ and $j \leq a$. Now
if $\psi_{j'j} \neq 0$  for some $j \leq a$ and some $j' > a$, we must have $\theta_{i'i} = 0$ for all $i,i'$, i.e.\ $\Theta(x) = 0$, a contradiction.
Similarly, $\psi_{j'j} = 0$ whenever $j > a$ and $j' \leq a$.
Therefore, we can write
\begin{equation}\label{proj2}
  \Psi_1(x) = \diag(\Psi_{1A}(x),\Psi_{1B}(x))
\end{equation}
in the chosen basis $(v_1, \ldots ,v_{a+b})$. It also follows that $\Psi(y)$ fixes
each of $A_1$ and $B_1$ for all $y \in N^+$, i.e.\ $A_1$ and $B_1$ are irreducible
$N^+$-modules.

Now, for any $x,y \in N_1$, $\Psi_1(x)\Psi_1(y) = \alpha(x,y)\Psi_1(xy)$. Together with
(\ref{proj2}) this implies that
$$\Psi_{1A}(x)\Psi_{1A}(y) = \alpha(x,y)\Psi_{1A}(xy),~~
    \Psi_{1B}(x)\Psi_{1B}(y) = \alpha(x,y)\Psi_{1B}(xy),$$
i.e.\ both $\Psi_{1A}$ and $\Psi_{1B}$ are projective representations of $N_1$ with
factor set $\alpha$. Since $\Psi_1(x) = \Psi(x)$ for all $x \in N^+$ and
(\ref{proj2}) certainly holds for $x \in N^+$, we also see that
$\Psi_{1A}$ extends the representation of $N^+$ on $A_1$, and similarly
$\Psi_{1B}$ extends the representation of $N^+$ on $B_1$. By \cite[Theorem 8.18]{N},
the formulae
$$\Phi_A(g) := \Theta(g)\otimes \Psi_{1A}(g),~~\Phi_B(g) := \Theta(g)\otimes \Psi_{1B}(g)$$
for $g \in N_1$ define irreducible (linear) representations of $N_1$ of dimension
$ea$ and $eb$, (acting on $C_1$ and $D_1$, respectively),
and so we are done.

\smallskip
(b) Next we consider the case $N$ acts transitively on $\{A_1, \ldots ,B_t\}$. In this case,
$N_1^\circ := I_N(A_1)$ has index $2t$ in $N$ and is contained in $N_1$. Note that
there is some $g \in N$ such that $B_1^g \cong A_1$ as $Q$-modules. Certainly,
such $g$ must belong to $N_1$ and also $g$ interchanges
$C_1$ and $D_1$.  Applying the arguments of (a) to $g$, we see that $\Psi_1(g)$ interchanges $A_1$ and $B_1$. It follows that $(\Psi_1)_{N_1}$ is irreducible.
In turn, this implies by \cite[Theorem 8.18]{N} that $\Phi_{N_1}$ is irreducible, i.e.\
$N_1$ is irreducible on $V_1$. But $[N_1:N^\circ_1] = 2$ and
$V_1 = C_1 \oplus D_1$ as $N_1^\circ$-modules. Hence $C_1$ is an irreducible
$N_1^\circ$-module. Since $N^\circ_1 = I_N(A_1)$ and $C_1$ is the
$A_1$-isotypic component for $Q$ on $V$, we conclude by Clifford's theorem that
$N$ is irreducible on $V$.
\end{proof}

\begin{lem}\label{key3}
Let $V$ be a finite dimensional vector space over $k$ and $G \leq \GL(V)$ a
finite irreducible subgroup. Write $V|_{G^+} = e\sum^t_{i=1}W_i$, where the
$G^+$-modules $W_i$ are irreducible and pairwise non-isomorphic. Suppose
there is a subgroup $Q \leq \GP$ with the following properties:

{\rm (a)} $\{Q^g \mid g \in G\} = \{Q^x \mid x \in \GP\}$; and

{\rm (b)} $(W_i)_Q \cong A_{i} \oplus B_{i1} \oplus \ldots \oplus B_{is}$,
where $a := \dim A_{i} \neq \dim B_{il}$ for all $1 \leq i \leq t$ and all
$1 \leq l \leq s$,
the $Q$-modules $A_{i}$, $B_{il}$ are irreducible, and the $Q$-modules
$A_i$, $1 \leq i \leq t$, are pairwise non-isomorphic.\\
Then the following statements hold.

{\rm (i)} Denoting $N := \bfN_G(Q)$, we have that $V_N \cong A \oplus B$ as
$N$-modules, where $A$ is irreducible,
$A_Q \cong e(\oplus^t_{i=1}A_i)$ and $B_Q \cong e(\oplus_{i,l}B_{il})$.

{\rm (ii)} Assume that $N$ is a $p'$-subgroup, $\GP$ is perfect, and that, whenever $i \neq j$, no $\GP$-composition factor of $W_i^* \otimes W_j$ is trivial.
If all $\GP$-composition factors of $\End(V)/\cM$ (if any) are trivial, then in fact
$\cM = \End(V)$.
\end{lem}

\begin{proof}
(i) follows from same proof as of Lemma \ref{key2}.
For (ii), note that, since $\GP$ is perfect, it
must act trivially on $\cE/\End(V)$, i.e.\ $\cM \supseteq [\End(V),\GP]$. It follows that
\begin{equation}\label{cm1}
  \cM \supseteq [\cE_{1i},\GP]
\end{equation}
for $\cE_{1i} := \End(V_i)$. On the other hand,
$\Hom(V_i,V_j) = [\Hom(V_i,V_j),\GP]$, and so
$$\cM \supseteq \oplus_{1 \leq i \neq j \leq t}\Hom(V_i,V_j).$$
It suffices to show that $\cM \supseteq \cE_{11}$ (and so by symmetry $\cM \supseteq \cE_{1i}$ for all $i$).

Applying the Artin-Wedderburn theorem to $N$, we see that
\begin{equation}\label{cm2}
  \cM \supset \End(A) \supseteq \End(C_1)
\end{equation}
where $(C_1)_Q \cong eA_1$. Also,
as in the proof of Lemma \ref{key2}, we can write
$$V_1 = U \otimes W_1,~~C_1 = U \otimes A_1,$$
such that  $U$ affords a projective representation $\Theta$ of
$G_1/\GP \cong N_1/N^+$, $W_1$ affords a projective representation $\Psi_1$ of
$G_1$ that extends the representation $\Psi$ of $\GP$ on $W_1$, and
$\Phi(g) = \Theta(g) \otimes \Psi_1(g)$ for the representation $\Phi$ of $G_1$ on
$V_1$.

\smallskip
Note that the subspace $\End(W_1)^\circ$ consisting of all transformations with trace $0$ is a $\GP$-submodule $X$ of codimension $1$ of $\End(W_1)$. Next, as a
$\GP$-module,
$$\cE_{11} = \End(V_1) \cong \End(U) \otimes \End(W_1) \cong e^2\End(W_1).$$
So we see that $\cE_{11}^+ := \End(U) \otimes \End(W_1)^\circ$ is a submodule
of codimension $e^2$ in $\cE_{11}$, and all $\GP$-composition factors of
$\cE_{11}/\cE_{11}^+$ are trivial. Since $\GP$ is perfect, it follows that
$\cE_{11}^+ \supseteq [\cE_{11},\GP]$. But
$$\dim \Hom_{k\GP}(\cE_{11},k) = e^2\dim \Hom_{k\GP}(\End(W_1),k) =
   e^2\dim \Hom_{k\GP}(W_1,W_1) = e^2.$$
Hence, $\cE_{11}^+ = [\cE_{11},\GP]$, and so by (\ref{cm1}) we have that
$$\cM \supset \cE_{11}^+ = \End(U) \otimes \End(W_1)^\circ.$$
On the other hand, by (\ref{cm2}) we also have that
$$\cM \supset \End(C_1) = \End(U) \otimes \End(A_1).$$
Obviously, $\End(W_1)^\circ + \End(A_1) = \End(W_1)$ (as $\End(A_1)$ contains elements with nonzero trace).
Hence we conclude that $\cM \supseteq \cE_{11}$, as stated.
\end{proof}

We also record the following trivial observation:

\begin{lem}\label{triv}
Let $E$ be a $kG$-module of finite length with submodules $X$ and $M$. Suppose that
$N \leq G$ and that the $N$-modules $X$ and $E/X$ share no common composition
factor (up to isomorphism). Suppose that the multiplicity of each composition factor $C$
of $X$ is at most its multiplicity as a composition factor of $M$
(for instance, $X$ is a subquotient of $M$). Then $M \supseteq X$.
\end{lem}

\begin{proof}
The hypothesis implies that the $N$-modules $X$ and $E/M$ have no common composition factor. On the other hand,
$X/(M \cap X) \cong (X+M)/M \subseteq E/M$ as $N$-modules. It follows that
$X = M \cap X$, as stated.
\end{proof}

\begin{prop}\label{extra1}
Let $(G,V)$ be as in the extraspecial case (ii) of Theorem \ref{str}. Then $(G,V)$ is weakly adequate.
\end{prop}

\begin{proof}
Decompose $V_\GP = e\sum^t_{i=1}W_i$ as in Lemma \ref{key}. Recall by Theorem
\ref{str}(ii) that $R := \bfO_{p'}(\GP) \lhd G$ acts irreducibly on each $W_i$.
First we show that
if $i \neq j$ then the $R$-modules $W_i$ and $W_j$ are non-isomorphic. Assume the
contrary: $W_i \cong W_j$ as $R$-modules. Then the $\GP$-modules $W_i$ and
$W_j$ are two extensions to $\GP \rhd R$ of the $R$-module $W_i$. By
\cite[Corollary 8.20]{N}, $W_j \cong W_i \otimes U$ (as $\GP$-modules) for some
one-dimensional $\GP/R$-module $U$. But  $\GP/R$ is perfect by Theorem \ref{str}(ii).
It follows that $U$ is the trivial module and $W_i \cong W_j$ as $\GP$-modules,
a contradiction.

\smallskip
For future use, we also show that the $\GP$-module $W_i$ has a unique complex lift.
Indeed, the existence of a complex lift $\chi$ of $W_i$ was established in
\cite[Theorem B]{BZ}. Suppose that $\chi'$ is another complex lift. Then both
$\chi$ and $\chi'$ are extensions of $\alpha := \chi_R$, and $\alpha$ is irreducible
since $R$ is irreducible on $W_i$. Then again by \cite[Corollary 8.20]{N},
$\chi' = \chi\lambda$ for some linear character $\lambda$ of $\GP/R$, and so
$\lambda = 1_{\GP/R}$ as $\GP/R$ is perfect. Thus $\chi' = \chi$.

\smallskip
Now we write $\GP/R = S_1 \times \ldots \times S_n$ with $S_i \cong S$ as in Theorem
\ref{str}(ii). We will define the subgroup $Q > R$ of $\GP$ with
$$Q/R = Q_1 \times \ldots \times Q_n$$
as follows. If $p = 17$ and $S = \PSL_2(17)$, then
$Q_i$ is a dihedral subgroup of order $16$.  If $S = \Omega^-_{2a}(2^b)'$ with
$ab = n$ (and $a \geq 2$ as $S$ is simple non-abelian), then $Q_i$ is chosen
to be the first parabolic subgroup (which is the normalizer of an isotropic $1$-space in
the natural module $\bbF_{2^b}^{2a}$, of index $(2^n+1)(2^{n-b}-1)/(2^b-1)$).
If $S = \Sp_4(2)' \cong \AAA_6$, choose $Q_i \cong 3^2:4$ of order $36$.
If $S = \Sp_4(2^b)$ with $b \geq 2$, we fix a prime divisor $r$ of $b$ and choose
$Q_i \cong \Sp_4(2^{b/r})$.
For $S = \Sp_{2a}(2^b)$ with $a \geq 3$, we choose $Q_i$ to be the first parabolic
subgroup (which is the normalizer of a $1$-space in the natural module $\bbF_{2^b}^{2a}$,
of index $2^{2n}-1$). In all cases, our choice of $Q_i$ ensures that
the $p'$-subgroup $Q_i$ is a maximal
subgroup of $S_i$ and moreover the $S_i$-conjugacy class of $Q_i$ is
$\Aut(S_i)$-invariant. In particular, $\bfN_\GP(Q) = Q$. Also note that any $g \in G$
normalizes $R$ and permutes the simple factors $S_i$ of $\GP/R$; in fact, its action
on $\GP/R$ belongs to $\Aut(S^n) = \Aut(S) \wr \SSS_n$.
It follows that $Q$ satisfies the conditions (i), (iii) of Lemma
\ref{key}. Since $W_i \not\cong W_j$ as $R$-modules for $i \neq j$, $W_i \not\cong W_j$
as $Q$-modules as well. Hence we are done by Lemma \ref{key}.
\end{proof}

\begin{thm}\label{simple1}
Suppose $(G,V)$ is as in the case (i) of Theorem \ref{str}. Then $(G,V)$ is weakly
adequate unless one of the following possibilities occurs for the
group $H < \GL(W)$ induced by the action of $\GP$ on any irreducible
$\GP$-submodule $W$ of $V$.

\smallskip
{\rm (i)} $p = (q^n-1)/(q-1)$, $n \geq 3$ a prime, and $H \cong \PSL_n(q)$.

\smallskip
{\rm (ii)} $(p,H,\dim W) = (5,2\AAA_7,4)$,
$(7,6_1 \cdot \PSL_3(4),6)$, $(11,2M_{12},10)$, $(19,3J_3,18)$.
\end{thm}

\begin{proof}
(a) Arguing as in part (b) of the proof of Theorem \ref{str} (and using its notation),
we see that for each $i$ there is
some $k_i$ such that the kernel $K_i$ of the action of $\GP$ on $W_i$ contains
$\prod_{j \neq k_i}L_j$, and so $\GP$ acts on $W_i$ as $H_i = L_{k_i}/(L_{k_i} \cap K_i)$.
We aim to define a subgroup $Q > \bfZ(\GP)$ of $\GP$ such that
$$Q = Q_1 * Q_2 * \ldots  * Q_n$$
where $Q_i/\bfZ(L_i)  \leq L_i/\bfZ(L_i) =: S_i \cong S$ and $Q$ satisfies the conditions of Lemma \ref{key}. In fact, we will find $Q_i$ so that the $p'$-subgroup $Q_i/\bfZ(L_i)$ is a maximal subgroup of $S_i$ and moreover the $S_i$-conjugacy class of $Q_i/\bfZ(L_i)$ is
$\Aut(S_i)$-invariant.
To this end, we first find $Q_1$, then for each $i > 1$ we can fix an element $g_i \in G$ conjugating $S_1$ to $S_i$ and choose $Q_i = Q_1^{g_i}$.
Since $G$ fixes $\GP$ and $\bfZ(\GP)$ and induces a subgroup of
$\Aut(S) \wr \SSS_n$ while acting on $\GP/\bfZ(\GP) \cong S^n$, it follows that
$Q$ satisfies the conditions (i), (iii) of Lemma \ref{key}. Moreover, in the cases where
\begin{equation}\label{direct}
  \GP = L_1 \times \ldots \times L_n \cong H^n,
\end{equation}
then we can also
write $Q = Q_1 \times \ldots \times Q_n$, which simplifies some parts of the arguments.

\smallskip
(b1) Suppose first that we are in the case (b1) of Theorem \ref{bz}. Assume that
$(H,p) = (\Sp_{2n}(q),(q^n+1)/2)$. Here $H$ is the full cover of $S$, so
(\ref{direct}) holds. Then we choose $Q_i$ to be
the last parabolic subgroup of $\Sp_{2n}(q)$
(which is the stabilizer of a maximal totally isotropic subspace in the natural module
$\bbF_q^{2n}$). Then $Q_i/\bfZ(L_i)$ is a maximal $p'$-subgroup of $S_i$ and moreover
the $S_i$-conjugacy class of $Q_i/\bfZ(L_i)$ is $\Aut(S_i)$-invariant.
By \cite[Theorem 2.1]{GMST}, the $H$-module $W$ is one of the two Weil modules of dimension $(q^n-1)/2$ of $H \cong \Sp_{2n}(q)$. Furthermore, by
\cite[Lemma 7.2]{GMST}, the restrictions of these two Weil modules
of $L_i$ to $Q_i$ are irreducible and non-isomorphic. It follows
that, if $W_i \not\cong W_j$ as $\GP$-modules and $K_i = K_j$, then
$W_i \not\cong W_j$ as $Q$-modules. On the other hand, if $K_i \neq K_j$, then
$k_i \neq k_j$ (otherwise we would have $K_i = K_j = \prod_{a \neq k_i}L_a$ since $L_{k_i}$ acts faithfully on $V_i$), whence $K_i \cap Q \neq K_j \cap Q$
and so $W_i \not\cong W_j$ as $Q$-modules. Thus condition (ii) of Lemma \ref{key}
holds as well, and so we are done.

Consider the case $(H,p) = (2Ru,29)$. Then $H$ is the full cover of $S$ and so
(\ref{direct}) holds. Choose $Q_i$ to be a unique (up to $L_i$-conjugacy) maximal
subgroup of type $(2 \times \PSU_3(5)):2$ of $L_i$, cf.\ \cite{Atlas}. Note
that $L_i$ has a unique conjugacy class $3A$ of elements of order $3$. By using
\cite{JLPW} and \cite{Atlas}, and comparing the character values at this class $3A$, we see
that $L_i$ has two irreducible $p$-Brauer characters $\varphi_{1,2}$, of
degree $28$, and their restrictions to $Q_i$ yield the same irreducible character of
$Q_i$. Now, if $K_i \neq K_j$, then $k_i \neq k_j$ (as $W$ is a faithful $kH$-module), whence $K_i \cap Q \neq K_j \cap Q$ and so $W_i \not\cong W_j$ as $Q$-modules.
Suppose that $K_i = K_j$. By Clifford's theorem, there is some $g \in G$ such that
$W_j = W_i^g$ as $\GP$-modules, and so as $L_i$-modules as well. In this case, $g$
induces an automorphism of $L_i = 2Ru$. But all automorphisms of $Ru$ are inner
(see \cite{Atlas}), so $W_i$ and $W_j$ afford the same Brauer $L_i$-character, whence
$W_i \cong W_j$ as $\GP$-modules. Thus condition (ii) of Lemma \ref{key}
holds as well, and so we are done.

Next assume that $(H,p) = (\SU_{n}(q),(q^n+1)/(q+1))$; in particular
$n \geq 3$ is odd.  Since $H$ is simple, (\ref{direct}) holds. Then we choose $Q_i$ to be 
the last parabolic subgroup of $\SU_{n}(q)$
(which is the stabilizer of a maximal totally isotropic subspace in the natural module
$\bbF_{q^2}^{n}$). Then the $p'$-subgroup $Q_i$ is a maximal subgroup of $S_i$ and moreover the $S_i$-conjugacy class of $Q_i$ is $\Aut(S_i)$-invariant.
Next, if $n \geq 5$ then by \cite[Theorem 2.7]{GMST}, $\PSU_n(q)$ has a unique irreducible module over
$k$ of dimension $p-1 = (q^n-q)/(q+1)$, which is again a Weil module. Furthermore,
Lemmas 12.5 and 12.6 of \cite{GMST} show that the restriction of this Weil module
of $L_i$ to $Q_i$ is irreducible. The same conclusions hold in the case $n=3$ by
Theorem 4.2 and the proof of Remark 3.3 of \cite{Geck}.
It follows that, if $W_i \not\cong W_j$ as $\GP$-modules,
then $K_i \neq K_j$, $k_i \neq k_j$ (as $W$ is a faithful $kH$-module), whence
$K_i \cap Q \neq K_j \cap Q$ and so $W_i \not\cong W_j$ as $Q$-modules.
Thus condition (ii) of Lemma \ref{key} holds, and so we are done again.

Note that we have listed the cases of $(p,H) = (5,2\AAA_7)$ and $(19,3J_3)$ as possible exceptions in (ii).

\smallskip
(b2) Suppose now that we are in the case (b2) of Theorem \ref{bz}; in particular,
$p = 7$ and $\dim W = 6$.  Assume first that $S = \AAA_7$. The arguments in
the cases $L_i \cong 3\AAA_7$ and $6\AAA_7$ are the same, so we assume
$L_i \cong 6\AAA_7$. Then we choose $Q_i/\bfZ(L_i)$ to be a unique (up to $L_i$-conjugacy) maximal subgroup of type $\AAA_6$. Restricting the faithful reducible complex
characters of degree $4$ of $2\AAA_7$ and $6$ of $3\AAA_7$ \cite{Atlas} to $Q_i$
(and comparing character values at elements of order $3$), we see that
$Q_i \cong 6\AAA_6$. Now, using \cite{JLPW} one can check that $L_i$ has six
irreducible $p$-Brauer characters of degree $6$, and their restrictions to $Q_i$ are irreducible and distinct. Now we can argue as in the case of $\Sp_{2n}(q)$.

Assume now that $H = 2J_2$, and so (\ref{direct}) holds. Choose
$Q_i/\bfZ(L_i)$ to be a unique (up to $L_i$-conjugacy) maximal subgroup of type
$3 \cdot \PGL_2(9)$ (see \cite{Atlas}). Also, using \cite{JLPW} one can check that $L_i$
has two irreducible $p$-Brauer characters of degree $6$, and their restrictions to $Q_i$ are irreducible and distinct. Now we can argue as in the case of $\Sp_{2n}(q)$.

Suppose that $H = 6_1 \cdot \PSU_4(3)$.
We will prove weak adequacy of $(G,V)$ in two steps. First, we choose $M_i/\bfZ(L_i)$ to be a unique (up to $S_i$-conjugacy) maximal subgroup of type $T \cong \SU_3(3)$ of
$S_i$ (see \cite{Atlas}). Since $T$ has trivial Schur multiplier, we have
that $M_i \cong Z_i \times T$, where $Z_i := \bfZ(L_i)$. According to \cite{JLPW}, $L_i$ has two irreducible $p$-Brauer characters of degree $6$ which have different
central characters. It follows that their restrictions to $M_i$ are irreducible and distinct. Setting
$$M := M_1 * \ldots * M_n,$$
we conclude by Lemma \ref{key} that $N := \bfN_G(M)$ is irreducible on $V$;
furthermore, $N/M \cong G/\GP$ is a $p'$-group.
But note that $M$ is {\it not} a $p'$-group.  Now, at the second step, we note
that $M \lhd N$ and $N^+ := \bfO^{p'}(N) = \bfO^{p'}(M) \cong T^n$, and moreover
each irreducible $N^+$-submodule in $V$ has dimension $6$.  Also, recall that
$T = \SU_3(3)$ and $p = 7$. So we are done by applying the result of the
case of $\PSU_n(q)$.

\smallskip
(b3) Consider the case (b3) of Theorem \ref{bz}; in particular, $p= 11$ and $\dim W = 10$.
Putting the possibility $H = 2M_{12}$ as a possible exception in (ii), we may
assume that $H = M_{11}$ or $2M_{22}$. Then we choose
$Q_i/\bfZ(L_i)$ to be a unique (up to $S_i$-conjugacy) maximal subgroup of type
$M_{10} \cong \AAA_6 \cdot 2_3$, respectively $\PSL_3(4)$ of $S_i$ (see \cite{Atlas}).
In the former case, $H$ is simple and so (\ref{direct}) holds. In the latter case,
since $H_j \cong 2M_{22}$, we see that the cyclic group $\bfZ(L_i) \lhd \GP$ must act as
a central subgroup of order $1$ or $2$ of $H_j$ on each $W_j$. Hence the faithfulness
of $G$ on $V$ implies that $L_i \cong 2M_{22}$.
Since $\PSL_3(4)$ has no nontrivial representation of degree $10$, we must have
that $Q_i \cong 2 \cdot \PSL_3(4)$ is quasisimple in this case. Now,  using \cite{JLPW} one can check that $L_i$ has two irreducible $p$-Brauer characters of degree $10$, and their restrictions to $Q_i$ are irreducible and distinct. Hence we can argue as in the case of $\Sp_{2n}(q)$.

\smallskip
(b4) Suppose we are in the case (b4) of Theorem \ref{bz}; in particular, $p= 13$ and
$\dim W = 12$. Since $H$ is the full cover of $S$, (\ref{direct}) holds. Then we may choose $Q_i/\bfZ(L_i)$ to be a unique (up to $S_i$-conjugacy) maximal subgroup of type
$J_2:2$, respectively $\SL_3(4):2_3$ of $S_i$ (see \cite{Atlas}).
Since $J_2$ has no nontrivial representation of degree $12$, in the former case we must have that $Q_i \cong (C_3 \times 2J_2)\cdot C_2$, where $C_3 = \bfO_3(\bfZ(L_i))$ and
the $C_2$ induces an outer automorphism of $J_2$. Also, according to \cite{ModAt},
$L_i$ has precisely two irreducible $p$-Brauer characters of degree $12$ which differ at
the central elements of order $3$. Using \cite{JLPW} we can now check that
the restrictions of these two characters to $Q_i$ are irreducible and distinct, and then
finish as in the case of $\Sp_{2n}(q)$.
In the latter case of $L_i = 2\gtwo(4)$, since $\SL_3(4)$ has no nontrivial representation of degree $12$, we must have that $Q_i \cong (6 \cdot \PSL_3(4)) \cdot 2_3$. Now, using \cite{JLPW} one can check that $L_i$ has a unique irreducible $p$-Brauer character of degree $12$, and its restriction to $Q_i$ is irreducible. Hence we can argue as in the case of $\PSU_{n}(q)$.

\smallskip
(c) Now we consider case (c) of Theorem \ref{bz}; in particular, $\dim W = p-2$.
Assume that $H= \AAA_p$ with $p \geq 5$. Since $H$ is simple, (\ref{direct}) holds. Choosing
$Q_i \cong \AAA_{p-1}$, we see that the $p'$-subgroup $Q_i$ is a maximal subgroup of
$S_i$ and that the $S_i$-conjugacy class of $Q_i$ is $\Aut(S_i)$-invariant. Also,
using \cite[Lemma 6.1]{GT2} for $p \geq 17$ and \cite{JLPW} for $p \leq 13$, we see that
$H$ has a unique irreducible $kH$-module of dimension $p-2$, and the restriction of
this module to $\AAA_{p-1}$ is irreducible. Now we can argue as in the case of $\PSU_n(q)$.

Next suppose that $(H,p) = (\SL_2(q),q+1)$; in particular, $p$ is a Fermat prime and
$H$ is simple, and so (\ref{direct}) holds. Choosing
$Q_i < \SL_2(q)$ to be a Borel subgroup (of index $p$), we see that $Q_i$ is a maximal
$p'$-subgroup of  $S_i$ and that the $S_i$-conjugacy class of $Q_i$ is $\Aut(S_i)$-invariant. Also, using \cite{Burk} one can check  that
$H$ has a unique irreducible $kH$-module of dimension $p-2$, and the restriction of
this module to $Q_i$ is irreducible. Now argue as above.

Suppose that $p = 5$ and $H = 3\AAA_6$ or $3\AAA_7$.
First we note that $L_i \cong 3\AAA_s$ with $s = 6$, respectively $s=7$. If not, then
$L_i \cong 6\AAA_s$, but then, since
$H_j \cong 3\AAA_s$, $\bfO_2(\bfZ(L_i))$ must act trivially on all $W_i$, contradicting
the faithfulness of $G$ on $V$. Now we choose $Q_i$
to be the normalizer of a Sylow $3$-subgroup in $L_i$, of order $108$. It is straightforward to
check that $\bfN_{S_i}(Q_i/\bfZ(L_i)) = Q_i/\bfZ(L_i)$ and that the $S_i$-conjugacy class of $Q_i$ is $\Aut(S_i)$-invariant. Also, using \cite{JLPW} one can check that
$H$ has two irreducible $5$-Brauer characters of degree $p-2$, and the restrictions
of them to $Q_i$ are irreducible and distinct. Now we can argue as in the case of
$\Sp_{2n}(q)$.

Suppose that $(p,H) = (11,M_{11})$ or $(23,M_{23})$. Again (\ref{direct}) holds as $H$ is simple. Choosing $Q_i$ to be
$M_{10} \cong \AAA_6 \cdot 2_3$ (in the notation of \cite{Atlas}),
respectively $M_{22}$, we have that $Q_i$ is
a unique maximal subgroup of $L_i$ of the given $p'$-order up to $L_i$-conjugacy. Furthermore, $L_i$ has a unique irreducible $kH$-module of dimension $p-2$, and the restriction of
this module to $Q_i$ is irreducible. Now argue as in the case of $\PSU_n(q)$.

\smallskip
(d) Finally, we consider case (d) of Theorem \ref{bz}: $(p,H) = (11,J_1)$ or
$(7,2\AAA_7)$. Then we choose $Q_i/\bfZ(L_i)$ to be a unique (up to $S_i$-conjugacy) maximal subgroup of type $2^3:7:3$, respectively $\AAA_6$ (cf.\ \cite{Atlas}).
In the former case, $H$ is simple, and so (\ref{direct}) holds. In the latter case, note that
$L_i$ is $2\AAA_7$. If not, then $L_i \cong 6\AAA_7$, but then, since
$H_j \cong 2\AAA_7$, $\bfO_3(\bfZ(L_i))$ must act trivially on all $W_i$, contradicting
the faithfulness of $G$ on $V$. It then follows that $Q_i \cong 2\AAA_6$
(as any $4$-dimensional $k\AAA_6$-representation is trivial).
Now, using \cite{JLPW} one can check that $H$ has a unique irreducible $p$-Brauer character of given degree, and its restriction to $Q_i$ is irreducible. Now we can argue as in the case of $\PSU_{n}(q)$.
\end{proof}

Next we use Lemma \ref{key2} to handle three exceptions listed in Theorem \ref{simple1}:

\begin{prop}\label{j3}
In the case $(p,H,\dim W) = (19,3J_3,18)$ of Theorem \ref{simple1}(ii), $(G,V)$ is
weakly adequate.
\end{prop}

\begin{proof}
Since $H$ is the full cover of $S$, we have
$\GP = L_1 \times \ldots \times L_n \cong H^n$. Since $H$ acts
faithfully on $W$, for each $i$ there is some $k_i$ such that the kernel $K_i$ of the action of $\GP$ on $W_i$ is precisely $\prod_{j \neq k_i}L_j$. We define a subgroup $Q$ of
$\GP$ such that
$$Q = Q_1 \times \ldots \times Q_n$$
where $Q_i/\bfZ(L_i) \cong \SL_2(16):2$ is a maximal subgroup of
$S_i = L_i/\bfZ(L_i) \cong J_3$. Since $\SL_2(16)$ has a trivial Schur multiplier and
$\bfZ(L_i) \leq \bfZ(Q_i)$, we have that
$Q_i \cong 3 \times (\SL_2(16):2)$.  Furthermore, the $S_i$-conjugacy class
of $Q_i$ is $\Aut(S_i)$-invariant. Hence $Q$ satisfies the condition
(i) of Lemma \ref{key2}.

\smallskip
Using \cite{GAP}, one can check that
$L_i$ has exactly four irreducible $19$-Brauer characters $\varphi_{1,2,3,4}$ of degree $18$, and $(\varphi_j)_{Q_i} = \alpha_j + \beta_j$, with
$\alpha_j$ of degree $1$ with kernel $[Q_i,Q_i]$, $\beta_j$ of degree $17$, and
$\beta_{1,2,3,4}$
are all distinct.  Now we show that $Q$ fulfills the condition (ii) of Lemma \ref{key2}.
Suppose that $W_i \not\cong W_j$ as $\GP$-modules. Then $Q$ acts coprimely
on $W_i$, with character $\tilde\alpha_i + \tilde\beta_i$, where
$\tilde\alpha_i$ has degree $1$ and $\tilde\beta_i$ has degree $17$. If $k_i \neq k_j$,
then $\tilde\alpha_i$ and $\tilde\alpha_j$ have different kernels and so are distinct,
and likewise $\tilde\beta_i$ and $\tilde\beta_j$ are distinct. Suppose now
that $k_i = k_j$. Then, because of the condition $W_i \not\cong W_j$,
we may assume that $W_i$ and $W_j$ both have kernel
$K := L_2 \times \ldots \times L_n$,
and afford $L_1$-characters $\varphi_k$ and
$\varphi_l$ with $1 \leq k \neq l \leq 4$.  Since the $G$-module $V$ is irreducible,
we have $W_i \not\cong W_j \cong W_i^g$ for some $g \in G$ which stabilizes
$K$ and $\GP/K \cong L_1$ but does not induce an inner automorphism of $L_1$.
The latter condition implies, cf.\ \cite{Atlas}, that $g$ interchanges the two
classes of elements of order $5$ and inverts the central element of order $3$ of
$L_1$. The same is true for $Q_1$. It follows that
$\alpha_k \neq \alpha_l$, $\beta_k \neq \beta_l$, and so
$$\tilde\alpha_i \neq \tilde\alpha_j, ~~\tilde\beta_i \neq \tilde\beta_j,$$
as claimed.

\smallskip
By Lemma \ref{key2}, $V \cong A \oplus B$ as a module over the
$p'$-group $N := \bfN_G(Q)$,
where the $N$-modules $A$ and $B$ are irreducible of dimension $e$ and
$17e$, respectively. Hence, by the Artin-Wedderburn theorem applied to $N$,
$$\cM := \langle \Phi(g) \mid g \in G, g \mbox{ semisimple } \rangle_k$$
contains $\cA:=\End(A) \oplus \End(B) = (A^* \otimes A) \oplus (B^* \otimes B)$
(if $\Phi$ denotes the representation of $G$ on $V$).
As in Lemma \ref{key2} and its proof, write
$A = \oplus^t_{i=1}C_i = e(\oplus^t_{i=1}A_i)$ and
$B = \oplus^t_{i=1}D_i = e(\oplus^t_{i=1}B_i)$ as $Q$-modules, where
$A_i$ affords $\tilde\alpha_i$ and $B_i$ affords $\tilde\beta_i$. Hence, the complement
to $\cA$ in $\End(V)$ affords the $Q$-character
$$\Delta := e^2\sum^t_{i,j = 1}(\tilde\alpha_i\overline{\tilde\beta_j} +
    \tilde\beta_i\overline{\tilde\alpha_j}).$$
In particular, all irreducible constituents of $\Delta_{[Q,Q]}$ are of degree $17$.
The same must be true for the quotient $\End(V)/\cM$.

As a $\GP$-module,
$$\End(V) = \oplus^t_{i,j=1}(V_i^* \otimes V_j) \cong
    e^2(\oplus^t_{i,j=1} W_i^* \otimes W_j).$$
Observe that the $\GP$-module $W_i^* \otimes W_j$ is irreducible of
dimension $324$ if $k_i \neq k_j$. Assume that $k_i = k_j$, say $k_i = k_j = 1$.
Using \cite{GAP} one can check that no irreducible constituent of
$\varphi_k\overline{\varphi_l}$ for $1 \leq k,l \leq 4$ can consist of only
irreducible characters of degree $17$ when restricted to the subgroup
$\SL_2(16)$ of $L_1 = 3J_3$. It follows that no irreducible constituent of
the $\GP$-module $\End(V)$ can consist of only irreducible constituents of
dimension $17$ when restricted to $[Q,Q]$. Hence $\cM = \End(V)$.
\end{proof}

\begin{prop}\label{m12}
In the case $(p,H,\dim W) = (11,2M_{12},10)$ of Theorem \ref{simple1}(ii), $(G,V)$ is
weakly adequate.
\end{prop}

\begin{proof}
As $H$ is the full cover of $S$, we have that
$\GP = L_1 \times \ldots \times L_n \cong H^n$. Since $H$ acts
faithfully on $W$, for each $i$ there is some $k_i$ such that the kernel $K_i$ of the action of $\GP$ on $W_i$ is precisely $\prod_{j \neq k_i}L_j$. We define a subgroup $Q$ of
$\GP$ such that
$$Q = Q_1 \times \ldots \times Q_n$$
where $Q_i/\bfZ(L_i) \cong 2^{1+4}_{+}\cdot \SSS_3$ is a maximal subgroup of
$S_i = L_i/\bfZ(L_i) \cong M_{12}$. Note that the $S_i$-conjugacy class
of $Q_i$ is $\Aut(S_i)$-invariant. Hence $Q$ satisfies the condition
(i) of Lemma \ref{key2}.

\smallskip
Using \cite{GAP}, one can check that
$L_i$ has exactly two irreducible $11$-Brauer characters $\varphi_{1,2}$ of degree $10$, and $(\varphi_j)_{Q_i} = \alpha+ \beta_j$, with
$\alpha$ of degree $4$, $\beta_j$ of degree $6$,
$\beta_1 \neq \beta_2$. Furthermore, $Z_i := \bfZ(Q_i) \cong C_2^2$, and
\begin{equation}\label{m121}
  \alpha_{Z_i} = 4\lambda,~~(\beta_j)_{Z_i} = 6\mu,
\end{equation}
where $\lambda$ and $\mu$ are the two linear characters of $Z_i$ that are faithful on
$\bfZ(L_i) < Z_i$. In particular,
\begin{equation}\label{m122}
  (\alpha\beta_j)_{Z_i} = 24\nu
\end{equation}
with $\nu := \lambda\mu \neq 1_{Z_i}$.

Now we show that $Q$ fulfills the condition (ii) of Lemma \ref{key2}.
Suppose that $W_i \not\cong W_j$ as $\GP$-modules. Then $Q$ acts
on $W_i$, with character $\tilde\alpha_i + \tilde\beta_i$, where
$\tilde\alpha_i(1) = 4$ and $\tilde\beta_i(1) = 6$. If $k_i \neq k_j$,
then $\tilde\alpha_i$ and $\tilde\alpha_j$ have different kernels and so are distinct,
and likewise $\tilde\beta_i$ and $\tilde\beta_j$ are distinct. In particular,
in this case $W_i^* \otimes W_j$ is also irreducible. Suppose now
that $k_i = k_j$. Then,
we may assume that $W_i$ and $W_j$ both have kernel
$K := L_2 \times \ldots \times L_n$,
and afford $L_1$-characters $\varphi_k$ and
$\varphi_l$ with $1 \leq k,l \leq 2$.
Since the $G$-module $V$ is irreducible,
we have $W_j \cong W_i^g$ for some $g \in G$ which stabilizes
$K$ and $\GP/K \cong L_1$. But $\varphi_k$ is $\Aut(L_1)$-invariant, cf.\ \cite{JLPW},
whence $l =k$, i.e.\ $W_j \cong W_i$, a contradiction.

\smallskip
By Lemma \ref{key2}, $V \cong A \oplus B$ as a module over the $p'$-group
$N := \bfN_G(Q)$,
where the $N$-modules $A$ and $B$ are irreducible of dimension $4e$ and
$6e$, respectively. Hence, by the Artin-Wedderburn theorem applied to $N$,
$$\cM := \langle \Phi(g) \mid g \in G, g \mbox{ semisimple } \rangle_k$$
contains $\cA:=\End(A) \oplus \End(B) = (A^* \otimes A) \oplus (B^* \otimes B)$
(if $\Phi$ denotes the representation of $G$ on $V$).
As in Lemma \ref{key2} and its proof, write
$A = \oplus^t_{i=1}C_i = e(\oplus^t_{i=1}A_i)$ and
$B = \oplus^t_{i=1}D_i = e(\oplus^t_{i=1}B_i)$ as $Q$-modules, where
$A_i$ affords $\tilde\alpha_i$ and $B_i$ affords $\tilde\beta_i$. Hence, the complement
to $\cA$ in $\End(V)$ affords the $Q$-character
$$\Delta := e^2\sum^t_{i,j = 1}(\tilde\alpha_i\overline{\tilde\beta_j} +
    \tilde\beta_i\overline{\tilde\alpha_j}).$$
Together with (\ref{m121}) and (\ref{m122}), this implies that the restriction of any irreducible constituents of $\Delta$ to $\bfZ(Q) = Z_1 \times \ldots \times Z_n$
does {\it not} contain $1_{\bfZ(Q)}$. Thus $\bfZ(Q)$ acts fixed point freely on the quotient $\End(V)/\cM$. Furthermore, the $Q$-character of this quotient does not contain
$\overline{\tilde\beta_i}\tilde\beta_j$ (as an irreducible constituent of degree $36$) for any $i \neq j$.

As a $\GP$-module,
$$\End(V) = \oplus^t_{i,j=1}(V_i^* \otimes V_j) \cong
    e^2(\oplus^t_{i,j=1} W_i^* \otimes W_j).$$
Now, if $i \neq j$ then the $\GP$-module $W_i^* \otimes W_j$ is irreducible
and its Brauer character, when restricted to $Q$, contains $\tilde\beta_i\tilde\beta_j$.
On the other hand, the Brauer character of $W_i^* \otimes W_i$ is the direct
sum of $1_{\GP}$ and another irreducible character of degree $99$
(as one can check using \cite{GAP}), whose
restriction to $\bfZ(Q)$ contains $1_{\bfZ(Q)}$ (which can be seen from (\ref{m121})).
Hence we conclude that $\cM = \End(V)$.
\end{proof}

\begin{lem}\label{2a7}
Let $\Char(k) = 5$ and let $W$ be a faithful irreducible $k(2\SSS_7)$-module of
dimension $8$, with corresponding representation $\Theta$. Decompose
$W_L = W_1 \oplus W_2$ as $L$-modules for $L = 2\AAA_7$. Then there is
a $5'$-element $z \in 2\SSS_7 \setminus L$ and a set $\cX \subset L$ such that

{\rm (i)} $x$ and $xz$ are $5'$-elements for all $x \in \cX$, and

{\rm (ii)} $\langle \Theta(x) \mid x \in \cX \rangle_k = \End(W_1) \oplus \End(W_2)$.
\end{lem}

\begin{proof}
Using \cite{AGR} and \cite{GAP}, K.\ Lux verified that one can find an element
$h \in 2\SSS_7 \setminus L$ (of order $12$) and a set $\cX \subset L$ satisfying
the condition (i) and such that $\langle \Theta(xz) \mid x \in \cX \rangle_k$ has dimension
$32$. Since $\Theta(z) \in \GL(W)$, it follows that
$\langle\Theta(x) \mid x \in \cX \rangle_k$ is a subspace of
dimension $32$ in $\End(W_1) \oplus \End(W_2)$. Since the latter also has dimension
$32$, we are done.
\end{proof}

\begin{prop}\label{a7}
In the case $(p,H,\dim W) = (5,2\AAA_{7},4)$ of Theorem \ref{simple1}(ii), $(G,V)$ is
weakly adequate.
\end{prop}

\begin{proof}
(a) Recall that $\GP = L_1 * \ldots * L_n$, and for each $i$ there is some $k_i$ such that the kernel $K_i$ of $\GP$ contains $\prod_{j \neq k_i}L_j$.  Relabeling the $W_i$
we may assume that $k_1 = 1$. Now, $L_1$ acts on each
$W_j$ either trivially or as the group $H_j \cong 2\AAA_7$. It follows that
$\bfO_3(\bfZ(L_1))$ acts trivially on each $W_j$ and so by faithfulness
$\bfO_3(\bfZ(L_1)) = 1$, yielding $L_1 \cong 2\AAA_7$.  On the other hand,
$L_1/(K_1 \cap L_1) = H_1 \cong 2\AAA_7$, whence $K_1 \cap L_1 = 1$,
$K_1 = \prod_{j \neq 1}L_j$. This is true for all $i$, so we have shown that
$$\GP = L_1 \times L_2 \times \ldots \times L_n \cong H^n.$$
Certainly, $G$ permutes the $n$ components $L_i$, and this action is transitive
by Theorem \ref{str}(i). Denoting $J_1 := \bfN_G(L_1)$, one sees that
$G_1 = I_G(W_1) = \Stab_G(V_1)$
is contained in $J_1$ (as it fixes $K_1 = \prod_{j>1}L_j$).
Fix a decomposition $G = \cup^t_{i=1}g_iJ_1$ with $g_1 = 1$ and
$L_i = L_1^{g_i} = g_i L_1 g_i^{-1}$, and
choose a subgroup $Q_1 < L_1$ such that $Q_1/\bfZ(L_1) \cong \PSL_2(7)$. Since
involutions in $\AAA_7$ lift to elements of order $4$ in $L_1$, we see that
$Q_1 \cong \SL_2(7)$. Now we define
$$Q = Q_1 \times  Q_1^{g_2} \times \ldots \times Q_1^{g_n} < \GP.$$
Note that $\bfN_\GP(Q) = Q$ and so $N := \bfN_G(Q)$ is a $p'$-group. Also,
$L_1$ has exactly two irreducible $5$-Brauer characters $\varphi_{1,2}$
of degree $4$, restricting irreducibly and differently to $Q_1$.

\smallskip
(b) Consider the case where $k_i \neq k_j$ whenever $i \neq j$, i.e.\ $J_1 = G_1$ and
$t = n$. We claim that $Q$ satisfies the conditions of Lemma \ref{key}. Indeed,
the condition $k_i \neq k_j$ implies that the $Q$-modules $W_i$ and $W_j$ are
irreducible and non-isomorphic for $i \neq j$. Next, for any $x \in J_1$, since
$x$ fixes $W_1$ (up to isomorphism), $x$ fixes the character $\varphi$ of the
$L_1$-module $W_1$ and so $x$ cannot fuse the two classes $7A$ and $7B$ of
elements of order $7$ in $L_1$, whence $x$ can induce only an inner automorphism of $L_1$. It follows that $Q_1^x = Q_1^t$ for some $t \in L_1$. Now we consider any
$g \in G$.  Then, for each $i$ we can find $j$ and $x_i \in J_1$ such that
$gg_i = g_jx_i$. By the previous observation, there is some $t_i \in L_1$ such that
$Q_1^{x_i} = Q_1^{t_i}$. Hence, setting $y_i = g_jt_ig_j^{-1} \in L_j$, we have that
$$Q_1^{gg_i} = Q_1^{g_jx_i} = g_jx_iQ_1x_i^{-1}g_j^{-1} =
   g_jt_iQ_1t_i^{-1}g_j^{-1} = y_ig_jQ_1g_j^{-1}y_i^{-1} = (Q_1^{g_j})^{y_i}.$$
It follows that $Q^g = Q^y$ with $y = \prod_iy_i \in \GP$, i.e.\ $Q$ fulfills the condition
(i) of Lemma \ref{key}. Now we can conclude by Lemma \ref{key} that $N$ is
irreducible on $V$ and so we are done.

\smallskip
(c) From now on we assume that, say, $k_1 = k_2$. Then $W_1$ and $W_2$
are non-isomorphic modules over $\GP/K_1 = L_1$. So we may assume that
$W_i$ affords the $L_1$-character $\varphi_i$ for $i = 1,2$. Note that any $x \in J_1$
sends $W_1$ to another irreducible $\GP$-module with the same kernel $K_1$, and
so $\varphi_1^x \in \{\varphi_1,\varphi_2\}$. The irreducibility of $G$ on $V$ implies
by Clifford's theorem that the induced action of $J_1$ on $\{\varphi_1,\varphi_2\}$ is
transitive, with kernel $G_1$. We have shown that $[J_1:G_1] = 2$ and $t = 2n$.
We will label $g_i(W_1)$ as $W_{2i-1}$ and $g_i(W_2)$ as $W_{2i}$.
We also have that $W_2 \cong W_1^h$ for all $h \in J_1 \setminus G_1$. Comparing
the kernels and the characters of $Q$ on $W_i$, we see that the $Q$-modules
$W_i$ are all irreducible and pairwise non-isomorphic. Let
$$\cE_1 := \oplus^t_{i=1}\End(V_i) = \oplus^n_{i=1}\cA_i,~~
    \cA_i := \End(V_{2i-1}) \oplus \End(V_{2i}),$$
$$\cE_{21} := \oplus^{n}_{i=1}\cB_i,~~
     \cB_i := (\Hom(V_{2i-1},V_{2i}) \oplus \Hom(V_{2i},V_{2i-1})),$$
$$\cE_{22} := \oplus_{1 \leq i \neq j \leq 2n,~ \{i,j\} \neq \{2a-1,2a\}}\Hom(V_{i},V_{j})$$
so that $\End(V) = \cE_1 \oplus \cE_{21} \oplus \cE_{22}$.
Note that the $\GP$-composition factors of $\cE_{21}$ are all of dimension $6$ and
$10$, whereas the $\GP$-composition factors
of $\cE_1$ are either trivial or of dimension $15$, as one can check using \cite{JLPW}. Furthermore, the $\GP$-composition factors of $\cE_{22}$ are all of dimension $16$.
In particular, no $\GP$-composition
factor of $\Hom(W_i,W_j)$ is trivial when $i \neq j$. Similarly, whenever $i \neq j$,
the only common $\GP$-composition factor shared by $\cA_i$ and $\cA_j$ is $k$,
and $\cB_i$ and $\cB_j$ share no common $\GP$-composition factor.

\smallskip
(d) Here we show that $\cA_i \oplus \cB_i$ is a subquotient of $\cM$.
To this end, note
that $J_1$ acts irreducibly on $V_1 \oplus V_2$. There is no loss to replace $G$ by the
image of $J_1$ in $\End(V_1 \oplus V_2)$ and
$V$ by $V_1 \oplus V_2$. In doing so, we also get that $n = 1$, $\GP = L_1$,
$[G:G_1] = 2$, $K_1 = 1$,  $G_1= C * L_1$, where $C := \bfC_G(L_1)$ is a $5'$-group.
So we can write $V_i = U_i \otimes W_i$ as $G_1$-modules, where
$U_i$ is an irreducible $kC$-module with corresponding representation $\Lambda_i$,
for $i = 1,2$. Hence for the representation $\Phi_i$ of $G_1$ on $V_i$ we have
$\Phi_i = \Lambda_i \otimes \Theta_i$, where $\Theta_i$ is the representation of
$L_1$ on $W_i$. Finally, for the representation $\Phi$ of $G$ on $V = V_1 \oplus V_2$
we have $\Phi(g) = \diag(\Phi_1(g),\Phi_2(g))$ whenever $g \in G_1$.

Recall the element $z \in 2\SSS_7$ and the set $\cX \subset L_1$ constructed in
Lemma \ref{2a7}. Now we fix a $5'$-element $h \in G \setminus G_1$ such that
$h$ induces the same action on $L_1/\bfZ(L_1) \cong \AAA_7$ as the action of
$z$ on $\AAA_7$. It follows that, for all elements $x \in \cX$ and for all
$u \in C$, $ux$ and $uxh$ are $5'$-elements, whence $\cM$ contains the subspaces
$$\cC:= \langle \Phi(ux) \mid u \in C, x \in \cX\rangle_k,~~
    \cC\Phi(h) := \{v\Phi(h) \mid v \in \cC\}.$$
We also have that
$\Theta_2 \cong \Theta_1^h = \Theta_1^z$. Setting
$\Theta(x) = \diag(\Theta_1(x),\Theta_2(x))$ for $x \in \cX$, we have by the construction of $\cX$ that
$$\langle \Theta(x) \mid x \in \cX \rangle_k = \End(W_1) \oplus \End(W_2).$$
Hence, for any $X \in \End(W_1)$, we can write the element
$\diag(X,0)$ of $\End(W_1) \oplus \End(W_2)$ as
$\diag(X,0) = \sum_{x \in \cX}a_x\Theta(x)$ for some $a_x \in k$; i.e.\
$$\sum_{x \in \cX}a_x\Theta_1(x) = X,~~\sum_{x \in \cX}a_x\Theta_2(x) = 0.$$
On the other hand, applying the Artin-Wedderburn theorem to the representation
$\Lambda_i$ of the $5'$-group $C$ on $U_i$, we have that
$$\langle \Lambda_i(u) \mid u \in C \rangle_k = \End(U_i).$$
In particular, any $Y \in \End(U_1)$ can be written as
$Y = \sum_{u \in C}b_u\Lambda_1(u)$ for some $b_u \in k$.
It follows that the element $\diag(Y \otimes X,0)$ of
$$\End(U_1) \otimes \End(W_1) \cong \End(U_1 \otimes W_1) = \End(V_1)
    \hookrightarrow \End(V)$$
can be written as
$$\diag(\sum_{u \in C,~x \in \cX}b_u a_x\Lambda_1(u)\otimes \Theta_1(x),
           \sum_{u \in C,~x \in \cX}b_u a_x\Lambda_2(u)\otimes \Theta_2(x))$$
$$= \sum_{u \in C,~x \in \cX}a_x b_u \cdot  \diag(\Phi_1(ux),\Phi_2(ux))
    = \sum_{u \in C,~x \in \cX}a_x b_u\Phi(ux),$$
and so it belongs to $\cC$. Thus $\cC \supseteq \End(V_1)$, and similarly
$\cC \supseteq \End(V_2)$. Since $G_1$ stabilizes each of $V_1$ and $V_2$, we
then have that
$$\cC = \End(V_1) \oplus \End(V_2) = \cA_1.$$
But $\Phi(h)$ interchanges $V_1$ and $V_2$. It follows that $\cM$ also contains
$$\cC\Phi(h) = \Hom(V_1,V_2) \oplus \Hom(V_2,V_1) = \cB_1,$$
as stated.

\smallskip
(e) Next we show that $\cE_{22}$ is a subquotient of $\cM$.
Choose $R_i \cong 2 \times (7:3) < L_i$, the normalizer of some Sylow $7$-subgroup of $L_i$. Note that $\bfN_{L_i}(R_i) = R_i$ and
\begin{equation}\label{r1}
  (\varphi_{j})_{R_1} = \alpha_j + \beta,
\end{equation}
where $\alpha_j,\beta \in \Irr(R_1)$
are of degree $3$ and $1$, respectively, and $\alpha_1 \neq \alpha_2$. Defining
$$R = R_1 \times R_2 \times \ldots \times R_n < \GP,$$
we see that $R$ satisfies the conditions of Lemma \ref{key3}. Hence the subspace
$A = e(\oplus^t_{i=1}A_i)$ defined in Lemma \ref{key3} (with $A_1$ affording
the $R_1$-character $\alpha_1$) is irreducible over the
$p'$-group $\bfN_G(R)$. By the Artin-Wedderburn theorem applied to
$\bfN_G(R)$ acting on $V = A \oplus B$, $\cM$ contains
$$\End(A) \supset \cD :=  \oplus_{1 \leq i \neq j \leq 2n,~ \{i,j\} \neq \{2a-1,2a\}}      \Hom(eA_{i},eA_{j}).$$
As noted above, each summand  $\Hom(V_{i},V_{j})$ in $\cE_{22}$ is acted on
trivially by $\prod_{s \neq k_i,k_j}L_s$, and affords the
$L_{k_i} \times L_{k_j}$-character $\varphi \otimes \varphi'$, where
$\varphi,\varphi' \in \{\varphi_1,\varphi_2\}$.
Working modulo $\cE_1 \oplus \cE_{21}$ and using this observation and
(\ref{r1}), we then see that all irreducible constituents of the $R$-character of the complement to $\cD$ in $\cE_{22}$ are of the form
$\gamma_{1} \otimes \gamma_2 \otimes \ldots \otimes \gamma_n$,
where $\gamma_i \in \Irr(R_i)$ and {\it all but at most one} of them have degree $1$
(and the remaining, if any, is some $\alpha_j$ of degree $3$). The same is true
for the complement to $\cM$ in $\cE_{22}$ (again modulo $\cE_1 \oplus \cE_{21}$). On the other hand, (\ref{r1}) and the aforementioned observation imply that the
$R$-character of the $\GP$-composition factor $\Hom(W_{i},W_{j})$
contains an irreducible $R$-character of degree $9$ (namely, an
$R_{k_i} \times R_{k_j}$-character of the form $\alpha \otimes \alpha'$,
with $\alpha,\alpha' \in \{\alpha_1,\alpha_2\}$). It follows that $\cE_{22}$ is a subquotient
of $\cM$.

\smallskip
(f) The results of (d), (e), together with the remarks made at the end of (c), imply
that all $\GP$-composition factors of $\End(V)/\cM$ (if any) are trivial. Hence
by Lemma \ref{key3} we conclude that $\cM = \End(V)$.
\end{proof}

\section{Weak adequacy for special linear groups}\label{sec:sl}

The exception (i) in Theorem \ref{simple1} requires much more effort to resolve.
We begin with setting up some notation. Let $n \geq 3$ and $q$ be a prime power
such that $p = (q^n-1)/(q-1)$. In particular, $n$ is a prime, $q = q_0^f$ for some
prime $q_0$ and $f$ is odd,
$\gcd(n,q-1) = 1$ and so $\PSL_n(q) = \SL_n(q) =: S$
and $G_n := \GL_n(q) = S \times \bfZ(G_n)$. Consider
the natural module
$$\cN = \bbF_q^n = \langle e_1, \ldots ,e_n \rangle_{\bbF_q}$$
for $G_n$, and let
$$Q = RL = \Stab_S(\langle e_2, \ldots, e_n \rangle_{\bbF_q}),$$
where $R$ is elementary abelian of order $q^{n-1}$ and $L \cong \GL_{n-1}(q)$.
Note that $Q$ is a $p'$-group.
It is well known, cf.\ \cite[Theorem 1.1]{GT1} that $G_n/\bfZ(G_n)$ has a unique irreducible
$p$-Brauer character $\delta$ of degree $p-2$, where $\delta(x) = \rho(x)-2$
for all $p'$-elements $x \in G_n$,
if we denote by $\rho$ the permutation character of $G_n$ acting on the set
$\Omega$ of $1$-spaces
of $\cN$. Let $\cD$ denote a $kG_n$-module affording $\delta$.

\begin{lem}\label{sl1a}
In the above notation, $\delta_Q = \alpha + \beta$, where
$\alpha \in \Irr(Q)$ has degree $q^{n-1}-1$, $\beta \in \Irr(Q)$ has degree
$(q^{n-1}-q)/(q-1)$, and
$$\alpha_R = \sum_{1_R \neq \lambda \in \Irr(R)}\lambda,~~
    \beta_R = \beta(1)1_R.$$
\end{lem}

\begin{proof}
Note that all non-trivial elements in $R$ are $L$-conjugate to a fixed
transvection $t \in R$, and $\delta(t) = \rho(t)-2 = (q^{n-1}-q)/(q-1)-1$. It follows
that
$$\delta_R =   \sum_{1_R \neq \lambda \in \Irr(R)}\lambda + \frac{q^{n-1}-q}{q-1}
    \cdot  1_R.$$
Next, $Q$ acts doubly transitively on the $1$-spaces of
$\langle e_2, \ldots ,e_n\rangle_{\bbF_q}$, with kernel containing $R$ and
with character $\beta + 1_Q$, where $\beta \in \Irr(Q)$ of degree
$(q^{n-1}-q)/(q-1)$. Hence $\beta$ is an irreducible constituent of $\delta$, and
the statement follows.
\end{proof}

In the subsequent treatment of $\SL_{n}(q)$, it is convenient to adopt the labeling of
irreducible
$\bbC G_n$-modules as given in \cite{J}, which uses Harish-Chandra induction
$\circ$. Each
such module is labeled as $S(s_{1},\la_{1}) \circ \ldots \circ S(s_{m},\la_{m})$, where
$s_{i} \in \overline{\bbF}_{q}^{\times}$ has degree $d_{i}$ (over $\bbF_q$), $\la_{i}$ is a
partition
of $k_{i}$, and $\sum^{m}_{i=1}k_{i}d_{i} = n$, cf.\ \cite{J}, \cite{KT1}.
Similarly, irreducible $kG_n$-modules are labeled as
$D(s_{1},\la_{1}) \circ \ldots \circ D(s_{m},\la_{m})$, with some extra conditions
including $s_i$ being a $p'$-element.
For $\la \vdash n$,
let $\chi^\la = S(1,\la)$ denote the unipotent character of $\GL_n(q)$ labeled by
$\la$. We make the convention that $\chi^{(n-2,2)} = 0$ for $n = 3$. Also, note
that $1_{G_n} = \chi^{(n)}$ and $\rho = 1_{G_n} + \chi^{(n-1,1)}$
(see e.g.\  \cite[Lemma 5.1]{GT1}). We next establish the
following result, which holds for arbitrary $\GL_n(q)$ with $n \geq 3$ and
which is interesting in its own right:

\begin{lem}\label{sl2}
In the above notation, we have the following decomposition of $\rho^2$ into
irreducible constituents over $G_n = \GL_n(q)$:
$$\begin{array}{ll}\rho^2 & =
    2\chi^{(n)} + 4\chi^{(n-1,1)} + \chi^{(n-2,2)} + 2\chi^{(n-2,1^2)} \vspace{1mm}\\
    & + \sum_{a \in \bbF_q^\times,~a^2 = 1 \neq a}S(a,(1^2)) \circ S(1,(n-2)) \vspace{1mm} \\ \vspace{1mm}
    & + \sum_{a \in \overline{\bbF}_q^\times,~a^{q-1} = 1 \neq a^2}S(a,(1)) \circ S(a^{-1},(1)) \circ S(1,(n-2))\\ \vspace{1mm}
    & +  \sum_{b \in \overline{\bbF}_q^\times,~b^{q+1} = 1 \neq b^2}S(b,(1)) \circ
    S(1,(n-2)).
    \end{array}$$
\end{lem}

\begin{proof}
Recall that $\rho$ is the permutation character of $G_n$ acting on
$\Omega$ and also on the diagonal $\{(x,x) \mid x \in \Omega\}$
of $\Omega \times \Omega$, whereas
$\rho^2$ is the permutation character of $G_n$ acting on $\Omega \times \Omega$.
Letting $H_n := \Stab_{G_n}(\langle e_1 \rangle_{\bbF_q},\langle e_2 \rangle_{\bbF_q})$, we then see that
$$\rho^2 = \rho + \Ind^{G_n}_{H_n}(1_{H_n}).$$
Note that $\Ind^{G_n}_{H_n}(1_{H_n})$ is just the Harish-Chandra induction of
the character $\Ind^{G_2}_{H_2}(1_{H_2}) \otimes 1_{G_{n-2}}$ of the Levi
subgroup $G_2 \times G_{n-2}$ of the parabolic subgroup
$$P :=\Stab_{G_n}(\langle e_1,e_2 \rangle_{\bbF_q})$$ of $G_n$, i.e.\
\begin{equation}\label{hc1}
  \Ind^{G_n}_{H_n}(1_{H_n}) = \Ind^{G_2}_{H_2}(1_{H_2}) \circ 1_{G_{n-2}}.
\end{equation}

Consider the case $q$ is odd. Then, according to the proof of \cite[Proposition 5.5]{NT},
\begin{equation}\label{hc2}
  \begin{array}{ll}\Ind^{G_2}_{H_2}(1_{H_2}) & =
    S(1,(2)) + 2S(1,(1^2)) + S(-1,(1^2)) \vspace{1mm} \\ \vspace{1mm}
    & + \sum_{a \in \overline{\bbF}_q^\times,~a^{q-1} = 1 \neq a^2}S(a,(1)) \circ S(a^{-1},(1)) \\ \vspace{1mm}
     & +  \sum_{a \in \overline{\bbF}_q^\times,~b^{q+1} = 1 \neq b^2}S(b,(1)).
    \end{array}
\end{equation}
Next, by \cite[Lemma 5.1]{GT1} we have
\begin{equation}\label{hc3}
  S(1,(2)) \circ S(1,(n-2)) = \Ind^{G_n}_{P}(1_P) =
    \chi^{(n)} + \chi^{(n-1,1)} + \chi^{(n-2,2)},
\end{equation}
\begin{equation}\label{hc4}
  S(1,(1)) \circ S(1,(1)) \circ S(1,(n-2)) =
    \chi^{(n)} + 2\chi^{(n-1,1)} + \chi^{(n-2,2)} + \chi^{(n-2,1^2)}.
\end{equation}
Since  $S(1,(1)) \circ S(1,(1)) = S(1,(2)) + S(1,(1^2))$, the statement follows
from (\ref{hc1})--(\ref{hc4}) and properties of the Harish-Chandra induction in $G_n$
(see \cite{J}).

The case $q$ is even can be proved similarly, using
$$\begin{array}{ll}\Ind^{G_2}_{H_2}(1_{H_2}) & =
    S(1,(2)) + 2S(1,(1^2)) \vspace{1mm} \\ \vspace{1mm}
    & + \sum_{a \in \overline{\bbF}_q^\times,~a^{q-1} = 1 \neq a^2}S(a,(1)) \circ S(a^{-1},(1)) \\ \vspace{1mm}
     & +  \sum_{a \in \overline{\bbF}_q^\times,~b^{q+1} = 1 \neq b^2}S(b,(1)).
    \end{array}$$
instead of (\ref{hc2}).
\end{proof}

\begin{lem}\label{sl3}
In the above notation, when $p = (q^n-1)/(q-1)$ we have the following decomposition of
$\delta^2$ into irreducible constituents over $S = \SL_n(q)$:
$$\begin{array}{ll}\delta^2 & =
    2D(1,(n)) + 2D(1,(n-1,1))+ D(1,(n-2,2)) + 2D(1,(n-2,1^2)) \vspace{1mm}\\
    & + \sum_{a \in \bbF_q^\times,~a^2 = 1 \neq a}D(a,(1^2)) \circ D(1,(n-2)) \vspace{1mm} \\ \vspace{1mm}
    & + \sum_{a \in \overline{\bbF}_q^\times,~a^{q-1} = 1 \neq a^2}D(a,(1)) \circ D(a^{-1},(1)) \circ D(1,(n-2))\\ \vspace{1mm}
    & +  \sum_{a \in \overline{\bbF}_q^\times,~b^{q+1} = 1 \neq b^2}D(b,(1)) \circ
    D(1,(n-2)).    \end{array}$$
In particular, if there is a composition factor $U$ of the $kS$-module
$\cD \otimes \cD$ with $U^R = 0$, then $n = 3$ and $U$ affords the
Brauer character $D(1,(1^3))$. Furthermore, the only composition factors of
$\cD \otimes \cD$ that are not of $p$-defect zero are the ones with
Brauer character $1_S = D(1,(n))$, $\delta = D(1,(n-1,1))$, and $D(1,(n-2,1^2))$.
\end{lem}

\begin{proof}
Let us denote by $\chi^\circ$ the restriction of any character $\chi$ of $G_n$ to
the set of $p'$-elements of $G_n$. Then
$$\delta^2  = (\rho^\circ  -2 \cdot 1_{G_n})^2 = (\rho^\circ)^2 - 4(\chi^{(n-1,1)})^\circ,$$
and we can apply Lemma \ref{sl2}. Since $p = (q^n-1)/(q-1)$ (or more generally,
if $p$ is a primitive prime divisor of $q^n-1$), all complex characters in
the decomposition for $\rho^2$ in Lemma \ref{sl2} are of $p$-defect $0$, except for
$\chi^{(n)}$, $\chi^{(n-1,1)}$, and $\chi^{(n-2,1^2)}$.  Furthermore,
$(\chi^{(n-2,1^2)})^\circ = D(1,(n-1,1)) + D(1,(n-2,1^2))$, cf.\ Proposition 3.1
and \S4 of \cite{GT1}; in particular,
$$D(1,(n-2,1^2))(1) = \frac{(q^n-q)(q^n-2q^2+1)}{(q-1)(q^2-1)}+1.$$
Since $G_n = S \times \bfZ(G_n)$, we arrive at the formulated
decomposition for $\delta^2$. Also, the degree of any irreducible constituent $\psi$ of
$\delta^2$ listed above is not divisible by $|R|-1 = q^{n-1}-1$, unless
$n = 3$ and $\psi = D(1,(1^3))$, whence $\psi_R$ must contain $1_R$ since
$L$ acts transitively on $\Irr(R) \setminus \{1_R\}$. In the exceptional case,
$\psi_R$ does not contain $1_R$, as one can see by direct computation (or by using
\cite[Theorem 5.4]{KT2}).
\end{proof}

\begin{cor}\label{sl4}
Assume that $p = (q^n-1)/(q-1)$ and $n \geq 5$. Then $S = \SL_n(q)$ is weakly
adequate on $\cD$.
\end{cor}

\begin{proof}
By Lemma \ref{sl1a} and the Artin-Wedderburn theorem applied to $Q$, $\cM$ contains
the subspace $\cA := (A \otimes A) \oplus (B \otimes B)$ of $\cD \otimes \cD = \End(\cD)$, with $A$ affording $\alpha$ and $B$ affording $\beta$. Thus, the complement
to $\cA$ in $\End(V)$ affords the $Q$-character $\Delta := 2\alpha\beta$. It follows
by Lemma \ref{sl1a} that $\Delta_R$ does {\it not} contain  $1_R$, whence
$R$ does not have any nonzero fixed point while acting on this complement.
The same must be true for the quotient $\End(V)/\cM$, which is a semisimple
$Q$-module. Since $n > 3$, by Lemma \ref{sl3} this can happen only
when $\cM = \End(V)$.
\end{proof}

Next we will extend the result of Corollary \ref{sl4} to the case $n = 3$.

\begin{prop}\label{sl5}
Assume that $p = (q^3-1)/(q-1)$. Then $S = \SL_3(q)$ is weakly
adequate on $\cD$.
\end{prop}

\begin{proof}
Note that $\delta$ is invariant under the graph automorphism $\tau$ of $S$, which
interchanges the two conjugacy classes of the maximal parabolic subgroup
$$Q = RL =\Stab_S(\cU) =\Stab_S(\langle e_1,e_2 \rangle_{\bbF _q})$$
and its opposite
$$Q^\sharp = R^\sharp L^\sharp = \Stab_S(\langle e_1 \rangle_{\bbF _q}).$$
Hence Lemma \ref{sl1a} also applies to $Q^\sharp$. To ease the notation, we will
drop the subscript $\bbF_q$ in various spans $\langle\cdot \rangle_{\bbF_q}$ in this proof.

\smallskip
First we will construct the  $Q$-submodules $\cA,\cB$ affording the character $\alpha$ and $\beta$ in $\cD$. Clearly, $R$ has $q+1$ fixed points $\omega \in \bP\cU$
and one orbit of length $q^2$
$$\cO := \{ \langle e_3 + y \rangle, ~~ y \in \cU\}$$
on $\Omega = \bP\cN$. Denoting
$\cI := \langle \sum_{\omega \in \bP\cN}\omega \rangle_k$, we can now
decompose $\cD = \cA \oplus \cB$ as $Q$-modules, where
$$\cA := [\cD,R] = \left(\left\{ \sum_{y \in \cU}a_y\langle e_3 + y \rangle \mid a_y \in k,
    \sum_{y \in \cU}a_y = 0\right\} \oplus \cI \right)/\cI,$$
$$\cB:= \bfC_\cD(R) = \left(\left\{ \sum_{\omega \in \bP\cU}b_\omega\omega \mid
    b_\omega \in k, \sum_{\omega \in \bP\cU}b_\omega = 0\right\} \oplus \cI \right)/\cI.$$

Next, $R^\sharp$ has $1$ fixed point $\langle e_1\rangle$ and $q+1$ orbits of length $q$:
$$\cO_{\infty} := \bP\cU \setminus \{\langle e_1 \rangle \},~~
    \cO_c := \{ \langle e_3 + ce_2 + de_1 \rangle, ~~ d \in \bbF_q\},~c \in \bbF_q$$
on $\bP\cN$. Then we can again decompose $\cD =  \cA^\sharp \oplus \cB^\sharp$
as $Q^\sharp$-modules, where $\cA^\sharp = [\cD,R^\sh]$ and
$\cB^\sh = \bfC_\cD(R^\sh)$. Note that
$\cO = \bP\cN \setminus \bP\cU = \cup_{c \in \bbF_q}\cO_c$. Hence, the $q(q-1)$ vectors
$$v_{c,d} = \langle e_3 + ce_2 + de_1 \rangle - \langle e_3 +ce_2 \rangle,~~
     c \in \bbF_q,~d \in \bbF_q^\times$$
belong to $\cA \cap \cA^\sh$, and similarly the $q-1$ vectors
$$u_a = \langle e_2 +ae_1 \rangle - \langle e_2 \rangle,~~a \in \bbF_q^\times$$
belong to $\cB \cap \cA^\sh$, and they are linearly independent. Thus
$$u_a \otimes v_{c,d} \in (\cA^\sh \otimes \cA^\sh) \cap (\cB \otimes  \cA),~~
    v_{c,d} \otimes u_a \in (\cA^\sh \otimes \cA^\sh) \cap (\cA \otimes  \cB)$$
and so both $(\cA^\sh \otimes \cA^\sh) \cap (\cB \otimes  \cA)$ and
$(\cA^\sh \otimes \cA^\sh) \cap (\cA \otimes  \cB)$ have dimension at least
$q(q-1)^2$. As a consequence,
\begin{equation}\label{two}
  \dim \left( (\cA^\sh \otimes \cA^\sh) \cap (\cA \otimes  \cB \oplus \cB \otimes  \cA)
         \right) \geq 2q(q-1)^2.
\end{equation}
Since $\cD$ is self-dual, it supports a non-degenerate $S$-invariant symmetric
bilinear form $(\cdot,\cdot)$, with respect to which $\cA$ and $\cB$ are orthogonal, and
so are $\cA^\sharp$ and $\cB^\sharp$. As usual, we can now identify
$\cD \otimes \cD$ with $\End(\cD)$ by sending $u \otimes v \in \cD \otimes \cD$ to
$$f_{u,v}~:~ x \mapsto (x,u)v$$
for all $x \in \cD$. Furthermore, in the proof of Corollary \ref{sl4} we have already mentioned that $\cM$ contains
the subspaces $\End(\cA) \oplus \End(\cB)$ (arguing with $Q$) and
$\End(\cA^\sh)$ (arguing with $Q^\sh$). It now follows from (\ref{two}) that
$$\dim (\End(\cA^\sh) \cap (\Hom(\cA,\cB) \oplus \Hom(\cB,\cA))) \geq 2q(q-1)^2.$$
Hence for $q \geq 5$ we have that
$$\begin{array}{ll}\dim \End(\cD) - \dim \cM & \leq
    (q^2+q-1)^2 - (q^2-1)^2 - q^2 -2q(q-1)^2\\
    & = 4q(q-1) < (q-1)(q^2-1) = \dim D(1,(1^3)).\end{array}$$
On the other hand, Lemma \ref{sl3} and the proof of Corollary \ref{sl4} show that
the only $S$-composition factor of $\End(\cD)/\cM$ (if any) is $D(1,(1^3))$. Hence, we
conclude that $\cM = \End(V)$ if $q \geq 5$. Since $p = (q^3-1)/(q-1)$, in the
remaining cases we have $q = 2,3$. The case $q=2$ is already handled before
as $S \cong \PSL_2(7)$, and the case $q=3$ has been checked by F.\ L\"ubeck
using computer.
\end{proof}

Now we can prove the weak adequacy of $G$ on $V$ in the case the
$\GP$-module is homogeneous.

\begin{prop}\label{sl-h}
Assume that $t = 1$, i.e.\ the $\GP$-module $V$ is homogeneous in the case
$(p,H,\dim W) = ((q^n-1)/(q-1),\SL_n(q),p-2)$
 of Theorem \ref{simple1}. Then $(G,V)$ is weakly adequate.
\end{prop}

\begin{proof}
Since $V|_{\GP} = eW$, by Theorem \ref{str} we have that $\GP = S = \SL_n(q)$.
Recall that $\gcd(n,q-1) = 1$ and $q = q_0^f$ where $q_0$ is a prime and $f$ is odd;
in particular, $\Out(S) \cong C_{2f}$ is cyclic.
It follows that $L := C \times S \lhd G = \langle L, \tau\rangle$
for some $\tau \in G$, and $C := \bfC_G(S)$ is a $p'$-group.
Let $\Psi$ denote the corresponding representation of $S$ on $W$ and $\Phi$
denote the corresponding representation of $G$ on $V$. Then, by Corollary \ref{sl4} and  Proposition \ref{sl5} we have that
$$\langle \Psi(y) \mid y \in S, y \mbox{ semisimple} \rangle_k = \End(W).$$

\smallskip
First we consider the case where $V_L$ is irreducible. Then $V \cong U \otimes W$,
where $U$ is an irreducible $kC$-module and $C$ acts trivially on $W$. Let $\Theta$
denote the corresponding representation of $C$ on $U$.
By the Artin-Wedderburn theorem, $\langle \Theta(x) \mid x \in C \rangle_k = \End(U)$.
Since $\Phi(xy) = \Theta(x) \otimes \Psi(y)$ for $x \in C$, $y \in S$, and since
$C$ is a $p'$-group, we conclude that $\cM$ contains $X \otimes Y$ for all
$X \in \End(U)$ and $Y \in \End(W)$, i.e.\ $\cM = \End(V)$.

\smallskip
Assume now that $V_L$ is reducible. Note that $V_L$ is
semisimple and multiplicity-free, as $G/L$ is cyclic. Since $W$
is $\tau$-invariant, it follows that
$$V_L = V_1 \oplus V_2 \oplus \ldots \oplus V_s
    \cong (U_1 \oplus U_2 \oplus \ldots \oplus U_s) \otimes W,$$
where $V_i = U_i \otimes W$ for some pairwise non-isomorphic
irreducible $kC$-module $U_i$, $1 \leq i \leq s$,
$\langle \tau \rangle$ acts transitively on the set of isomorphism classes of
$U_1, \ldots, U_s$, $C$ acts trivially on $W$ as before,
and $\Phi(\tau)$ permutes the summands $V_1, \ldots ,V_s$ transitively.
Let $\Theta_i$ denote the corresponding representation of $C$ on $U_i$, and let
$\Theta$ denote the corresponding representation of $C$ on
$U := U_1 \oplus \ldots \oplus U_s$.
Since $U_i \not\cong U_j$ for $i \neq j$, by the Artin-Wedderburn theorem,
$\langle \Theta(x) \mid x \in C \rangle_k = \End(U_1) \oplus \ldots \oplus \End(U_s)$.
It follows as above that $\cM$ contains $X \otimes Y$ for all
$Y \in \End(W)$ and all $X \in \End(U_i)$ (viewed as an element of $\End(U)$ by
letting it act as zero on $U_{j}$ for all $j \neq i$). In other words, $\cM$ contains the subspace $\End(V_1) \oplus \ldots \oplus \End(V_s)$ of $\End(V)$.

It remains to show that $\cM$ contains $\Hom(V_i,V_j)$ for any $i \neq j$.
Since $\Phi(\tau)$ permutes the summands $V_1, \ldots ,V_s$ transitively,
we can find $\sigma \in \langle \tau \rangle \setminus CS$ such that
$\Phi(\sigma)$ sends $V_i$ to $V_j$, and $\sigma$ induces a nontrivial outer
automorphism of $S$.
Observe that the condition $p = (q^n-1)/(q-1)$ implies that all elements in the coset
$S\sigma$ are $p'$-elements.
(Indeed, assume that $x\sigma$ has order divisible by $p$ for some $x \in S$.
Then some $p'$-power $g$ of $x\sigma$ is a $p$-element in $S$. It follows that
$\sigma$ preserves the conjugacy class $g^S$, which is impossible by inspecting
the eigenvalues of $g$.) So all elements in $L\sigma$ are $p'$-elements.
Hence $\cM$ also contains the subspace
$$\cA := \langle \Phi(h\sigma) \mid h \in L \rangle_k =
   \langle \Phi(h) \mid h \in L \rangle_k \cdot \Phi(\sigma).$$
Again by the Artin-Wedderburn theorem,
$$\langle \Phi(h) \mid h \in L \rangle_k = \End(V_1) \oplus \ldots \oplus \End(V_s).$$
Since $\Phi(\sigma)$ sends $V_i$ (isomorphically) to $V_j$, we conclude that
$$\cA \supset \End(V_j,V_j)\Phi(\sigma) = \Hom(V_i,V_j),$$
and so $\cM = \End(V)$.
\end{proof}

Next we consider the subgroup
$$Q' = R'L' =\Stab_S(\langle e_n\rangle_{\bbF_q},
    \langle e_2, \ldots, e_{n} \rangle_{\bbF_q}),$$
where $R'$ is a $q_0$-group of special type of order $q^{2n-3}$ and
$L' \cong \GL_{n-2}(q) \times \GL_1(q)$.
Note that the graph automorphism $x \mapsto \tw t x^{-1}$ of $S$ sends $Q'$ to
$(Q')^g$, where $g \in S$ sends $e_1$ to $e_n$, $e_n$ to $-e_1$, and
fixes all other $e_i$. Since the $S$-conjugacy class of the $p'$-group $Q'$ is fixed by
all field automorphisms, it is $\Aut(S)$-invariant. Also,
$Q'$ is just the normalizer in $S$ of the root subgroup
$Z' := \bfZ(R') = [R',R']$ (of order $q$), whence $\bfN_S(Q') = Q'$.

\begin{lem}\label{sl1b}
In the above notation, $\delta_{Q'} = \alpha' + \beta'_1 + \beta'_2 + \gamma' + 1_{Q'}$, where
$\alpha' \in \Irr(Q')$ has degree $q^{n-2}(q-1)$, $\beta'_{1,2} \in \Irr(Q')$ have degree
$q^{n-2}-1$, $\gamma' \in \Irr(Q')$ has degree $(q^{n-2}-q)/(q-1)$ if $n > 3$ and
$\gamma' = 0$ if $n = 3$, and
$$\alpha'_{Z'} = q^{n-2}\sum_{1_{Z'} \neq \lambda \in \Irr(Z')}\lambda,~~
    Z'  \leq \Ker(\beta'_1) \cap \Ker(\beta'_2) \cap \Ker(\gamma').$$
\end{lem}

\begin{proof}
Note that all non-trivial elements in $Z'$ are $L'$-conjugate to a fixed
transvection $t \in Z'$, and $\delta(t) = \rho(t)-2 = (q^{n-1}-q)/(q-1)-1$. It follows
that
$$\delta_{Z'} =   q^{n-2}\sum_{1_{R'} \neq \lambda \in \Irr(Z')}\lambda +
    (2(q^{n-2}-1) + \frac{q^{n-2}-q}{q-1} + 1)\cdot  1_{Z'}.$$
Since $R'$ is of special type, it also follows that $[\cD,Z']$ gives rise to an
irreducible $Q'$-module of dimension $q^{n-2}(q-1)$, with character $\alpha'$.
Now we can write $R'/Z' = (R'_1/Z') \times (R'_2/Z')$ as a direct product of two
$L'$-invariant subgroups. Next, $Q'$ acts on the
subset $\Omega'$ of $\Omega$ consisting of all $1$-spaces of
$\langle e_2, \ldots, e_n\rangle_{\bbF_q}$ (with kernel containing $R'_1$), with
two orbits. Arguing
as in the proof of Lemma \ref{sl1a}, we see that this permutation action affords the
$Q'$-character $\beta'_2+ \gamma' + 2 \cdot 1_{Q'}$, where the irreducible characters
$\beta'_2$ and $\gamma'$ (if $n > 3$; $\gamma' = 0$ if $n=3$) have indicated degrees. In general, $Q'$ has $3$ orbits on
$\Omega$, whence $1_{Q'}$ enters $\delta_{Q'}$. Also, note that $t$ has an
$S$-conjugate $t' \in R'_1 \setminus Z'$ and $\alpha'(t') = 0$. So if we set
$$\beta'_1(1) := \delta_{Q'} - (\alpha' + \beta'_2 + \gamma' + 1_{Q'}),$$
then we see that $\beta'_1 = \beta'_1(t) = q^{n-2}-1$ and $\beta'_1(t') = -1$.
Since $L'$ acts transitively on the non-trivial elements of $R'_1/Z'$, we conclude
by Clifford's theorem that $\beta'_1 \in \Irr(Q')$.
\end{proof}

As mentioned above, $S = \SL_n(q)$ has a unique irreducible $kS$-module $\cD$ of
dimension $p-2$. It follows by Theorem \ref{str} that in the situation (i) of Theorem
\ref{simple1},
$$\GP = S_1 \times \ldots \times S_t,$$
with $S_i \cong S$ and $\GP$ acts on $W_i$ with kernel $K_i := \prod_{j \neq i}S_j$.
Now, as $\GP$-modules, we have that
$$\cE := \End(V) \cong \bigoplus_{1 \leq i,j \leq t}V_i^* \otimes V_j \cong
     e^2 \bigoplus_{1 \leq i,j \leq t}W_i^* \otimes W_j$$
where $V_i^* \otimes V_i \cong \End(V_i)$ is acted on trivially by $K_i$,
whereas $W_i^* \otimes W_j$ is an {\it irreducible} $k\GP$-module with kernel
$K_i \cap K_j$ for $i \neq j$. It follows that the two $\GP$-submodules
$$\cE_1 := \bigoplus_{1 \leq i \leq t} V_i^* \otimes V_i,~~
    \cE_2 := \bigoplus_{1 \leq i \neq j \leq t}V_i^* \otimes V_j$$
of $\End(V)$ share no common composition factor.

Now we can prove the main result of this section:

\begin{thm}\label{main-sl}
Suppose we  are in the case (i) of Theorem \ref{simple1}, i.e.\
$(p,H,\dim W) = ((q^n-1)/(q-1),\SL_n(q),p-2)$. Then $(G,V)$ is weakly adequate.
\end{thm}

\begin{proof}
(a) Consider the subgroup
$$Q'^t = Q' \times \ldots \times Q' = Q'_1 \times \ldots \times Q'_t
   < S_1 \times \ldots \times S_t$$
of $\GP$. By
Lemma \ref{sl1b} and the discussion preceding it, $Q'^t$ satisfies the hypotheses
of Lemma \ref{key3}, with $A_i$ affording the $Q'$-character $\alpha'$,
and $\bfN_G(Q'^t)$ is a $p'$-group.
Note that $A_i \not\cong A_j$ for $i \neq j$ since $K_i \cap Q'^t \neq K_j \cap Q'^t$.
Also, the summands $A$ and $B$ of the $Q'^t$-module $V$ constructed in
Lemma \ref{sl1b} have no common composition factor and $A$ is irreducible.
Hence,
$$\cM \supseteq \End(A) \supset e^2\bigoplus_{1 \leq i \neq j \leq t}A_i^* \otimes A_j
    =: \cA$$
by the Artin-Wedderburn theorem. Note that $\cA \subset \cE_2$. Furthermore, if
$\Delta$ is the $Q'^t$-character of the complement of $\cA$ in $\cE_2$, then,
by Lemma \ref{sl1b}, each irreducible constituent of $\Delta$, when
restricted to
$$Z'^t = Z' \times \ldots \times Z' = Z'_1 \times \ldots \times Z'_t,$$
is trivial at (at least) {\it all but one} $Z'_i$. The same is true for the
$\GP$-module $\cE/(\cE_1+\cM)$. On the other hand, as mentioned above, all
$\GP$-composition factors of $\cE/\cE_1 \cong \cE_2$ are of the form
$W_i^* \otimes W_j$ with $i \neq j$. The Brauer character of any such
$W_i^* \otimes W_j$, being
restricted to $S_i \times S_j$, is $\delta \otimes \delta$, and so it contains
the $Q'_i \times Q'_j$-irreducible constituent  $\alpha' \otimes \alpha'$ which is
nontrivial at both $Z'_i$ and $Z'_j$ by Lemma \ref{sl1b}.  It follows that
$\cE_1 + \cM = \cE$, i.e.\ $\cM$ surjects onto $\cE_2$. Applying Lemma \ref{triv}
to the subgroup $\GP \leq G$, we conclude that $\cM \supseteq \cE_2$.

\smallskip
(b) We already mentioned that the $\GP$-modules
$\cE_1 = \oplus^t_{i=1}\cE_{1i}$ and $\cE_2$ share no common composition factor;
in particular, $k$ is not a composition factor of $\cE_2$.
Furthermore, since $\prod_{j \neq i}S_j$ acts trivially on $V_i$, we see that
the only common $\GP$-composition factor that $\cE_{1i}$ and $\cE_{1j}$ with
$i \neq j$ can share is the principal character $1_{\GP}$.
Recall that $\cE_{1i} \cong \cD \otimes \cD$ as $S_i$-modules. The irreducibility of
$G$ on $V$ implies that $G_i :=\Stab_G(V_i)$ acts irreducibly on $V_i$,
and certainly $\GP \lhd G_i$ acts homogeneously on $V_i$. By Proposition \ref{sl-h}
applied to $G_i$, $\cE_{1i}$ is a subquotient of $\cM$. We have therefore shown that
all {\it non-trivial} $\GP$-composition factors of $\cE = \End(V)$ also occur
in $\cM$ with the same multiplicity, and so all the composition
factors of the $\GP$-module $\cE/\cM$ (if any) are trivial. Applying Lemma \ref{key3}
to the subgroup $Q'^t < \GP$, we conclude that $\cM = \cE$.
\end{proof}

Finally we can prove

\begin{thm}\label{simple1-wa}
Suppose $(G,V)$ is as in the case (i) of Theorem \ref{str}. Then $(G,V)$ is weakly
adequate.
\end{thm}

\begin{proof}
In view of Theorems \ref{simple1}, \ref{main-sl}, and Propositions \ref{j3}, \ref{m12},
\ref{a7},  we need to handle the case $(p,H,\dim W) = (7,6 \cdot \PSL_3(4),6)$.
In this case, $L_i$ acts on each $W_j$ either trivially or as
$H_j \cong 6 \cdot \PSL_3(4)$. It follows by the faithfulness of $G$ on $V$ that
$\bfZ(L_i)$ has exponent $6$, and so $L_i$ is (isomorphic to) either
$X := (2 \times 2) \cdot 3 \cdot \PSL_3(4)$ or a quotient $6 \cdot \PSL_3(4)$ of $X$.
We can also find $k_i$ such that
the kernel $K_i$ of $\GP = L_1 * \ldots * L_n$ acting on $W_i$ contains
$\prod_{j \neq k_i}L_j$. Without loss we may assume $k_1 = 1$.

\smallskip
(a) We claim that $L_1$ contains a subgroup
$Q_1 = Z_1 \times \AAA_5$, whose conjugacy class
is $\Aut(L_1)$-invariant (with $Z_1 := \bfZ(L_1)$). For this purpose, without loss
we may assume that $L_1 \cong X$.
We consider a Levi subgroup $C_3 \times \SL_2(4) \cong C_3 \times \AAA_5$ of $\SL_3(4)$ which acts semisimply on the natural module $\bbF_4^3$.
Then its conjugacy class in $\SL_3(4)$ is fixed by all the outer automorphisms of $\SL_3(4)$.  Consider a faithful representation
$\Lambda~:~X \to \GL_{18}(\bbC)$,
which is the sum of three irreducible representations, on which $X$
acts with different kernels $C_2$, and let $Y$ be the full inverse image of $\AAA_5$
in $X$. Note that involutions in $\PSL_3(4)$ lift to involutions in
$6 \cdot \PSL_3(4)$, whereas involutions in $\AAA_5$ lift to elements of order $4$ in
$2 \cdot \AAA_5$  (see \cite{Atlas}). It follows that $\Lambda(x)$ has order $2$
for the inverse image $x \in X$ of any involution in $\AAA_5$, and so $|x| = 2$.
Hence $Y \cong (2 \times 2) \times \AAA_5$, and the claim follows.

Defining $Q_i < L_i$ similarly, we see that
$$Q = Q_1 * Q_2 * \ldots * Q_n$$
satisfies the condition (i) of Lemma \ref{key2}. Since $Q_1$ is self-normalizing in
$L_1$, we see that $\bfN_{\GP}(Q) = Q$ and $N := \bfN_G(Q)$ is a $p'$-group.

We will now inflate Brauer characters of $L_1$ acting on $W_1$ to $X$ and
then replace $L_1$ by $X$.
According to \cite{JLPW}, $L_1$ has exactly six irreducible $7$-Brauer characters
$\varphi_s$ of degree $6$, $1 \leq s \leq 6$,
lying above the six distinct characters $\lambda_s$ of
$Z_1$ (with kernels being the three distinct central subgroups of order $2$), and
$(\varphi_s)_{Q_1} = \lambda_s \otimes (\alpha + \beta)$,
where $\alpha \neq \beta \in \Irr(\AAA_5)$, and either
\begin{equation}\label{15}
  \{\alpha,\beta\} = \{1a,5a\}
\end{equation}
or
\begin{equation}\label{33}
  \{\alpha,\beta\} = \{3a,3b\}
\end{equation}
depending on whether $\varphi_s$ takes value $2$ or $-2$ on involutions in $\AAA_5$.
(Here we adopt the notation that $\Irr(\AAA_5) = \{1a,3a,3b,4a,5a\}$.)  In either case,
we have that $(W_1)_Q = A_1 \oplus B_1$, where the $Q$-modules $A_1$ and $B_1$
are irreducible and non-isomorphic. As shown in the proof of Lemma \ref{key},
$N\GP = G$ and $N_1\GP = G_1 :=\Stab_G(V_1)$ for $N_1 := \bfN_{G_1}(Q)$. So we
fix a decomposition $G = \cup^t_{i=1}g_iG_1$ with $g_i \in N$, $g_1 = 1$, and define
$A_i := g_i(A_1) \subset W_i$, $B_i := g_i(B_1) \subset W_i$. In particular, either
(\ref{15}) holds for all $(W_i)_Q$, or (\ref{33}) holds for all $(W_i)_Q$.

We claim that $Q$ also satisfies
the condition (ii) of Lemma \ref{key2}. Indeed, assume that $W_i \not\cong W_j$.
Now if $k_i \neq k_j$, then $L_{k_i} > Q_{k_i}$ acts trivially on $W_j$, but
$\bfZ(Q_{k_i}) = Z_{k_i}$ acts nontrivially by scalars on $W_i$.
In the case $k_i = k_j$, we may assume that $K_i \geq \prod_{s > 1}L_s$, and so
$W_i$ and $W_j$ afford the $L_1$-characters
$\varphi, \varphi' \in \{\varphi_1, \ldots ,\varphi_6\}$, lying
above {\it different} characters $\lambda, \lambda'$ of $Z_1$. Now
$\bfZ(Q_1) = Z_1$ acts on $W_i$ and $W_j$ by scalars but via different characters
$\lambda,\lambda'$, so we are done.

\smallskip
(b) Suppose we are in the case of (\ref{33}) and moreover $G_1 =\Stab_G(V_1)$
interchanges the two classes $5A = x^{L_1}$ and $5B = (x^2)^{L_1}$ of elements of
order $5$ of $L_1 = 6_1 \cdot \PSL_3(4)$. Certainly, we can choose
$x \in \AAA_5 < Q_1$.  Since $N_1\GP = G_1$, we can find
some element $g \in N_1$ that interchanges the classes $5A$ and $5B$. In this case
$g$ also interchanges the characters $\alpha = 3a$ and $\beta = 3b$ of $\AAA_5$,
but fixes $W_1$ and the central character $\lambda \in \{\lambda_1, \ldots, \lambda_6\}$
of $Z_1$. It follows that $\{A_1, \ldots ,B_t\}$ forms a single $N$-orbit, and so
by Lemma \ref{key2} the $p'$-group $N$ acts irreducibly on $V$, and we are done.

\smallskip
(c) From now on we may assume that we are {\it not} in the case considered in (b).
We claim that $\{A_1, \ldots ,A_t\}$ and $\{B_1, \ldots ,B_t\}$ are two distinct $N$-orbits.
Assume the contrary. Then by the construction of $A_i$ and $B_j$ there must be
some $h \in N$ such that $B_1 \cong A_1^h$. This is clearly impossible in the case
of (\ref{15}). In the case of (\ref{33}), $h \in G_1$ and furthermore $h$ fuses the
two classes of elements of order $5$ in $\AAA_5$. Hence $h \in G_1$ fuses the
classes $5A$ and $5B$ of $L_1$, contrary to our assumption.

Now we can apply Lemma \ref{key2} to see that $V_N = A \oplus B$ and so
\begin{equation}\label{sl34a}
  \cM \supseteq  \End(A) \oplus \End(B)
\end{equation}
by the Artin-Wedderburn theorem. We also decompose $\End(V) = \cE_1 \oplus \cE_2$
as $\GP$-modules, and note that the $Q$-modules
$$\cE_1 := \bigoplus^t_{i=1}\End(V_i) \cong e^2\bigoplus^t_{i=1}W_i^* \otimes W_i,~~
  \cE_2 := \bigoplus_{1 \leq i \neq j \leq t}\Hom(V_i,V_j) \cong
  e^2\bigoplus_{1 \leq i \neq j \leq t}W_i^* \otimes W_j$$
share no common composition factor. Indeed, the $p'$-group
$\bfZ(\GP) = Z_1 * \ldots * Z_n \leq \bfZ(Q)$ acts trivially on $\cE_1$ and nontrivially
by scalars on each $W_i^* \otimes W_j$ when $i \neq j$.

Moreover, if
$k_i \neq k_j$, say $K_i \geq \prod_{s \neq 1}L_s$ and $K_j \geq \prod_{s \neq 2}L_s$,
then $W_i^* \otimes W_j$ and $W_j^* \otimes W_i$ are irreducible over
$L_1 \times L_2$ (and acted on trivially by $\prod_{s>2}L_s$), with nontrivial central characters
$\nu_1^{-1} \otimes \nu_2$ and $\nu_1 \otimes \nu_2^{-1}$ over $Z_1 * Z_2$,
where $\nu_1, \nu_2 \in \{\lambda_1, \ldots ,\lambda_6\}$ have order $6$. If
$W_i \not\cong W_j$ but  $k_i = k_j$, say $k_i = k_j = 1$,
then $W_i$ and $W_j$ afford the $L_1$-characters $\varphi \neq \varphi'$ lying
above different characters $\lambda \neq \lambda'$ of $Z_1$. We distinguish
different scenarios for $\lambda$ and $\lambda'$.

(c1) $\lambda$ and $\lambda'$ coincide at $\bfO_2(Z_1)$ (then they must be
different at $\bfO_3(Z_1)$, and in fact $\lambda' = \lambda^{-1}$). Here,
$W_i^* \otimes W_j$ and $W_j^* \otimes W_i$ are reducible over
$L_1$ (and acted on trivially by $\prod_{s>1}L_s$), with distinct nontrivial central characters $\lambda^{-2}$ and $\lambda^{2}$ over $Z_1$. Furthermore,
the $L_1$-character of $W_i^* \otimes W_j$ is $\gamma_3 + \delta_3$,
where $\gamma_3 \in \IBr(L_1)$ has degree $15$, $\delta_3 \in \IBr(L_1)$ has degree $21$, and
\begin{equation}\label{sl34b1}
  (\gamma_3)_{\AAA_5} = 3a+3b+4a+5a,~~
  (\delta_3)_{\AAA_5} = 2 \cdot 1a + 4a + 3 \cdot 5a.
\end{equation}

(c2) $\lambda$ and $\lambda'$ coincide at $\bfO_3(Z_1)$ (then they must be
different at $\bfO_2(Z_1)$). Here,
$W_i^* \otimes W_j$ and $W_j^* \otimes W_i$ again are reducible over
$L_1$ (and acted on trivially by $\prod_{s>1}L_s$), with the same nontrivial central character $\lambda^{-1}\lambda'$ over $Z_1$. Furthermore,
the $L_1$-character of $W_i^* \otimes W_j$ is $\gamma_2 + \delta_2$,
where $\gamma_2 \in \IBr(L_1)$ has degree $10$, $\delta_2 \in \IBr(L_1)$ has degree $26$, and
\begin{equation}\label{sl34b2}
  (\gamma_2)_{\AAA_5} = 1a+4a+5a,~~
  (\delta_2)_{\AAA_5} = 1a + 3a +3b + 4a + 3 \cdot 5a.
\end{equation}
Here we have used the fact that the character of $W_i^* \otimes W_j$ takes value
$(\pm 2)^2 = 4$ at involutions in $\AAA_5$.

(c3) $\lambda$ and $\lambda'$ differ at both $\bfO_2(Z_1)$ and $\bfO_3(Z_1)$. Here,
$W_i^* \otimes W_j$ and $W_j^* \otimes W_i$ are irreducible over
$L_1$ (and acted on trivially by $\prod_{s>1}L_s$), with distinct nontrivial central characters $\lambda^{-1}\lambda'$ and $\lambda(\lambda')^{-1}$ over $Z_1$. Furthermore, the $L_1$-character of $W_i^* \otimes W_j$ is $\gamma_6$,
where $\gamma_6 \in \IBr(L_1)$ has degree $36$ and
\begin{equation}\label{sl34b3}
  (\gamma_6)_{\AAA_5} = 2 \cdot 1a + 3a + 3b + 2 \cdot 4a + 4 \cdot 5a.
\end{equation}

\smallskip
(d) According to (\ref{sl34a}), $\cM$ contains the subspace
$\cA := \End(C_1) \oplus \End(D_1)$
of $\End(V_1)$, which affords the character $e^2(\alpha^2 + \beta^2)$
of $\AAA_5 < Q_1$ (and is acted on trivially by $Z_1$). Note
that the $L_1$-character of $\End(W_1)$ is
$\varphi_i\overline{\varphi}_i = 1_{L_1} + \psi$, where
$\psi  \in \IBr(L_1)$ of degree $35$ and
$$\psi_{\AAA_5} = 1a + 3a +3b + 2 \cdot 4a + 4 \cdot 5a.$$
On the other hand, the $\AAA_5$-character of the complement to
$\cA$ in $\End(V_1)$ is
$$e^2(\alpha+\beta)^2 -e^2(\alpha^2+\beta^2) = 2e^2\alpha\beta,$$
which is $2e^2 \cdot 5a$ in the case of (\ref{15}) and
$2e^2(4a+5a)$ in the case of (\ref{33}); in particular, it does {\it not} contain $1a$. It
follows by the observation right after (\ref{sl34a}) and Lemma \ref{triv} that
$\cM \supseteq\End(V_1)$ and so $\cM \supseteq  \cE_1$.

\smallskip
(e) By (\ref{sl34a}), $\cM$ contains the subspace
$\cB_{ij} := \Hom(C_i,C_j) \oplus \Hom(D_i,D_j)$
of  $\cE_{ij} := \Hom(V_i,V_j)$ whenever $i \neq j$
(recall that $(C_i)_Q \cong eA_i$ and $(D_i)_Q \cong eB_i$). We distinguish two cases according as $k_i$ and $k_j$ are equal or not.

First suppose that $k_i \neq k_j$,
say $k_i = 1$ and $k_j = 2$. Then $\cE_{ij}$
affords the $L_1 \times L_2$-character
$e^2\overline{\theta}_1 \otimes \theta_2$ (where
$\theta_i \in \IBr(L_i)$ has degree $6$) and is acted on trivially by $\prod_{s>2}L_s$.
Now the
$Q_1 \times Q_2$-character of the complement to $\cB_{ij}$ in $\Hom(V_i,V_j)$ when
restricted to the subgroup $\AAA_5 \times \AAA_5$ is
$$e^2(\alpha_1+\beta_1)\otimes(\alpha_2+\beta_2) -
   e^2(\alpha_1 \otimes \alpha_2+\beta_1 \otimes \beta_2)
   = e^2(\alpha_1 \otimes \beta_2 + \beta_1 \otimes \alpha_2)$$
(where $\alpha_1,\beta_1$ play the role of $\alpha$ and $\beta$ for the first factor
$\AAA_5$ and similarly for $\alpha_2,\beta_2$).  Also, the restriction of
$\overline{\theta}_1 \otimes \theta_2$ to $\AAA_5 \times \AAA_5$
always contains an irreducible constituent distinct from $\alpha_1 \otimes \beta_2$ and
$\beta_1 \otimes \alpha_2$, namely $\beta_1 \otimes \beta_2$.

Assume now that $k_i = k_j = 1$. Then the $\AAA_5$-character of the complement to
$\cB_{ij}$ in $\cE_{ij}$ is
$$e^2(\alpha+\beta)^2 -e^2(\alpha^2+\beta^2) = 2e^2\alpha\beta$$
which is $2e^2 \cdot 5a$ in the case of (\ref{15}) and $2e^2(4a+5a)$ in the case of (\ref{33}). On the other hand, according to (\ref{sl34b1})--(\ref{sl34b3}), the restriction to
$\AAA_5$ of each of the irreducible constituents $\gamma$ and $\delta$ of
$W_i^* \otimes W_j$ always contains either $1a$ or $3a$.

Now assume that $\cM \neq \End(V)$. Working modulo $\cE_1 \subset \cM$, we
see that $\cM \supseteq \cB := \oplus_{i \neq j} \cB_{ij}$ has a nonzero complement
in $\cE_2 = \oplus_{i \neq j}\cE_{ij}$. But the above analysis shows that {\it any}
$\GP$-composition factor of $\cE_2$ contains a $Q$-irreducible constituent which is
{\it not} a $Q$-constituent of  the complement to $\cB$ in $\cE_2$, a contradiction.
\end{proof}

{\bf Proof of Theorem \ref{thm: weak adequacy}.}
(a) First we consider the case $k$ is algebraically closed. Assume that $\GP$ is
$p$-solvable. Then $G$ is also $p$-solvable. Furthermore, $(\dim V)/(\dim W)$ divides
$|G/\GP|$ by \cite[Theorem 8.30]{N}, and so $p \nmid (\dim V)$. So we are done
by Lemma \ref{p-sol}. So we may now assume that $\GP$ is not $p$-solvable,
$p > \dim W > 1$, and apply Theorem
\ref{str} to $G$. Then the statement follows from Theorem \ref{thm6} in the case
$\GP$ is a central product of quasisimple groups of Lie type in characteristic $p$ (if
in addition $p > 3$), and
from the results of \S\S\ref{sec:cross}, \ref{sec:sl} in the remaining cases.

Suppose that $p = 3$ and $\GP = L_1 * \ldots * L_n$ is a central product of quasisimple groups of Lie type in characteristic $p$ (with $\bfZ(L_i)$ a $p'$-group for each $i$,
see Theorem \ref{str}(iii)). Write $V_\GP = e\oplus^t_{i=1}W_i$ as
usual. It is well known that the only quasisimple groups of Lie type in characteristic
$p$ that have a faithful representation of degree $2$ over $k$ are $\SL_2(p^a)$.
Since $\dim W = 2$, we must have that
$L_j \cong \SL_2(q)$ for a power $q > 3$ of $3$ for all $j$ (as
the $\GP$-modules $W_i$ are $G$-conjugate); moreover,  for each $i$, there is
a unique $k_i$ such that $L_j$ acts nontrivially on $W_i$ precisely when $j = k_i$.
Note that $L_i$ contains
a unique conjugacy class of cyclic subgroups $T_i$ of order $C_{q-1}$. It is
straightforward to check that the restrictions of all Brauer characters
$\varphi \in \IBr_p(L_i)$ of degree $2$ to $Q_i := \bfN_{L_i}(T_i)$ are all irreducible
and pairwise distinct. Letting $Q := Q_1 * \ldots * Q_n$ and arguing as in the case (b1) of
the proof of Theorem \ref{simple1}, we see that $Q$ satisfies all the hypotheses of
Lemma \ref{key}, whence we are done.

\smallskip
(b) Now we consider the general case. We will view $G$ as a subgroup of
$\GL(V)$ and let
$\cM := \langle g \mid g \in G \mbox{ semisimple}\rangle_k$ as usual. Since
the $kG$-module $V$ is absolutely irreducible, the $\overline{k}G$-module
$\overline{V}:= V \otimes_k \overline{k}$ is irreducible, and the condition $d < p$ implies
that the dimension of any irreducible $\GP$-submodule in $\overline{V}$ is also less than $p$.
By the previous case, $\cM \otimes_k \overline{k} = \End(\overline{V})$. It follows that
$\dim_k\cM = (\dim V)^2$ and so $\cM = \End(V)$.
\hfill $\Box$

\section{Extensions and self-extensions. I: Generalities}
First we record a convenient criterion concerning self-extensions in blocks of cyclic
defect:

 \begin{lem}  \label{cyclic}
 Suppose that $G$ is a finite group and that $V$ is an irreducible
 $\fpb G$-representation  that belongs to a block of cyclic defect.
 Then $\Ext^1_G (V,V) \ne 0$ if and only if $V$ admits at least two
 non-isomorphic lifts to characteristic zero.
 In this case, $\dim \Ext^1_G(V,V)=1$.
 \end{lem}

 \begin{proof}
 Let $B$ denote the block of $V$.
 If $B$ has defect zero, $V$ is projective and lifts uniquely to characteristic
 zero.
 Otherwise, $B$ is a Brauer tree algebra.
 Note that $\Ext^1_G (V,V)\ne 0$ if and only if $V$ embeds as
 subrepresentation of $\PIM (V)/V$.
 The Brauer tree shows that this happens if and only if either (i) $B$
 has an exceptional vertex and $V$ is the unique edge incident
 with it, or (ii) $B$ does not have an exceptional vertex and $V$ is
 the unique edge of the tree.
 In (i), each exceptional representation in $B$ lifts $V$, in (ii) both ordinary
 representations in $B$ lift $V$, and it is clear that $V$ has at most one lift
 in all other cases.
 To verify the final claim, note that $\Hom(V,\PIM(V)/V) \cong \Ext^1_G(V,V)$,
 and that in a Brauer tree algebra $V$ occurs at most once in $\soc(\PIM(V)/V)$.
  \end{proof}

In fact, as pointed out to us by V.\ Paskunas, one direction of Lemma \ref{cyclic} holds 
for any finite group $G$: {\it If $\Ext^1_G(V,V) = 0$ then $V$ has at most one characteristic 
$0$ lift}.
Indeed, if $V$ has no self-extension, we may first realize all characteristic zero lifts over
some finite extension $\bbE$ of $\bbQ_p$, as well as $V$ over the residue field of $\bbE$. 
Then the universal deformation ring $R$ of $V$ over the ring $\cO_\bbE$ is a quotient of $\cO_\bbE$. But then 
$|\Hom_{\text{$\cO_\bbE$-alg}}(R, \cO_\bbE)| \le 1$, i.e.\  $V$ has at most one characteristic zero lift.

We will frequently use the following simple observations:

\begin{lem}\label{zero-ext}
Let $V$ be a finite dimensional vector space over $k$ and $G \leq \GL(V)$ a finite
absolutely irreducible subgroup. Write $V|_{G^+} = e\bigoplus^t_{i=1}W_i$, where the
$G^+$-modules $W_i$ are absolutely irreducible and pairwise non-isomorphic. Suppose
that $\Ext^1_{\GP}(W_i,W_j) = 0$ for all $i,j$. Then $\Ext^1_G(V,V) = 0$.
\end{lem}

\begin{proof}
Since $\GP$ contains a Sylow $p$-subgroup of $G$, $\Ext^1_G(V,V)$ embeds in
$$\Ext^1_{\GP}(V_\GP,V_\GP) = \Ext^1_{\GP}(e\bigoplus^t_{i=1}W_i,e\bigoplus^t_{i=1}W_i)
    \cong e^2\bigoplus_{i,j}\Ext^1_{\GP}(W_i,W_j) = 0.$$
\end{proof}

\begin{lem}\label{kernel}
Let $N$ be a normal subgroup of a finite group $X$ and let
$A$ and $B$ be finite dimensional $k(X/N)$-modules. Consider
$\Ext^1_X(A,B)$ where we inflate $A$ and $B$ to $kX$-modules.

{\rm (i)} If $\Ext^1_X(A,B) = 0$ then $\Ext^1_{X/N}(A,B) = 0$.

{\rm (ii)} If $\Ext^1_{X/N}(A,B) = 0$ and $\bfO^p(N) = N$  then $\Ext^1_X(A,B) = 0$.
\end{lem}

\begin{proof}
(i) is trivial. For (ii),
let $V$ be any extension of the $kX$-module $A$ by the $kX$-module $B$ and let
$\Phi~:~X \to \GL(V)$ denote the corresponding representation. Since $N$ acts trivially
on $A$ and on $B$, we see that $\Phi(N)$ is a $p$-group. But $\bfO ^p(N) = N$, hence
$\Phi(N) = 1$, i.e.\ $N$ acts trivially on $V$. Now, $V \cong A \oplus B$ as
$\Ext^1_{X/N}(A,B) = 0$.
\end{proof}

Next we recall the Holt's inequality in cohomology \cite{H}:

\begin{lem}\label{holt}
Let $G$ be a finite group, $N \lhd G$, and let $V$ be a finite dimensional $kG$-module.
Then for any integer $m \geq 0$ we have
$$\dim H^m(G,V) \leq \sum^m_{j=0}\dim H^j(G/N,H^{m-j}(N,V)).$$
\end{lem}

From now on we again assume that $k$ is algebraically closed.

\begin{cor}\label{product1}
Let $G = G_1 \times G_2$ be a direct product of finite groups and let $V_i$ be a
non-trivial  irreducible $kG_i$-module for $i = 1,2$.

{\rm (i)} If we view $V_1 \otimes V_2$ as a $kG$-module, then
$H^1(G,V_1 \otimes V_2) = 0$.

{\rm (ii)} If we inflate $V_1$ and $V_2$ to $kG$-modules, then $\Ext^1_G(V_1,V_2) = 0$.
\end{cor}

\begin{proof}
For (i), applying Lemma \ref{holt} to $N:=G_1$ we get
$$\dim H^1(G,V) \leq \dim H^0(G_2,H^1(G_1,V)) + \dim H^1(G_2,H^0(G_1,V)).$$
Now the $G_1$-module $V$ is a direct sum of $\dim (V_2)$ copies of $V_1$ and
$V_1$ is non-trivial irreducible, whence $H^0(G_1,V) = 0$. Next,
$H^1(G_1,V) \cong H^1(G_1,V_1) \otimes V_2$ as $G_2$-modules, with $G_2$ acting
trivially on the first tensor factor. It follows that
$$H^0(G_2,H^1(G_1,V)) \cong H^1(G_1,V_1) \otimes H^0(G_2,V_2) = 0$$
as $V_2$ is non-trivial irreducible, and so we are done.

(ii) follows from (i) since
$\Ext^1_G(V_1,V_2) \cong H^1(G,V_1^* \otimes V_2)$ and $V_1^*$ is a non-trivial
absolutely irreducible $kG_1$-module.
\end{proof}

\begin{cor}\label{product2}
Let the finite group $H$ be a central product of quasisimple subgroups
$H = H_1 * \ldots * H_n$, where $\bfZ (H_i)$ is a $p'$-group for all $i$.
For $i = 1,2$, let $W_i$ be a non-trivial irreducible
$kH$-module such that the action of $H$ on $W_i$ induces a quasisimple
subgroup of $\GL(W_i)$. Suppose that the kernels of the actions of $H$ on $W_1$ and on $W_2$ are different. Then $\Ext^1_H(W_1,W_2) = 0$.
\end{cor}

\begin{proof}
View $H$ as a quotient of $L := H_1 \times \ldots \times H_n$ by a central $p'$-subgroup and
inflate $W_i$ to a $kL$-module. Next, write $W_i = W^i_1 \otimes \ldots \otimes W^i_n$ for some absolutely irreducible $kH_j$-module $W^i_j$, $1 \leq i \leq 2$, $1 \leq j \leq n$.
Since $H_j$ is quasisimple, if $\dim W^i_j = 1$ then $H_j$ acts trivially on $W_i$. On
the other hand, if $\dim W^i_j > 1$, then $H_j$ induces a quasisimple subgroup
of $\GL(W^i_j)$. Hence, the condition that the action of $H$ on $W_i$ induces a quasisimple
subgroup of $\GL(W_i)$ implies that $\dim W^i_j > 1$ for exactly one index $j = k_i$,
whence the kernel of $L$ on $W_i$ is
$$H_1 \times \ldots \times H_{k_i-1} \times \bfC_{H_{k_i}}(W^i_{k_i}) \times H_{k_i+1}
   \times \ldots \times H_n.$$
Note that by the hypothesis on $H_i$,
$\prod_{j \neq k_1,~k_2}H_j$ has no non-trivial $p$-quotient. Hence,
by Lemma \ref{kernel} there is no loss to mod $L$ out by $\prod_{j \neq k_1,~k_2}H_j$.
If $k_1 \neq k_2$, then we are reduced to the case $L = H_{k_1} \times H_{k_2}$, $W_1$ is a non-trivial
$H_{k_1}$-module inflated to $L$ and $W_2$ is a non-trivial $H_{k_2}$-module inflated to $L$,
whence we are done by Corollary \ref{product1}(ii). Suppose now that $k_1 = k_2$, say
$k_1 = k_2 = 1$ for brevity.
Then we are reduced to the case $L = H_{1}$, and $K_1 \neq K_2$, where
$K_i = \bfC_{H_1}(W^i_{1}) \leq \bfZ(H_{1})$. By Schur's lemma, $\bfZ(H_{1})$
acts on $W_i$ by scalars and semisimply, via a linear character $\lambda_i$.  Since
$K_1 \neq K_2$, we see that $\lambda_1 \neq \lambda_2$. It follows (by considering
$\bfZ(H_{1})$-blocks, or by considering $\lambda_i$-eigenspaces for
$\bfZ(H_{1})$ in any extension of $W_1$ by $W_2$) that
$\Ext^1_L(W_1,W_2) = 0$.
\end{proof}

More generally, we record the following consequence of the K\"unneth formula,
cf.\ \cite[3.5.6]{Ben}.

 \begin{lem} \label{lem:kunneth}  Let $H$ be a finite group.  Assume that
 $H$ is a central product of subgroups $H_i, 1 \le i \le t$, and that
$\bfZ(H)$ is a $p'$-group.
 Let $X$ and $Y$ be irreducible $kH$-modules.  Write
 $X=X_1 \otimes \ldots \otimes X_t$ and $Y=Y_1 \otimes \ldots \otimes Y_t$
 where $X_i$ and $Y_i$ are irreducible $kH_i$-modules.
 \begin{enumerate}[\rm(i)]
 \item If $X_i$ and $Y_i$ are not isomorphic for two distinct $i$, then
 $\ext_H^1(X,Y)=0$.
 \item If $X_1$ and $Y_1$ are not isomorphic but $X_i \cong Y_i$ for $i >1$, then
 $\ext_H^1(X,Y) \cong \ext_{H_1}^1(X_1,Y_1)$.
 \item  If $X_i \cong Y_i$ for all $i$, then
 $\ext_H^1(X,Y) \cong \oplus_i \ext_{H_i}^1(X_i, Y_i)$.
 \end{enumerate}
 \end{lem}

We continue with several general remarks.

\begin{lemma}\label{semi1}
Let $V$ be a $kG$-module of finite length.

{\rm (i)} Suppose that $X$ is a composition factor of $V$ such that
$V$ has no indecomposable subquotient of length $2$ with $X$ as a composition
factor. Then $V \cong X\oplus M$ for some submodule $M \subset X$.

{\rm (ii)} Suppose that $\Ext^1_G(X,Y) = 0$ for
any two composition factors $X$, $Y$ of $V$. Then $V$ is semisimple.
\end{lemma}

\begin{proof}
(i) We will assume that $V \not\cong X$. Let $U$ be a submodule of $V$ of
smallest length that has $X$ as a composition factor. First we show that
$U \cong X$. If not, then $U$ has a composition series
$U = U_0 > U_1 > \ldots > U_m = 0$ for some $m \geq 2$. Note that
$U/U_1 \cong X$, as otherwise $X$ would be a composition factor
of $U_1 \subset U$, contradicting the choice of $U$. Now $U/U_2$ is a
subquotient of length $2$ of $V$ with $X$ as a quotient. By the
hypothesis, $U/U_2 = U'/U_2 \oplus U''/U_2$ with $U'/U_2 \cong X$ and
$U'' \supset U_2$, again contradicting the choice of $U$.

Now let $M$ be a submodule of $V$ of largest length such that $M \cap U = 0$.
In particular, $V/M \supseteq (M+U)/M \cong X$. Assume furthermore that
$V \neq M +U$. Then we can find a submodule $V' \subseteq V$ such that
$V'/(M+U)$ is simple. Again, $V'/M$ is a subquotient of length $2$ of
$V$ with $X$ as a submodule. So by the hypothesis,
$V'/M = (M+U)/M \oplus N/M$ for some submodule $N \subseteq V$ containing
$M$ properly. But then
$$N \cap U = (N \cap (M+U)) \cap U = M \cap U = 0,$$
contrary to the choice of $M$. Thus $V = M \oplus U$ is decomposable.

(ii) Induction on the length of $V$. If $V$ is not simple, then by (i) we have
$V \cong V' \oplus V''$ for some nonzero submodules $V'$ and $V''$.
Now apply the induction hypothesis to $V'$ and $V''$.
\end{proof}

\begin{lem}\label{mult1}
Let $V$ be a $kG$-module. Suppose that $U$ is a composition factor of $V$ of
multiplicity $1$, and that $U$ occurs both in $\soc(V)$ and $\hd(V)$. Then
$V \cong U \oplus M$ for some submodule $M \subset V$.
\end{lem}

\begin{proof}
Let $U_1 \cong U$ be a submodule of $V$. Since $U$ occurs in $\hd(V)$, there is
$M \subset V$ such that $V/M \cong U$. Now if $M \supseteq U_1$, then
$U$ would have multiplicity $\geq 2$ in $V$. Hence $V = U_1 \oplus M$.
\end{proof}

\begin{lemma}\label{block1}
Let $V$ be a $kG$-module of finite length. Suppose the set of isomorphism
classes of composition factors of $V$ is a disjoint union $\cX \cup \cY$ of
non-empty subsets such that, for any $U \in \cX$ and $W \in \cY$, there is no
indecomposable subquotient of length $2$ of $V$ with composition factors
$U$ and $W$. Then $V$ is decomposable.
\end{lemma}

\begin{proof}
Let $X$, respectively $Y$, denote the largest submodule of $V$ with all composition
factors belonging to $\cX$, respectively belonging to $\cY$. By
their definition, $X \cap Y = 0$. We claim that $V = X \oplus Y$. If not, we can find
a submodule $Z \supset X \oplus Y$ of $V$ such that $U := Z/(X \oplus Y)$ is
a simple $G$-module. Suppose for instance that $U \in \cX$. Applying Lemma
\ref{semi1}(i) to the $G$-module $Z/X$ and its composition factor $U$,
we see that $Z/X \cong U \oplus Y$. This implies that
$Z$ contains a submodule $T$ with $T/X \cong U$, contradicting the choice of $X$.

Now $X, Y \neq 0$ as $\cX,\cY \neq \varnothing$. It follows that $V$ is decomposable.
\end{proof}

\begin{lemma}\label{block2}
Let $V$ be an indecomposable $kG$-module.

{\rm (i)} If the $\GP$-module $V_\GP$ admits a composition factor $L$ of  dimension $1$, then all composition factors of $V_\GP$ belong to $B_0(\GP)$.

{\rm (ii)} Suppose a normal $p'$-subgroup $N$ of  $G$ acts by scalars on a composition factor $L$ of the $G$-module $V$. Then $N$ acts by scalars on $V$. If in addition $V$ is faithful then $N \leq \bfZ(G)$.
\end{lemma}

\begin{proof}
(i) Since $\GP = \bfO^{p'}(\GP)$, it must act trivially on $L$. Let $X$, respectively $Y$,
denote the largest submodule of the $\GP$-module $V$ with all composition
factors belonging, respectively not belonging, to $B_0(\GP)$. By
their definition and the definition of $\GP$-blocks, $V = X \oplus Y$.
Note that both $X$ and $Y$ are $G$-stable as $\GP \lhd G$. Since $V$ is
indecomposable, we see that $Y = 0$ and $V = X$.

(ii) Note that $N$ acts completely reducibly on $V$ and $G$ permutes
the $N$-homogene\-ous components of $V$. Since $V$ is indecomposable,
it follows that this action is transitive, whence all composition factors of
the $N$-module $V$ are $G$-conjugate. But, among them, the (unique) linear
composition factor of $L_{N}$ is certainly $G$-invariant. Hence this
is the unique composition factor of $V_{N}$, and so $N$ acts by scalars on $V$.
\end{proof}

\section{Indecomposable representations of $\SL_2(q)$}
\label{sec4}

We first prove a lemma.

\begin{lem}
\label{lemma1}
Suppose that $S,T$ are irreducible $\SL_2 (\bbF_q)$-representations over $\barFp$
with $q=p^n$, $n\ge 2$, and $E$ is a non-split extension of $T$ by $S$.
Then $\dim E \ge p$ and $S\not\cong T$. Moreover, if  $\dim S = \dim T$
then $\dim E \geq (p^2-1)/2$.
\end{lem}

\begin{proof}
This is immediate from Corollary 4.5(a) in \cite{AJL}.
\end{proof}

\begin{prop}
\label{prop1}
Suppose that $V$ is a reducible, self-dual, indecomposable representation
of $\SL_2 (\bbF_q)$ over $\barFp$, where $q=p^n$.
If $\dim V < 2p-2$, then $q=p$ and either of the following holds:

\begin{enumerate}[\rm(i)]
\item $\dim V = p$ and $V\cong \PIM (\bbone)$.
\item $\dim V = p+1$ and $V$ is the unique nonsplit self-extension
of $L\left( \frac{p-1}{2}\right)$.
\item $\dim V = p-1$ and $V$ is the unique nonsplit self-extension
of $L\left( \frac{p-3}{2}\right)$.
\end{enumerate}
Conversely, all the listed cases give rise to examples.
\end{prop}

\begin{proof}
Note that $p > 2$.

\smallskip
(a) Suppose first that $q=p$. If $V$ is projective, then since $\dim V < 2p$, we must have
$V \cong \PIM(\bbone)$, which is uniserial of shape $(L(0)|L(p-3)|L(0))$ and of
dimension $p$. (See for example \cite{Alp}.) If $V$ is non-projective, then, as
$\SL_2(p)$ has a cyclic Sylow $p$-subgroup,
$V$ is one of the ``standard modules'' described in \cite[\S5]{Janusz}. As $V$ is
self-dual, the standard modules are described by paths in the Brauer tree as in
\cite[(5.2)(b)]{Janusz} with $P_0 = Q = P_{k+1}$. By inspecting the Brauer trees of $\SL_2(p)$ (see e.g.\  \cite{Alp}) and using that $\dim V < 2p-2$ we deduce
moreover that $k = 1$ above, obtaining the modules in (ii), (iii).

In case (i), it is obvious that the module is self-dual since it is $\cP(\bbone)$.
In cases (ii) and (iii) the uniqueness of the isomorphism class of the extension
implies that it is self-dual.

\smallskip
(b)  Now suppose that $q > p$.
We need to show that no such $V$ exists.
(In fact we will show this holds even under the weaker bound $\dim V < 2p$.)
Pick an irreducible subrepresentation $L(\lambda)$ of $V$, where
$\lambda = \sum\limits_{i=0}^{n-1} p^i \lambda_i$, $0\le \lambda_i \le p-1$.
Then $V$ has a subquotient isomorphic to a nonsplit extension
$0\to L(\lambda) \to E \to L(\mu) \to 0$, where $\mu=\sum\limits_{i=0}^{n-1}
p^i \mu_i$, $0\le \mu_i \le p-1$.
By Lemma \ref{lemma1} we know that $\lambda\ne \mu$, hence $2\dim L(\lambda) +
\dim L(\mu) < 2p$.
By Corollary 4.5(a) in \cite{AJL} we deduce that, up to a cyclic relabelling of the indices,
$\lambda = \lambda_0+p$, $\mu = p-2-\lambda_0$, and
$\mu > \frac{2p-3}{3}\ge 1$.
In particular, $\mu$ uniquely determines $\lambda$.
Hence if $\soc V$ contains two non-isomorphic irreducible representations,
then $V$ admits indecomposable subrepresentations of length two that intersect in zero,
so $\dim V \ge 2p$ by Lemma \ref{lemma1}.
Therefore, $\soc V \cong L(\lambda)^{\oplus r}$, some $r\ge 1$.

Suppose first that $r\ge 2$.
We claim that $\soc_2 V /\soc V \cong L(\mu)^{\oplus s}$, for some $0 \le \mu < p^n$ and some $s\ge 1$.
(Here $\soc_i M$ is the increasing filtration determined by $\soc_0 M = 0$
and $\soc_i M /\soc_{i-1} M = \soc (M/\soc_{i-1} M)$.
Note that the socle filtration is compatible with subobjects.)
Note that any constituent of $\soc_2 V /\soc V$ extends $L(\lambda)$, hence by above
it is uniquely determined, unless $n = 2$ and $\lambda_0 = 1$. In the latter case,
the constituents can be $L(\mu')$, $L(\mu'')$, where $\mu' = p-3$, $\mu'' = p(p-3)$. But only one
of them can occur since $\dim L(\lambda)+\dim L(\mu')+\dim L(\mu'') = 2p$, and this proves the claim.
Note that $L(\mu)$ can occur only once in $V$ by Lemma \ref{lemma1}; in particular, $s = 1$.
We claim that $\dim \Ext^1\big( L(\mu) , L(\lambda)\big) \ge r \ge 2$.
Otherwise, $\soc_2 V$ is decomposable, so we obtain a splitting $\pi : \soc_2 V \to L(\lambda) \subset \soc V$.
But $\Ext^1(V/\soc_2 V, L(\lambda)) = 0$, so we can extend $\pi$ to a splitting of $V$; contradiction.
Hence $\dim \Ext^1\big( L(\mu) , L(\lambda)\big) \ge 2$ and by Corollary 4.5(b) in \cite{AJL}
we deduce that $n=2$ and $\lambda_i,\mu_i\in
\left\{ \frac{p-3}{2},\frac{p-1}{2}\right\}$ for all $i$.
(Note that we can get all four combinations with $\lambda_i + \mu_i = p-2$,
unlike what is claimed in that corollary.)
This contradicts that $\big| \{ \lambda_ i , \mu_i : 0 \le i \le n-1\} \big| \ge 3$
(by above).

Suppose that $r=1$, so $\soc V$ is irreducible.
Note that $\soc_3 V = V$ by Lemma \ref{lemma1}, as each constituent in a socle
layer extends at least one constituent of the previous socle layer.
As $\soc V$ is irreducible, $V$ embeds in the projective indecomposable
module $U_n (\lambda)$ whose socle is $L(\lambda)$.
We have $V \subset \soc_3 U_n (\lambda)$.
Note that $\lambda_i < p-1$ for all $i$, as $\dim V < 2p$.
By Lemma \ref{lemma1}, $L(\lambda)$ does not occur in $\soc_2 U_n (\lambda)/
\soc U_n (\lambda)$.
Also, $L(\lambda)$ occurs precisely $n$ times in $\soc_3 U_n (\lambda) / \soc_2 U_n (\lambda)$.
(Theorems 4.3 and 3.7 in \cite{AJL} imply that this is the case, unless $n=2$
and $\lambda_i \in \left\{ \frac{p-3}{2} , \frac{p-1}{2}\right\}$ for all $i$.
But by above $\lambda_i < \frac{p-3}{3} \le \frac{p-3}{2}$ for some $i$.)
Let $M_i = L(\lambda_0) \otimes L(\lambda_1)^{(p)} \otimes \cdots \otimes
Q_1 (\lambda_i)^{(p^i)} \otimes \cdots \otimes L(\lambda_{n-1})^{(p^{n-1})}$
and $M := M_0 + \cdots + M_{n-1} \subset U_n (\lambda)$
in the notation of \cite{AJL}, \S3.
Note by Theorems 4.3 and 3.7 in \cite{AJL} that $\soc_2 U_n (\lambda) \subset
M \subset \soc_3 U_n (\lambda)$ and that $M/\soc_2 U_n (\lambda) \cong
L(\lambda)^{\oplus n}$.
Therefore $V\subset M$, so
\[
\frac{V}{L(\lambda)} \subset \frac{M}{L(\lambda)} =
\frac{M_0}{L(\lambda)} \oplus \cdots \oplus \frac{M_{n-1}}{L(\lambda)} .
\]
As $\hd \big(M_i / L(\lambda)\big) \cong L(\lambda)$, there exists
$i$ such that $V/L(\lambda)$ surjects onto $M_i / L(\lambda)$.
Thus $\dim V \ge \dim M_i \ge 2p$.
\end{proof}

 \section{Finite groups with indecomposable modules of small dimension}

Throughout this section, we assume that $k = \bar{k}$ is a field of
characteristic $p > 3$. We want to describe the structure of finite groups $G$ that admit
reducible indecomposable modules of dimension $\leq 2p-2$.
The next results essentially reduce us to the case of quasisimple groups.

\begin{lem}\label{eg}
Let $G$ be a finite group, $p > 3$, and $V$ be a faithful $kG$-module of dimension
$< 2p$. Suppose that $\bfO_p(G) = 1$ and $\bfO_{p'}(G) \leq \bfZ(G)$. Then
$F(G) = \bfO_{p'}(G) = \bfZ(G)$, $F^*(G) = E(G)\bfZ(G)$,
and $\GP = E(G)$ is either trivial or a central
product of quasisimple groups of order divisible by $p$. In particular, $G$ has
no composition factor isomorphic to $C_p$ and so $H^1(G,k) = 0$.
\end{lem}

\begin{proof}
(a) Since $\bfO_p(G) = 1$, $Z := \bfZ(G) \leq F(G) \leq \bfO_{p'}(G)$. It follows that
$F(G) = Z = \bfO_{p'}(G)$, and $F^*(G) = E(G)Z$. If moreover $E(G) = 1$, then
$$Z = F(G) = F^*(G) \geq \bfC_G(F^*(G)) = G,$$
whence $G$ is an abelian $p'$-group, and $\GP = 1 = E(G)$.

\smallskip
(b) Assume now that $E(G) > 1$ and write $E(G) = L_1 * \ldots *L_t$, a central product of
$t \geq 1$ quasisimple subgroups. Since $\bfO_{p'}(E(G)) \leq \bfO_{p'}(G) = Z$,
$p||L_i|$ for all $i$.

Next we show that $\bfN_G(L_i)/\bfC_G(L_i)L_i$ is a $p'$-group for all $i$. Indeed,
note that the $L_i$-module $V$ admits a nontrivial composition factor $U$ of
dimension $< 2p$. Otherwise it has a composition series with all composition factors
being trivial, whence $L_i$ acts on $V$ as a $p$-group. Since $V$ is faithful and
$L_i$ is quasisimple, this is a contradiction. So we can apply Theorem \ref{bz} and
\cite[Theorem 2.1]{GHT} to the image of $L_i$ in $\GL(U)$. In particular, denoting $S_i := L_i/\bfZ(L_i)$, one
can check that
$\Out(S_i)$ is a $p'$-group, unless it is a simple group of Lie type in characteristic
$p$. In the former case we are done since
$\bfN_G(L_i)/\bfC_G(L_i)L_i \hookrightarrow \Out(S_i)$. Consider the latter case.
Observe that $\bfZ(L_i) \leq \bfZ(E(G)) \leq F(G)$ is a $p'$-group. So we may
replace $L_i$ by its simply connected isogenous version
$\cG^F$, where $F~:~\cG \to \cG$ is a Steinberg endomorphism on a
simple simply connected algebraic group $\cG$ in characteristic $p$.  If moreover
$p$ divides $|\bfN_G(L_i)/\bfC_G(L_i)L_i|$, then $\bfN_G(L_i)$ induces an outer
automorphism $\sigma$ of $L_i$ of order $p$. As $p > 3$, this can happen only
when $\sigma$ is a field automorphism. More precisely, $L_i$ is defined over
a field $\bbF_{p^{bp}}$ (for some $b \geq 1$), where $\bbF_{p^{bp}}$ is
the smallest splitting field for $L_i$ (cf.\ \cite[Proposition 5.4.4]{KL}) and $\sigma$ is induced by the field automorphism $x \mapsto x^{p^b}$. Since $\dim U \geq 2 > (\dim V)/p$,
$U$ must be $\sigma$-invariant. In turn, this implies by
\cite[Proposition 5.4.2]{KL} that $U$ and its $(p^b)^{\mathrm {th}}$ Frobenius twist
are isomorphic. In this case, the proofs of Proposition 5.4.6 and Remark
5.4.7 of \cite{KL} show that $\dim U \geq 2^p > 2p$, a contradiction.

\smallskip
(c) Recall that $\bfC_G(E(G)) = \bfC_G(F^*(G)) \leq F^*(G)=E(G)Z$, whence
$\bfC_G(E(G)) =  Z$. Also, $G$ acts via conjugation on the set $\{L_1, \ldots ,L_t\}$,
with kernel say $N$. We claim that $p \nmid |G/N|$. If not, then we may assume that some
$p$-element $g \in G$ permutes $L_1, \ldots,L_p$ cyclically. Arguing as in (b), we see
that $L_1$ acts nontrivially on some composition factor $U$ of the $E(G)$-module $V$,
and we can write $U = U_1 \otimes \ldots \otimes U_t$, where
$U_i \in \IBr_p(L_i)$.
If $U$ is not $g$-invariant, then $\dim V \geq p(\dim U) \geq 2p$, a contradiction.
Hence $U$ is $g$-invariant. It follows that $2 \leq \dim U_1 = \ldots = \dim U_p$
and so $\dim U \geq 2^p > 2p$, again a contradiction.

Now $N/E(G)Z$ embeds in $\prod^t_{i=1}\Out(L_i)$. Furthermore, the projection of $N$ into $\Out(L_i)$ induces
a subgroup of $\bfN_G(L_i)/\bfC_G(L_i)L_i$, which is a $p'$-group by (b). It follows
that $N/E(G)Z$ is a $p'$-group, and so $\GP = E(G)$. The last statement also follows.
\end{proof}

The next result on $H^1$ follows by standard results on
$H^1$ -- see  \cite[Lemma 5.2]{GKKL} and the main result of \cite{Gcr}.

\begin{lem}  \label{lem:h1} Let $G$ be a finite group and
let $V$ be a faithful irreducible $kG$-module.   Assume that $H^1(G,V) \ne 0$.
Then $\bfO_{p'}(G)=\bfO_p(G)=1$,
$E(G)= L_1 \times \ldots \times L_t$ and $V_{E(G)} =W_1 \oplus \ldots \oplus W_t$,
where the $L_i$ are isomorphic non-abelian simple groups of order divisible by $p$,
$W_i$ is an
irreducible $kL_i$-module,  and $L_j, j \ne i$ acts trivially on $W_i$.   Moreover,
$\dim H^1(G,V) \le \dim H^1(L_1,W_1)$, $\dim W_i \ge p-2$ and  $\dim V \ge t(p-2)$.
In particular, if $G$ is not almost simple, then  either $\dim V = 2p-4, 2p-2$
or $\dim V \ge 2p$, or $(p,\dim V) = (5,9)$.
\end{lem}




\begin{lem} \label{lem:opprime}
Let $V$ be a faithful indecomposable $kG$-module with two composition factors
$V_1$, $V_2$.  Assume that $\bfO_p(G) =1$ and $\dim V \leq 2p-2$.
If $J:=\bfO_{p'}(\GP) \not\leq \bfZ(\GP)$, then the following hold:
\begin{enumerate}[\rm(i)]
\item $p= 2^a + 1$ is a Fermat prime,
\item $\dim V_1 = \dim V_2 = p-1$,
\item  $J/\bfZ(J)$ is elementary abelian of order $2^{2a}$,
\item  $H^1(\GP,k) \ne 0$.
\end{enumerate}
\end{lem}

\begin{proof}
Since $\Ext^1_G(V_1,V_2) \hookrightarrow \Ext^1_\GP(V_1,V_2)$, there are
irreducible $\GP$-submodules $W_i$ of $V_i$ for $i = 1,2$ such that
$\Ext^1_\GP(W_1,W_2) \neq 0$. Assume that $J$ acts by scalars on at least one of the
$W_i$. Then by Lemma \ref{block2}(ii), $J$ acts by scalars on both $W_{1,2}$. If
$W'_1$ is any $\GP$-composition factor of $V_1$, then $W'_1$ is $G$-conjugate
to $W_1$. But $J \lhd G$, so we see that $J$ acts by scalars on $W'_1$. Thus
$J$ acts by scalars on all $\GP$-composition factors of $V_1$, and similarly for $V_2$.
Consider a basis of $V$ consistent with a $\GP$-composition series of $V$, and any
$x \in J$ and $y \in \GP$. Then $[x,y]$ acts as the identity transformation
on each $\GP$-composition factor in this series, and so it is represented by an
upper unitriangular matrix in the chosen basis. The same is true for any
element in $[J,\GP] \lhd G$. Since $V$ is faithful, we see that
$[J,\GP] \leq \bfO_p(G) = 1$ and so $J \leq \bfZ(\GP)$, a contradiction.

Thus $J$ cannot act by scalars on any $W_i$.
Let $\Phi_i$ denote the representation of $\GP$ on $W_i$.
Then $H := \Phi_i(\GP) < \GL(W_i)$ has no nontrivial $p'$-quotient and contains a non-scalar
normal $p'$-subgroup $\Phi_i(J)$.  Applying Theorem \ref{bz} and also
\cite[Theorem A]{BZ} to $H$, we conclude that $p = 2^a+1$ is a Fermat prime,
$\dim W_i = p-1$, and $Q := \bfO_{p'}(H)$ acts irreducibly on $W_i$.
Furthermore, $\bfZ(Q) = \bfZ(H)$, and $H/Q$ acts irreducibly on $Q/\bfZ(Q)$,
an elementary abelian
$2$-group of order $2^{2a}$. Now $\Phi_i(J)$ is a normal $p'$-subgroup of $H$ that is
{\it not} contained in $\bfZ(Q)$. It follows that $\Phi_i(J)\bfZ(Q) = Q$,
$\bfZ(\Phi_i(J)) = \Phi_i(J) \cap \bfZ(Q)$, $J$ is irreducible on $W_i$, and
$\Phi_i(J)/\bfZ(\Phi_i(J)) \cong Q/\bfZ(Q)$ is elementary
abelian of order $2^{2a}$. Since $\dim V \leq  2p-2$, it also follows that $W_i = V_i$.

\smallskip
Letting $A := V_1^* \otimes V_2$, we then see that
$A = [J,A] \oplus \bfC_A(J)$ as $J$-modules. Next,
$$0 \neq \ext_G^1(V_1,V_2) \cong H^1(G,A) \cong H^1(G,\bfC_A(J)),$$
since $H^1(G,[J,A])=0$ by the inflation restriction sequence. It follows
that $\bfC_A(J) \neq 0$. But $J$ is irreducible on both $V_{1,2}$, so we must have
that $\dim \bfC_A(J) = 1$ and $V_1 \cong V_2$ as $J$-modules.  Since $\GP$ acts
trivially on any $1$-dimensional module, it follows that $H^1(\GP,k) \neq 0$.
Since $W_1 \cong W_2$ as $J$-modules and $V$ is a faithful semisimple
$J$-module, we also see
that $\Ker (\Phi_1) \cap J = \Ker (\Phi_2) \cap J = 1$. Thus $\Phi_i$ is faithful on $J$,
and so $J/\bfZ(J)$ is elementary abelian of order $2^{2a}$.
\end{proof}

\begin{lem}\label{two-cfs}
Let $V$ be a faithful indecomposable $kG$-module with two composition factors
$V_1$, $V_2$ of dimension $>1$, $p > 3$, and $\bfO_p(G) = 1$.

{\rm (i)} Assume that $\bfO_{p'}(\GP) \leq \bfZ(\GP)$, and either $\dim V < 2p-2$ or $\dim V_1 = \dim V_2 = p-1$. If $G^+$ is not quasisimple,
then $\GP = L_1*L_2$ is a central product of two quasisimple groups,
$\dim V_1 = \dim V_2 = p-1$ and, up to relabeling the $L_i$'s, one of the following holds.

\hskip1pc{\rm (a)} $V_i = A_i \otimes B$ as $\GP$-modules, where $A_i \in \IBr_p(L_1)$ is
of dimension $(p-1)/2$ and $B \in \IBr_p(L_2)$ is of dimension $2$; furthermore,
$\Ext^1_{L_1}(A_1,A_2) \neq 0$.

\hskip1pc{\rm (b)} $V_i = (A_i \otimes k) \oplus (k \otimes B_i)$ as $\GP$-modules,
where $A_i \in \IBr_p(L_1)$ has dimension $(p-1)/2$, and some $g \in G$
interchanges $L_1$ with $L_2$ and $A_i$ with $B_i$.  Furthermore,
$\Ext^1_{L_1}(A_1,A_2) \neq 0$.

{\rm (ii)} If $\dim V < 2p-2$, then $\GP$ is quasisimple.
\end{lem}

\begin{proof}
(i) By Lemma \ref{eg} applied to $\GP$, $\GP = (\GP)^+ = E(\GP) =L_1*L_2* \ldots * L_t$, a central product of
$t$ quasisimple groups. Suppose $t >1$. Since
$\Ext^1_G(V_1,V_2) \hookrightarrow \Ext^1_\GP(V_1,V_2)$,
there are irreducible $\GP$-submodules $W_i$ of $V_i$ for $i = 1,2$ such that
$\Ext^1_\GP(W_1,W_2) \neq 0$.  Write
$W_i = W_{i1} \otimes \ldots \otimes W_{it}$ where $W_{ij}$ is an irreducible $L_j$-module. By Lemma \ref{lem:kunneth}, we may assume  that
$W_{1j} \cong W_{2j}$ for $j=2, \ldots, t$, and
either $\Ext^1_{L_1}(W_{11},W_{21}) \neq 0$, or
$W_{11} \cong W_{21}$ and $\Ext^1_{L_j}(W_{1j},W_{2j}) \neq 0$ for some
$j$. Interchanging $L_1$ and $L_j$ in the latter case, we can always assume that
$\Ext^1_{L_1}(W_{11},W_{21}) \neq 0$.
By \cite[Theorem A]{Gcr} we then have
\begin{equation}\label{w1}
  \dim W_{11}+\dim W_{21} \geq p-1 > 2.
\end{equation}

Now if $W_{1j}$ is nontrivial for some $j \geq 2$, say $W_{12} \not\cong k$, then
$$\dim V \geq \dim W_1 + \dim W_2 \geq 2(\dim W_{11}+\dim W_{21}) = 2p-2.$$
It follows that $V_i = W_i = W_{i1} \otimes W_{i2} \otimes k \otimes \ldots \otimes k$,
$\dim W_{i1} = (p-1)/2$, $\dim W_{i2} = 2$. Furthermore, $t=2$ as $V$ is faithful,
and we arrive at (a).

We may now assume that $W_{1j} \cong W_{2j} \cong k$ for all $j > 1$.
Suppose that $G$ normalizes $L_1$. Since every $\GP$-composition
factor of $V_1$ is $G$-conjugate to $W_1$, it follows that $L_2$ acts trivially on
all composition factors of $V_1$. The same is true for $V_2$. As $L_2$ is
quasisimple, we see that $L_2$ acts trivially on $V$, contrary to the faithfulness of
$V$. Thus there must be some $g \in G$ conjugating $L_1$ to $L_j$ for some $j > 1$,
say $L_1^g = L_2$. By (\ref{w1}) we may assume that $W_{11} \not\cong k$. Then $g(W_1) \not\cong W_1$ as $L_2$ acts trivially on $W_1$ but not on $g(W_1)$.
Thus $(V_1)_\GP$ has at least two distinct simple summands $W_1$ and $g(W_1)$.
If furthermore $W_{21} \not\cong k$, then $(V_2)_\GP$ also has at least two
distinct simple summands $W_2$ and $g(W_2)$, and so
$$\dim V \geq 2(\dim W_1 + \dim W_2)  = 2(\dim W_{11}+\dim W_{21})\geq 2p-2.$$
In this case, we must have that $V_i = W_i \oplus g(W_i)$, $\dim W_i = (p-1)/2$,
$t = 2$ as $V$ is faithful, and we arrive at (b).

Consider the case $W_{21} \cong k$. Now (\ref{w1}) implies that
$\dim W_1 = \dim W_{11} \geq p-2$, whence $\dim V_1 \geq 2p-4$.
On the other hand, $\dim V_2 \geq 2$. It follows that $2p-4 = 2$,
again a contradiction.

\smallskip
(ii) By Lemma \ref{lem:opprime}, $\bfO_{p'}(\GP) \leq \bfZ(\GP)$. Hence we are done by (i).
\end{proof}

\begin{lem}\label{ext-defi}
Let $H$ be a quasisimple finite group of Lie type in char $p > 3$. Assume that
$V_1,V_2 \in \IBr_p(H)$ satisfy
$\dim V_1+\dim V_2 <2p$.

{\rm (i)} If $H \not\cong \SL_2(q),~\PSL_2(q)$, then $\Ext^1_H(V_1,V_2) = 0$.
In particular, there is no reducible indecomposable $kG$-module with
$\GP \cong H$ and $\dim V < 2p$.

{\rm (ii)} Suppose $H \cong \SL_2(q)$ or $\PSL_2(q)$, $\Ext^1_H(V_1,V_2) \neq 0$,
and $\dim V_1 = \dim V_2$. Then $q = p$ and $V_1 = L((p-3)/2)$ or
$L((p-1)/2)$.
\end{lem}

\begin{proof}
(i) Note that $\bfZ(H)$ is a $p'$-group as $p > 3$. Hence, we can replace $H$ by
the fixed point subgroup $\cG^F$ for some Steinberg endomorphism
$F~:~\cG \to \cG$ on some simple simply connected algebraic group $\cG$
defined over a field of characteristic $p$ (see Lemma \ref{kernel}).
Hence, if $H \not\cong \Sp_{2n}(5)$, the result
follows by \cite[Theorem 1.1]{McN}. In the exceptional case
$H=\Sp_{2n}(5)$, then $p=5$ and so we are only considering modules of
dimension at most $9$. If $n \ge 3$, then $\dim V_1 + \dim V_2 > 10$ unless
at least one of the $V_i$ is trivial and the other is either trivial or the natural
module of dimension $2n$, and in both cases $\Ext^1_H(V_1,V_2) = 0$.
If $n=2$, one just computes to see that all the relevant $\ext_H^1(V_1, V_2)$
are trivial (done by Klaus Lux).

Suppose now that $V$ is a reducible indecomposable $kG$-module with
$\GP \cong H$ and $\dim V < 2p$. By Lemma \ref{semi1}(ii), there are composition factors $V_1$, $V_2$ of $V$ such that $\Ext^1_G(V_1, V_2) \neq 0$. It then follows
that $\Ext^1_H(W_1,W_2) \neq 0$ for some simple $H$-summands $W_i$ of
$V_i$ for $i = 1,2$ and $\dim W_1 + \dim W_2 < 2p$, a contradiction.

\smallskip
(ii) Again we can replace $H$ by $\SL_2(q)$. The statement then follows from
Lemma \ref{lemma1} when $q > p$, and from \cite{AJL} if $q=p$.
\end{proof}

There are a considerable number of examples of non-split extensions
$(V_1|V_2)$ with $\GP$ non-quasisimple and $\dim V_1 + \dim V_2=2p-2$.
For example,  suppose that $G=\SL_2(p) \times \SL_2(p)$ and
$V_1 = L(1) \otimes L(a)$ and $V_2=L(1) \otimes L(p-a-3)$.  Then
by \cite{AJL} and Lemma \ref{lem:kunneth},  $\ext_G^1(V_1, V_2) \ne 0$.
For our adequacy results, we do need to consider the case where
$\dim V_1 = \dim V_2= p-1$ in more detail.

\begin{lem} \label{2p-2}
Let $V$ be a faithful indecomposable $kG$-module with two composition factors
$V_1$, $V_2$, both of dimension $p-1$.  Assume that $p > 3$ and $\bfO_p(G) =1$.
Then one of the following holds:
\begin{enumerate}[\rm(i)]
\item $\bfO_{p'}(\GP) \not\leq \bfZ(\GP)$ and Lemma \ref{lem:opprime} applies;
\item  $\GP$ is quasisimple; or
\item  $\GP = \SL_2(p) \times \SL_2(p^a)$ (modulo some central subgroup) and
one of the following holds:
\begin{enumerate} [\rm(a)]
\item   $V_1 \cong V_2 \cong  L((p-3)/2) \otimes L(1)^{(p^b)}$ as $\GP$-modules
(for some $0 \leq b < a$).
\item   $a=1$ and $V_1 \cong V_2 \cong   X \oplus Y$ where $\GP$ acts as a quasisimple
group on $X, Y$ and $\dim X = \dim Y = (p-1)/2$ (so $X, Y \cong L((p-3)/2)$
for the copy of $\SL_2(p)$ acting nontrivially on $X$ or $Y$).
\end{enumerate}
\end{enumerate}
\end{lem}

\begin{proof}
Assume that neither (i) nor (ii) holds. Then by Lemma \ref{two-cfs}(i),
$E(\GP) = \GP = L_1 * L_2$ is a central product
of two quasisimple groups, and either (a) or (b) of Lemma \ref{two-cfs}(i) occurs.
In either case, we see that $L_1$ admits an indecomposable module $W$ of length
$2$ with composition factors $A_1$ and $A_2$, both of dimension $(p-1)/2$.
By \cite[Theorem A]{BZ} applied to $W$, $L_1$ is of Lie type in characteristic $p$.
Also, $\bfZ(L_1) \leq \bfZ(\GP) \leq \bfO_{p'}(G)$ is a $p'$-group. Hence
$L_1 \cong \SL_2(p)$ (modulo a central subgroup) by Lemma \ref{ext-defi} and
$A_1 \cong A_2 \cong L((p-3)/2)$.
In particular, $L_2 \cong \SL_2(p)$ in case (b), and (iii)(b) holds. In the case of (a),
$B \in \IBr_p(L_2)$ has dimension $2$. Since $p > 3$, by Theorem \ref{bz} we
conclude that $L_2$ is of Lie type in characteristic $p$, and in fact
$L_2 \cong \SL_2(p^a)$ (modulo a central subgroup) and $B \cong L(1)^{(p^b)}$ for some $a \geq 1$ and $0 \leq b < a$. Thus (iii)(a) holds.
\end{proof}

\begin{prop}\label{indec-str}
Let $p > 3$ and let $G$ be a finite group with a faithful, reducible, indecomposable $kG$-module $V$ of dimension $\leq 2p-3$. Suppose in addition that $\bfO_p(G) = 1$. Then
$\GP = E(\GP)$, $G$ has no composition factor isomorphic to $C_p$, and one of the following holds.

{\rm (i)} $\GP$ is quasisimple.

{\rm (ii)} $\GP$ is a central product of two quasisimple groups and
$\dim V = 2p-3$.  Furthermore, $V$ has one composition factor of dimension $1$,
and either one of dimension $2p-4$ or two of dimension $p-2$. In either case,
$V \not\cong V^*$.
\end{prop}

\begin{proof}
(a) Note that $\bfO_p(\GP) \leq \bfO_p(G) = 1$. Next we show that
$J := \bfO_{p'}(\GP) \leq \bfZ(\GP)$.  As in the proof of Lemma \ref{lem:opprime},
it suffices to show that $J$ acts by scalars on every $\GP$-composition factor of
$V$. So assume that there is a $\GP$-composition factor $X$ of $V$ on
which $J$ does not act by scalars.  Again as in the proof of Lemma \ref{lem:opprime},
we see by Theorem \ref{bz} that $\dim X \geq p-1$. Since $\dim V \leq 2p-3$,
it follows that $X$ is a $\GP$-composition factor of multiplicity $1$, and moreover $J$ acts by scalars on any other $\GP$-composition factor $Y$ of $V$. Also, $X$ extends to
a $G$-composition factor (of multiplicity $1$) of $V$. Now, by Lemma \ref{semi1}(i),
there is an indecomposable subquotient of length $2$ of $V$ with $G$-composition
factors $X$ and $T \not\cong X$. In particular,  by symmetry we may assume that
$0 \neq \Ext^1_G(X,T) \hookrightarrow \Ext^1_\GP(X,T)$, and so
$\Ext^1_\GP(X,Y) \neq 0$ for some simple $\GP$-summand $Y$ of $T$. But this
is impossible by Lemma \ref{block2}(ii) (as $J$ acts by scalars on $Y$ but not on $X$).

Applying Lemma \ref{eg} to $\GP$, we see that
$$\GP = (\GP)^+ = E(\GP) = L_1 * \ldots *L_t,$$
a central product of $t$ quasisimple subgroups. Note that $t \geq 1$ as otherwise
$G$ is a $p'$-group and so $V$ does not exist. Furthermore, $G$ has no
composition factors isomorphic to $C_p$.

\smallskip
(b) Assume now that $t \geq 2$. Suppose in addition
that, for every composition factor $V_i$ of $V$, at most one of the components
$L_j$ of $\GP$ acts nontrivially on $V_i$. For $1 \leq j \leq t$, let $\cX_j$ denote the
set of isomorphism classes of composition factors $V_i$ of $V$ on which $L_j$ acts
nontrivially. Also let $\cX_0$ denote the set of isomorphism classes of
composition factors $V_i$ of $V$ on which $\GP$ acts trivially. By the faithfulness
of $V$, $\cX_j \neq \varnothing$ for $j > 0$. Consider for instance $X \in \cX_1$.
By Lemma \ref{semi1}(i), there is some $X' \in \cX_j$ (for some $j$) and some
indecomposable subquotient $W$ of length $2$ of $V$ with composition factors
$X$, $X'$. Note that the $p$-radical of the group induced by the action of
$G$ on $W$ is trivial as $C_p$ is not a composition factor of $G$.
Applying Lemma \ref{two-cfs}(ii) to $W$, we see that $j=0$ or $1$. Moreover,
if for {\it all} $X \in \cX_1$ there is no such $W$ with $X' \in \cX_0$, then Lemma \ref{block1} applied to $(\cX := \cX_1,\cY := \cup_{i \neq 1}\cX_i)$ implies that $V$ is decomposable, a contradiction. Thus for some $X \in \cX_1$, such $W$ exists with
$X' \in \cX_0$. Note that in this case $\dim X \geq p-2$. Indeed, $\GP$
acts trivially on $X'$, and by symmetry we may assume that
$$0 < \dim \Ext^1_G(X',X) \leq \dim \Ext^1_\GP(X',X).$$
Hence, for some simple summand $X_1$ of the $\GP$-module $X$ we have
$0 \neq \Ext^1_\GP(k,X_1) \cong H^1(\GP,X_1)$. Note that $C_p$ is not
a composition factor of $\GP$, so by Lemma \ref{kernel} we may assume here that
$\GP$ acts faithfully on $X_1$. Applying Lemma \ref{lem:h1} to $\GP$, we get
$\dim X \geq \dim X_1 \geq p-2$.

Similarly, for some $Y \in \cX_2$, we get an
indecomposable subquotient $T$ of length $2$ of $V$ with composition factors
$Y$ and $Y' \in \cX_0$, and moreover $\dim Y \geq p-2$. Since $\dim V \leq 2p-3$
and $\cX_0 \ni X',Y'$, we conclude that $\dim V = 2p-3$, $\dim X = \dim Y = p-2$,
$t = 2$, and
$X' \cong Y'$ has dimension $1$. Suppose in addition that $V \cong V^*$. Observe
that $X^* \not\cong Y,X'$, so $X \cong X^*$. Similarly, $Y$ and $X'$ are self-dual.
Thus all three composition factors of $V$ have multiplicity $1$ each and are self-dual.
At least one of them occurs in $\soc(V)$, and then also in $\hd(V)$ by
duality. It follows by Lemma \ref{mult1} that $V$ is decomposable, a contradiction.
Thus we arrive at (ii).

\smallskip
(c) Finally, we consider the case where at least two of the $L_i$'s act nontrivially on
some composition factor $V_i$ of $V$. By Lemma \ref{semi1}(i), there is some
indecomposable subquotient $W$ of length $2$ of $V$ with composition factors
$V_i$ and $V_j$. By Lemma \ref{two-cfs}(ii) applied to $W$, $\dim V_j = 1$. In turn
this implies by Lemma \ref{lem:h1} that $\dim V_i \geq 2p-4$. Since
$\dim V \leq 2p-3$, we must have that $\dim V_i = 2p-4$, $V = W$,
$t = 2$ and $\dim V = 2p-3$.
Applying Lemma \ref{mult1} and using the indecomposability of $V$ as above,
we see that $V \not\cong V^*$ and again arrive at (ii).
\end{proof}

\section{Extensions and self-extensions. II}

Let $q$ be any odd prime power. It is well known, see e.g.\ \cite{TZ2} and \cite{GMST}, that
the finite symplectic group $\Sp_{2n}(q)$ has two complex irreducible {\it Weil}
characters $\xi_{1,2}$ of degree $(q^n+1)/2$, respectively $\eta_{1,2}$ of degree $(q^n-1)/2$,
whose reductions modulo any {\it odd} prime $p\nmid q$ are absolutely irreducible and distinct and are called {\it ($p$-modular) Weil characters} of $\Sp_{2n}(q)$.

\begin{lem}\label{weil-ext}
Let $q$ be an odd prime power and $p$ an odd prime divisor of $q^n+1$ which
does not divide $\prod^{2n-1}_{i=1}(q^i-1)$.
Let $S:= \Sp_{2n}(q)$
and let $W_1$ and $W_2$ denote irreducible $kS$-modules affording the two
irreducible $p$-modular Weil characters of $S$ of degree $(q^n-1)/2$.
Then for $1 \leq i,j \leq 2$ we have that
$\Ext^1_S(W_i,W_j) = 0$, unless $i \neq j$ and $n = 1$, in which case
$\dim\Ext^1_S(W_i,W_j) =  1$.
\end{lem}

\begin{proof}
The conditions on $(n,q)$ imply that $(n,q) \neq (1,3)$. In this case, \cite[Theorem 1.1]{TZ1}
implies that each $W_i$ has a unique complex lift (a complex module affording some
$\eta_i$). Also, the Sylow $p$-subgroups of $S$ are cyclic of order $(q^n+1)_p$.
Hence $\Ext^1_S(W_i,W_i) = 0$ by Lemma \ref{cyclic}.

Note that an involutory diagonal automorphism $\sigma$ of $S$ fuses $\eta_1$ with
$\eta_2$ and $W_1$ with $W_2$. Consider the semi-direct product
$H := S : \langle \sigma \rangle$ and the irreducible $kH$-module
$V := \Ind^H_S(W_1)$ of dimension $q^n-1$. Certainly, $\Ind^H_S(\eta_1)$ is a complex lift
of $V$.

Assume that $n > 1$. Now if $(n,q) \neq (2,3)$, then by \cite[Theorem 5.2]{TZ1}, $S$ has
exactly five irreducible complex characters of degree $\leq (q^n-1)$: $1_S$, $\eta_{1,2}$,
and $\xi_{1,2}$. When $(n,q) = (2,3)$, there is one extra complex character of degree
$6$, cf.\ \cite{Atlas}. It follows that if $\chi$ is any complex lift of $V$, then
$\chi_S = \eta_1 + \eta_2$. Since $\sigma$ fuses $\eta_1$ and $\eta_2$, we see that
$\chi  = \Ind^H_S(\eta_1)$. Thus $V$ has a unique complex lift, and so by Lemma
\ref{cyclic} and Frobenius reciprocity we have
$$0 = \Ext^1_H(V,V) = \Ext^1_H(\Ind^H_S(W_1),V) \cong
    \Ext^1_S(W_1,V_S) \cong \Ext^1_S(W_1,W_1) \oplus \Ext^1_S(W_1,W_2).$$
In particular, $\Ext^1_S(W_1,W_2) = 0$.

Next suppose that $n = 1$. Inspecting the character table of $\SL_2(q)$ as given in
\cite[Table 2]{DM}, we see that $S$ has a $\sigma$-invariant complex irreducible
character $\chi$ of degree $q-1$ such that the restriction of $\chi$ to $p'$-elements of
$S$ is the Brauer character of $V_S$. Since $H/S$ is cyclic and generated by $\sigma$,
it follows that $\chi$ extends to a complex irreducible character $\tilde\chi$ of $H$.
Now $\tilde\chi \neq \Ind^H_S(\eta_1)$ (since the latter is reducible over $S$), but both of them are complex lifts of $V$ (by Clifford's theorem). Applying Lemma \ref{cyclic} and
Frobenius reciprocity as above, we see that
$\dim \Ext^1_H(V,V) = \dim \Ext^1_S(W_1,W_2) = 1$.
\end{proof}

\begin{lem}\label{spor}
Let $H$ be a quasisimple group with $\bfZ(H)$ a $p'$-group. Let $W$ and $W'$ be absolutely irreducible $kH$-modules in characteristic $p$ of dimension $d$,
where $(H,p,d)$ is one of the following triples
$$\begin{array}{l}
    (2\AAA_7,5,4), ~(3J_3,19,18), (2Ru,29,28),~(6_1 \cdot \PSL_3(4),7,6),\\
    (6_1 \cdot \PSU_4(3),7,6),~
    (2J_2,7,6),~(3\AAA_7,7,6),~(6\AAA_7,7,6), ~(M_{11},11,10),\\
    (2M_{12},11,10),~(2M_{22},11,10),~(6Suz,13,12),~
    (2\gtwo(4),13,12),~ (3\AAA_6,5,3), \\
    (3\AAA_7,5,3), ~(M_{11},11,9),~(M_{23},23,21), ~
    (2\AAA_7,7,4), ~(J_1,11,7).\end{array}$$
If $\bfZ(H)$ acts the same way on $W$ and $W'$, assume in addition that
there is an automorphism of $H$ which sends $W$ to $W'$.
Then $\Ext^1_H(W,W') = 0$, with the following two exceptions
$(H,p,d) = (3\AAA_6,5,3)$ and $(2\AAA_7,7,4)$, where $\dim \Ext^1_H(W,W) = 1$.
\end{lem}

\begin{proof}
Note that the Sylow $p$-subgroups of $H$ have order $p$. Hence,
in the case $W \cong W'$ we can apply Lemma \ref{cyclic}; in particular,
we arrive at the two exceptions listed above. This argument settles
the cases of $(M_{11},11,9)$, $(M_{23},23,21)$, $(J_1,11,7)$,
$(2\gtwo(4),13,12)$.

If $W \not\cong W'$ and
$\bfZ(H)$ acts differently on $W$ and $W'$ then we also get $\Ext^1_H(W,W') = 0$
since $\bfZ(H)$ is a central $p'$-group.
So it remains to consider the case
where $W \not\cong W'$ and $\bfZ(H)$ acts the same way on them. Suppose in addition
that there is an involutory automorphism $\sigma$ of $H$ that swaps $W$ and
$W'$ and the module $\Ind^J_H(W)$ of $J := H:\langle \sigma \rangle$ has at most one
complex lift. Then we can apply Lemma \ref{cyclic} to $J$ as in the proof of Lemma \ref{weil-ext} to conclude that $\Ext^1(W,W') = 0$. These arguments are used to handle
the cases of $(2\AAA_7,5,4)$, $(3\AAA_7,5,3)$, $(3\AAA_7,7,6)$, $(2J_2,7,6)$,
$(6Suz,13,12)$, $(6_1 \cdot \PSL_3(4),7,6)$, and $(6_1 \cdot \PSU_4(3),7,6)$.

In the six remaining cases of $(6\AAA_7,7,6)$, $(3J_3,19,18)$,
$(2Ru,29,28)$, $(M_{11},11,10)$, $(2M_{12},11,10)$, and $(2M_{22},11,10)$
we note (using \cite{JLPW} or \cite{GAP})
that the non-isomorphic $H$-modules $W$ and $W'$ with the same action
of $\bfZ(H)$ are {\it not} $\Aut(H)$-conjugate.
\end{proof}

\begin{cor}\label{nonzero1}
Suppose that $q > 3$ is an odd prime power such that $p = (q+1)/2$ is a prime. Then
there is a finite absolutely irreducible linear group $G < \GL(V) = \GL_{q-1}(k)$ of degree
$q-1$ over $k$ such that $\GP \cong \SL_2(q)$, all irreducible $\GP$-submodules in $V$
are Weil modules of dimension $(q-1)/2$, and $\dim \Ext^1_G(V,V) = 1$. In particular,
$(G,V)$ is not adequate.
\end{cor}

\begin{proof}
Our conditions on $p,q$ imply that $q \equiv 1 (\mod 4)$. Now we can just appeal to
the proof of Lemma \ref{weil-ext}, taking $H = \GU_2(q)/C$, where $C$ is the unique
subgroup of order $(q+1)/2$ in $\bfZ (\GU_2(q))$.
\end{proof}

\begin{prop}\label{extra2}
Suppose $(G,V)$ is as in the extraspecial case (ii) of Theorem \ref{str}. Then
$\Ext^1_G(V,V) = 0$.
\end{prop}

\begin{proof}
Write $V|_\GP = e\sum^t_{i=1}W_i$ as usual and let $K_i$ be the kernel of the action
of $\GP$ on $W_i$. By Lemma \ref{zero-ext}, it suffices to show
that $\Ext^1_\GP(W_i,W_j) = 0$ for all $i,j$. Recall that $R := \bfO_{p'}(\GP)$ acts
irreducibly on $W_i$. By Theorem \ref{str}, $K_i$ has no composition factor
$\cong C_p$, whence $\Ext^1_\GP(W_i,W_i) = \Ext^1_{\GP/K_i}(W_i,W_i)$ by
Lemma \ref{kernel}(ii). Next, $\GP/K_i$ has cyclic Sylow $p$-subgroups (of order $p$)
by Theorem \ref{bz}(e), and we have shown in the proof of Proposition \ref{extra1} that
the $\GP/K_i$-module $W_i$ has a unique complex lift. Hence
$\Ext^1_{\GP/K_i}(W_i,W_i) = 0$ by Lemma \ref{cyclic}.

Suppose now that $i \neq j$ and let $M$ be any extension of the $\GP$-module
$W_i$ by the $\GP$-module $W_j$. Recall that the $R$-modules $W_i$ and $W_j$
are irreducible and non-isomorphic, as shown in the proof of Proposition \ref{extra1}.
But $R$ is a $p'$-group, so by Maschke's theorem $M = M_1 \oplus M_2$ with
$M_i \cong W_i$ as $R$-modules. Now for any $g \in \GP$,
$g(M_i) \cong (W_i)^g \cong W_i$ as $R$-modules, and so $g(M_i) = M_i$. Thus
$M = M_1 \oplus M_2$ as $\GP$-module. We have shown that
$\Ext^1_\GP(W_i,W_j) = 0$.
\end{proof}

\begin{prop}\label{simple2}
Suppose $(G,V)$ is as in case (i) of Theorem \ref{str}. Then
$\Ext^1_G(V,V) = 0$ unless one of the following possibilities occurs for the
group $H < \GL(W)$ induced by the action of $\GP$ on any irreducible
$\GP$-submodule $W$ of $V$.

{\rm (i)} $p = (q+1)/2$, $\dim W = p-1$, and $H \cong \SL_2(q)$.

{\rm (ii)} $p = 2^f +1$ is a Fermat prime, $\dim W = p-2$, and $H \cong \SL_2(2^f)$.

{\rm (iii)} $(H,p,d) = (3\AAA_6,5,3)$ and $(2\AAA_7,7,4)$.
\end{prop}

\begin{proof}
Write $V|_\GP = e\sum^t_{i=1}W_i$ as usual and let $K_i$ be the kernel of the action
of $\GP$ on $W_i$. By Lemma \ref{zero-ext}, it suffices to show
that $\Ext^1_\GP(W_i,W_j) = 0$ for all $i,j$. Note that neither $\GP$ nor $K_i$ can
have $C_p$ as a composition factor, according to Theorem \ref{str}. Furthermore,
if $K_i \neq K_j$ then we are done by Corollary \ref{product2}. So we may assume
that $K_i = K_j$ and then replace $\GP$ by $H = \GP/K_i = \GP/K_j$ by Lemma
\ref{kernel}. Now we will go over the possibilities for $(H,W_i)$ listed in Theorem
\ref{bz}(b)--(d).

\smallskip
Suppose we are in the case (b1) of Theorem \ref{bz}. Assume first that
$(p,H) = ((q^n+1)/2,\Sp_{2n}(q))$. It is well known (cf.\ \cite[Theorem 2.1]{GMST}) that
$H$ has exactly two irreducible modules of dimension $(q^n-1)/2$, namely the two Weil
modules of given dimension. Hence we can apply Lemma \ref{weil-ext} and arrive
at the exception (i).

Next, assume that $(p,H) = ((q^n+1)/(q+1),\PSU_n(q))$; in
particular, $n \geq 3$ is odd. Applying Theorem 2.7 and Proposition 11.3 of
\cite{GMST}, we see that there is a unique irreducible $kH$-module of dimension
$p-1 = (q^n-q)/(q+1)$ and, furthermore, this module has a unique complex lift. Hence
we are done by Lemma \ref{cyclic}.

\smallskip
Suppose now that we are in the case (c) of Theorem \ref{bz}. If $H = \AAA_p$, then
using \cite[Lemma 6.1]{GT2} for $p \geq 17$ and \cite{Atlas} for $p \leq 13$, we see that
$H$ has a unique irreducible $kH$-module of dimension $p-2$ and furthermore that
module has no complex lift unless $p=5$, whence we are done by Lemma \ref{cyclic}. Note
that the exception $p = 5$ is recorded in (ii) (with $f = 2$).

Next, assume that $(p,H) = ((q^n-1)/(q-1), \PSL_n(q))$. If $n=2$, then $p=q+1=2^f+1$ is a
Fermat prime, in which case $H = \SL_2(2^f)$ has a unique irreducible $kH$-module $W$
of dimension $p-2$, with $2^{f-1}$ complex lifts, whence $\dim \Ext^1_H(W,W) = 1$
by Lemma \ref{cyclic}. This exception is recorded in (ii). If $n \geq 3$, then by
\cite[Theorem 1.1]{GT1}, $H$ has a unique irreducible $kH$-module $W$
of dimension $p-2$ with no complex lifts, whence $\dim \Ext^1_H(W,W) = 0$
by Lemma \ref{cyclic}.

It remains to consider the $19$ cases listed in Lemma \ref{spor}. Furthermore,
by Corollary \ref{product2} we need only consider the case where $\GP$ acts on
$W_i$ and $W_j$ with the same kernel. Since $\GP$ has no composition factor
isomorphic to $C_p$, by Lemma \ref{kernel}(ii), we may view $W_i$ and $W_j$
as modules over the same quasisimple group $H$, with the same kernel. The
irreducibility of $G$ on $V$ further implies that $W_j \cong W_i^g$ for some
$g \in G$, whence the $H$-modules $W_i$ and $W_j$ are $\Aut(H)$-conjugate.
Now we are done by applying Lemma \ref{spor}.
\end{proof}

\begin{cor}\label{nonzero2}
Suppose that $p = 2^f+1$ is a Fermat prime. Then
there is a finite absolutely irreducible linear group $G < \GL(V) = \GL_{p-2}(k)$ of degree
$p-2$ over $k$ such that $G = \GP \cong \SL_2(2^f)$ and $\dim \Ext^1_G(V,V) = 1$. In particular, $(G,V)$ is not adequate.
\end{cor}

\begin{proof}
See the proof of Proposition \ref{simple2} and the exception (ii) listed therein.
\end{proof}

{\bf Proof of Theorem \ref{thm: ext=0}.}
(a) Assume first that $G$ is not $p$-solvable. Then $\GP$ has no
composition factor isomorphic to $C_p$ and
$H^1(G,k) = 0$ by Theorem \ref{str}. By Lemma \ref{zero-ext}, we need to
verify if $\Ext^1_\GP(W_i,W_j) = 0$ for any two simple $\GP$-submodules
$W_i$ and $W_j$ of $V$, of dimension $1 < d < p$. Suppose for instance that
$\Ext^1_\GP(W_1,W_2) \neq 0$.

Suppose in addition that $p > 3$. Then the perfect group $\GP$ admits a reducible
indecomposable module $U$ with two composition factors $W_1$ and $W_2$,
of dimension $2d$, say with kernel $K$. Since $\GP$ has no
composition factor isomorphic to $C_p$, $\bfO_p(X) = 1$ for the group
$X := \GP/K$ induced by the action of $\GP$  on $U$.
Suppose that $X$ is not quasisimple. By Proposition \ref{indec-str},
we have $d =p-1$. Then by Lemma \ref{2p-2}, either we arrive at the exception
(b)(ii) listed in Theorem \ref{thm: ext=0}, or else Lemma \ref{lem:opprime}
applies.  In the latter case, we see that $H^1(X,k) \neq 0$, whence $X$ and $\GP$ admit $C_p$ as a composition factor, a contradiction. Thus $X$ is quasisimple and
$\bfZ(X)$ is a $p'$-group. If $X$ is of Lie type in characteristic $p > 3$, then we must
have $d = (p \pm 1)/2$ and arrive (using Lemma \ref{ext-defi})
at the exception (b)(i). Otherwise we are in
the case (i) of Theorem \ref{str} and so by Proposition \ref{simple2} we arrive
at the exceptions (b)(iii)--(v).

\smallskip
(b) Now we consider the case $p = 3$ and $G$ is not $p$-solvable. Then
the perfect group $\GP$ acts nontrivially on $W_1$ and $W_2$ of dimension $2$.
Applying Theorem \ref{str}, we see that $\GP = L_1 * \ldots * L_n$ is a
central product of quasisimple groups; moreover, for all $j$ we have that
$L_j = \SL_2(q)$ with $q= 3^a > 3$ or $q = 5$. Also, for each $i$, there is
a unique $k_i$ such that $L_j$ acts nontrivially on $W_i$ precisely when $j = k_i$.
Since $\Ext^1_X(k,k) = 0$ for any perfect group $X$, by Lemma
\ref{lem:kunneth} we may assume that $k_1 = k_2 = 1$ and
$\Ext^1_{L_1}(W_1,W_2) \neq 0$. If $q=5$, then the case (b)(iii) holds. Otherwise
we arrive at (b)(vi) -- indeed, $\Ext^1_{L_1}(L(3^{a-2}),L(3^{a-1})) \neq 0$ by
\cite[Corollary 4.5]{AJL}.

\smallskip
(c) We may now assume that $\GP$ is $p$-solvable (and so is $G$). In particular, the subgroup
$H < \GL(W_i)$ induced by the action of $\GP$ on $W_i$ is $p$-solvable, whence
$p$ is a Fermat prime, and $H = \bfO_{p'}(H)P$ with $P \cong C_p$.
Since $\GP$ projects onto $H$, $\GP$ also has $C_p$ as a composition factor, and so
$H^1(\GP,k) \neq 0$; in particular,  $\Ext^1_\GP(V,V) \neq 0$.
We arrive at the exception (a) of Theorem \ref{thm: ext=0}.
\hfill $\Box$


\medskip
{\bf Proof of Corollary \ref{cor: adequate}.} Suppose that $(G,V)$ is {\it not} adequate, and let
$\overline{V} := V \otimes_k \overline{k}$.
By the assumptions, $\dim W < p$. Since $(\dim \overline{V})/(\dim W)$ divides $|G/\GP|$ by
\cite[Theorem 8.30]{N}, $p \nmid (\dim_{\overline{k}}\overline{V}) = \dim_k V$. Next, $(G,V)$ is weakly adequate by Theorem
\ref{thm: weak adequacy}. It follows that $\Ext^1_G(V,V) \neq 0$ and so
$\Ext^1_G(\overline{V},\overline{V}) \neq 0$. Now we can apply Theorem \ref{thm: ext=0}.
\hfill $\Box$

\end{document}